\documentclass[10pt, a4paper]{amsart}

\usepackage{amssymb,amsfonts}
\usepackage[all,arc]{xy}
\usepackage{enumerate}
\usepackage{mathrsfs}
\usepackage{tikz-cd}
\usetikzlibrary{decorations.pathmorphing}
\usepackage{enumitem}
\usepackage{empheq}
\usepackage{kotex}
\usepackage{ifpdf}
\usepackage{mathtools}
\usepackage{pst-node, pst-plot, pstricks}
\usepackage{mathrsfs}
\usepackage{amsmath}
\usepackage{epsfig}
\usepackage{amscd}
\usepackage{graphicx, color}
\graphicspath{ {./images/} }

\def\E{\ifmmode{\mathbb E}\else{$\mathbb E$}\fi} 
\def\N{\ifmmode{\mathbb N}\else{$\mathbb N$}\fi} 
\def\R{\ifmmode{\mathbb R}\else{$\mathbb R$}\fi} 
\def\Q{\ifmmode{\mathbb Q}\else{$\mathbb Q$}\fi} 
\def\C{\ifmmode{\mathbb C}\else{$\mathbb C$}\fi} 
\def\H{\ifmmode{\mathbb H}\else{$\mathbb H$}\fi} 
\def\Z{\ifmmode{\mathbb Z}\else{$\mathbb Z$}\fi} 
\def\P{\ifmmode{\mathbb P}\else{$\mathbb P$}\fi} 
\def\T{\ifmmode{\mathbb T}\else{$\mathbb T$}\fi} 
\def\SS{\ifmmode{\mathbb S}\else{$\mathbb S$}\fi} 
\def\DD{\ifmmode{\mathbb D}\else{$\mathbb D$}\fi} 

\DeclareSymbolFont{yhlargesymbols}{OMX}{yhex}{m}{n}

\DeclareMathAccent{\widetriangle}{\mathord}{yhlargesymbols}{"E6}

\def\E{\ifmmode{\mathbb E}\else{$\mathbb E$}\fi} 
\def\N{\ifmmode{\mathbb N}\else{$\mathbb N$}\fi} 
\def\R{\ifmmode{\mathbb R}\else{$\mathbb R$}\fi} 
\def\Q{\ifmmode{\mathbb Q}\else{$\mathbb Q$}\fi} 
\def\C{\ifmmode{\mathbb C}\else{$\mathbb C$}\fi} 
\def\H{\ifmmode{\mathbb H}\else{$\mathbb H$}\fi} 
\def\Z{\ifmmode{\mathbb Z}\else{$\mathbb Z$}\fi} 
\def\P{\ifmmode{\mathbb P}\else{$\mathbb P$}\fi} 
\def\T{\ifmmode{\mathbb T}\else{$\mathbb T$}\fi} 
\def\SS{\ifmmode{\mathbb S}\else{$\mathbb S$}\fi} 
\def\DD{\ifmmode{\mathbb D}\else{$\mathbb D$}\fi} 

\newcommand{\ben}{\begin{enumerate}}
\newcommand{\een}{\end{enumerate}}
\newcommand{\be}{\begin{equation}}
\newcommand{\ee}{\end{equation}}
\newcommand{\bea}{\begin{eqnarray}}
\newcommand{\eea}{\end{eqnarray}}
\newcommand{\beastar}{\begin{eqnarray*}}
\newcommand{\eeastar}{\end{eqnarray*}}
\newcommand{\bc}{\begin{center}}
\newcommand{\ec}{\end{center}}

\theoremstyle{theorem}
\newtheorem{thm}{Theorem}[section]
\newtheorem{cor}[thm]{Corollary}
\newtheorem{lem}[thm]{Lemma}
\newtheorem{prop}[thm]{Proposition}
\newtheorem{'thm'}[thm]{'Theorem'}

\theoremstyle{definition}
\newtheorem{defn}[thm]{Definition}
\newtheorem{rem}[thm]{Remark}

\newtheorem{exam}[thm]{Example}

\newtheorem{proof-sketch}[thm]{Proof-Sketch}

\newtheorem{claim}[thm]{Claim}
\newtheorem{choice}[thm]{Choice}
\newtheorem{lem-defn}[thm]{Lemma-Definition}
\newtheorem{prop-defn}[thm]{Proposition-Definition}
\newtheorem{thm-defn}[thm]{Theorem-Definition}

\newtheorem{notation}[thm]{\rm\bfseries{Notation}}

\newtheorem*{thm*}{Theorem}

\numberwithin{equation}{section}

\def\R{{\mathbb R}}

\def\E{{\mathbb E}}
\def\Z{{\mathbb Z}}
\def\C{{\mathbb C}}
\def\R{{\mathbb R}}
\def\P{{\mathbb P}}

\def\N{{\mathbb N}}

\def\11{{\mathbb I}}

\def\sgn{{\text{\rm sgn}}}

\def\C{\mathbb{C}}
\def\Z{\mathbb{Z}}

\def\T{\mathbb{T}}

\def\Q{\mathbb{Q}}

\def\E{\ifmmode{\mathbb E}\else{$\mathbb E$}\fi} 
\def\N{\ifmmode{\mathbb N}\else{$\mathbb N$}\fi} 
\def\R{\ifmmode{\mathbb R}\else{$\mathbb R$}\fi} 
\def\Q{\ifmmode{\mathbb Q}\else{$\mathbb Q$}\fi} 
\def\C{\ifmmode{\mathbb C}\else{$\mathbb C$}\fi} 
\def\H{\ifmmode{\mathbb H}\else{$\mathbb H$}\fi} 
\def\Z{\ifmmode{\mathbb Z}\else{$\mathbb Z$}\fi} 
\def\P{\ifmmode{\mathbb P}\else{$\mathbb P$}\fi} 
\def\SS{\ifmmode{\mathbb S}\else{$\mathbb S$}\fi} 
\def\DD{\ifmmode{\mathbb D}\else{$\mathbb D$}\fi} 

\def\R{{\mathbb R}}

\def\E{{\mathbb E}}
\def\Z{{\mathbb Z}}
\def\C{{\mathbb C}}
\def\R{{\mathbb R}}

\def\N{{\mathbb N}}

\def\darr#1{\raise1.5ex\hbox{$\leftrightarrow$}
\mkern-16.5mu #1}

\def\roughly#1{\raise.3ex\hbox{$#1$\kern-.75em
\lower1ex\hbox{$\sim$}}}

\def\opname#1{\mathop{\kern0pt{\rm #1}}\nolimits}

\def\dim{\opname{dim}}

\def\Incl{\operatorname{Incl}}
\def\Eval{\operatorname{Eval}}

\begin{document}

\quad \vskip1.375truein

\bibliographystyle{plain}

\title[Categorical structures of Kuranishi spaces]
{Categorical structures of Kuranishi spaces with $L_{\infty}[1]$-algebras}

\author{Taesu Kim}
\email{kimtaesu@postech.ac.kr}


\begin{abstract}
We introduce $L_{\infty}$-Kuranishi spaces by associating an $L_{\infty}[1]$-algebra to each chart, defined on open neighborhoods of points in the zero locus of the Kuranishi section. We show that these spaces form a category that admits a natural embedding from the category of smooth manifolds. To achieve this categorical structure, we modify certain notions from \cite{FOOO1}. For instance, we replace the tangent bundle condition for chart embeddings with a quasi-isomorphism condition on the $L_{\infty}[1]$-structures.
\end{abstract}

\thanks{This work was carried out while the author was affiliated with the Department of Mathematics at Pohang University of Science and Technology (POSTECH) and was supported by the BK21 FOUR program funded by the Ministry of Education (MOE), Korea (No. 4120240414885)}

\maketitle

\date{}

\tableofcontents

\section{Introduction}

Let $E \rightarrow U$ be a vector bundle over a smooth manifold $U$ with a distinguished smooth section $s$, and suppose that a finite group $\Gamma$ acts on this data. For a compact topological space $X$, we consider a Kuranishi chart 
\[
\mathcal{U} = (U, E, s, \Gamma, \psi)
\]
on $X$. This means that each point in $X$ admits a neighborhood homeomorphic to the quotient $s^{-1}(0)/\Gamma$ of the zero locus, as depicted in the following diagram:
\[
\begin{tikzcd}
{} & {} & E \arrow{d}{\pi}\\
X & \arrow[swap, hook]{l}{\psi \circ (\cdot)/\Gamma,\simeq} s^{-1}(0) \arrow[hook]{r} & U. \arrow[""{inner sep=0}, "\Gamma", from=2-3, to=2-3, out=15, in=-15, loop, distance=2em] \arrow[bend left,looseness=1.2]{u}{s}
\end{tikzcd}
\]

Fukaya and Ono \cite{FO} introduced this notion to study the moduli space of pseudoholomorphic maps in symplectic geometry. Their approach systematically perturbs the section $s$, in a manner compatible with the $\Gamma$-action, to a transverse one, ensuring that $s^{-1}(0)$ becomes a manifold. By patching together a homology class of $s^{-1}(0)$, they constructed a global object known as the virtual fundamental class of $X$, which now serves as a foundation for Floer theory, Gromov-Witten theory, and related fields.

This approach has been further developed by works of \cite{FOOO1}, \cite{FOOO2}, \cite{Joyce1}, \cite{Joyce2}, \cite{MW1}, \cite{MW2}, and \cite{Pardon}, addressing foundational questions and endowing the theory with better-behaved structures. Nonetheless, it has revealed certain nontrivial issues as well; for instance, as remarked in \cite{FO}, their method exhibits an undesirable dependence on specific choices of the bundle $E$ and the ambient space $U$ of the zero locus of $s$, both of which should ideally be irrelevant.

To illustrate the issue more clearly, we begin with a simple example. Let $V$ be a finite-dimensional vector space. We consider the chart given by $\mathcal{U} \times V := (U \times V, E \times V, s \times \mathrm{id}_V, \Gamma, \psi).$ That is, the bundle is now given by $E \times V \xrightarrow{\pi \times \mathrm{id}_V} U \times V$, and the section is $s \times \mathrm{id}_V.$ The following immediate observation,
\[
(s \times \mathrm{id}_V)^{-1}(0) \simeq s^{-1}(0),
\]
suggests that there should be a way to identify these two objects. Rather surprisingly, no such identification exists yet; the notion of isomorphism of Kuranishi charts in \cite{FOOO1}, defined by Fukaya-Oh-Ohta-Ono (abbreviated as FOOO), is too restrictive to apply to this case.

The primary objective of this paper is to modify the theory of \cite{FO} and \cite{FOOO1} in order to minimize such dependencies, or, more precisely, to reformulate it within a categorical framework where various choices can be handled coherently, thereby rendering the ambient data trivial from the perspective of homotopy theory. To achieve this, we incorporate locally defined $L_{\infty}[1]$-algebras as the input for Kuranishi charts, drawing inspiration from the notion of $L_{\infty}$-spaces discussed in \cite{AKSZ}, \cite{AT}, \cite{BLX}, \cite{Costello}, and \cite{Tu1}.

Indeed, the incomparability of $\mathcal{U}$ and $\mathcal{U} \times V$ in the above example, at both the chart and space levels, is due to the lack of an appropriate notion of morphisms. Our second goal is, therefore, to define morphisms of charts coherently up to those of Kuranishi spaces. Consequently, we establish a category that naturally generalizes the category of smooth manifolds. Within this framework, \cite{Kim1} shows that the moduli space of pseudoholomorphic disks forms a well-defined object, assuming a condition on the existence of certain stratification structures on the base manifolds of the charts. Furthermore, this category is expected to exhibit additional favorable properties from a homotopy-theoretic viewpoint, which is discussed in detail in \cite{Kim4}.

Throughout this construction, we exploit the flexibility of the aforementioned structures. By flexibility, we precisely mean the use of commutativity up to homotopy (instead of strict commutativity), which is typically satisfied with considerably less effort. The definition of $L_{\infty}$​-Kuranishi chart embedding is given partly in terms of quasi-isomorphisms; accordingly, coordinate changes, as instances of such embeddings, provide better opportunities to achieve our goal. In this process, the homotopy invertibility guaranteed by the Whitehead theorem plays a crucial role.

\subsection{$L_{\infty}$-Kuranishi charts}

Let $X$ be a compact metrizable topological space. An $L_{\infty}$\textit{-Kuranishi chart of} $X$ is defined by a tuple
\[
\mathcal{U} = (U, E, \Gamma, s, \psi),
\]
where $U = (U, \beta)$ is a smooth manifold with a closed two-form $\beta \in \Omega^2(U)$ that satisfies a condition on the existence of a stratification that we introduce below. $E \rightarrow U$ is a vector bundle with a distinguished smooth section $s.$ Let $\Gamma$ be a finite group acting on $U$ that restricts to $ s^{-1}(0),$ and $\psi : s^{-1}(0)/\Gamma \hookrightarrow X$ a homeomorphism onto the image.

To each zero point $x \in s^{-1}(0)$, we assign \textit{the presymplectic neighborhood} $W_x$ and \textit{the local $L_{\infty}[1]$-algebra} $\mathcal{C}_x,$ which we shall soon introduce. To do so, we need some preparations. We first require that the closed two-form $\beta$ on $U$ allow the stratification 
\[U = \bigcup\limits_i \mathcal{S}_i,
\]
into \textit{submanifolds} $\mathcal{S}_i := \{y \in U \mid \mathrm{rk}\left(\ker\beta_y\right) = i\}$ and the system of their tubular neighborhoods.

\begin{rem}
\cite[Corollary 6.2]{KO} proves that there exists a residual subset of such closed two-forms. Furthermore, the stratification structure turns out to be Whitney (cf. \cite[Definition 6.5 \& Theorem 6.6]{KO}), which leads to the construction of \textit{stratified $L_{\infty}$-spaces.}
\end{rem}

Given a closed two-form $\beta \in \Omega^2(U)$ of the above type, we consider a contractible open ball $W_x \subset U$ near each zero point $x \in s^{-1}(0)$ and endow it with a presymplectic structure  $\beta_{W_x}$, and call $(W_x, \beta_{W_x})$ \textit{the local presymplectic neighborhood} of $x \in s^{-1}(0).$ (See Subsection 2.1 for its detailed construction.) Then the kernel of $\beta_{W_x}$ determines a regular foliation (i.e., each leaf having the same dimension) denoted by $T\mathcal{F}_x \subset TW_x.$

To each zero point $x \in s^{-1}(0)$, we also assign \textit{the local $L_{\infty}[1]$-algebra}
\[
\nonumber \mathcal{C}_x := \overbrace{\bigwedge\nolimits^{-\bullet}\Gamma(E^*|_{W_x})}^{\text{Koszul}} \oplus \overbrace{\Omega^{\bullet + 1}_{\mathrm{aug}}(\mathcal{F}_x)}^{\text{de Rham}}.
\]
Here $\bigwedge{}^{-\bullet}\Gamma(E^*|_{W_{x}})$ is given by the Koszul complex of the $C^{\infty}(W_{x})$-module $\Gamma(E|_{W_{x}})$ (considered as an $\mathbb{R}$-module) with vanishing higher-degree operations, while $\Omega^{\bullet+1}_{\mathrm{aug}}(\mathcal{F}_{x})$ is the foliation de Rham complex with augmentation, obtained from the regular foliation $T\mathcal{F}_{x} \subset TW_{x}.$ In fact, we can equip the foliation de Rham complex $\Omega^{\bullet+1}(\mathcal{F}_{x})$ with an $L_{\infty}[1]$-algebra structure as in Example \ref{gtyemb} and consider its extension to the augmentation $\Omega^{\bullet+1}_{\mathrm{aug}}(\mathcal{F}_{x})$ via a recursive argument (cf. Proposition \ref{augomega}).

\subsection{Morphisms of charts and relation to the FOOO's works}

Let $f : X \rightarrow X'$ be a continuous map between compact topological spaces. A \textit{chart morphism} between $L_{\infty}$-Kuranishi charts ${\mathcal{U}} = (U,E,\Gamma,s,\psi)$ and ${\mathcal{U}}' = (U',E',\Gamma',s',\psi')$ of $X$ and $X',$ respectively, is defined by a pair
\[
\Phi = \left(\phi, \widehat{\phi}\right) : {\mathcal{U}} \rightarrow {\mathcal{U}}' ,
\]
where each component is given by:
\begin{enumerate}
\item[--] $\phi : U \rightarrow U', \text{ a } (\Gamma, \Gamma')\text{-equivariant map of manifolds,}$
\item[--] $\widehat{\phi} = \left\{\widehat{\phi}_x :  \mathcal{C}'_{\phi(x)} \rightarrow \mathcal{C}_x \right\}_{x \in s^{-1}(0)}$ is a family of $L_{\infty}[1]$-morphisms,
\end{enumerate}
satisfying 
\begin{enumerate}[label = (\roman*)]
\item $\psi' \circ \phi = f \circ \psi$ on $s^{-1}(0),$
\item $\phi(W_x) \subset W'_{\phi(W_x)},$
\item $\widehat{\phi}_x$ factors through $\mathcal{C}'_{\phi(x), \phi},$ that is, we have $\widehat{\phi}_x = \widehat{\phi}_x^{\mathrm{c}} \circ \widehat{\varepsilon}_{\phi(x), \phi}$ for some $L_{\infty}[1]$-morphism, $\widehat{\phi}^{\mathrm{c}}_x :  \mathcal{C}'_{\phi(x), \phi} \rightarrow \mathcal{C}_x.$
\end{enumerate}
Here $\mathcal{C}_{\phi(x),\phi}'$ stands for \textit{the completion} of $\mathcal{C}'_{\phi(x)}$ at the image of $\phi,$ induced from the \text{completed} V-algebra at the image. (See Definition \ref{ladef} for the precise definition.) On the other hand, $\widehat{\varepsilon}_{\phi(x), \phi} : \mathcal{C}'_{\phi(x)} \rightarrow \mathcal{C}'_{\phi(x), \phi}$ is the $L_{\infty}[1]$-morphism given by considering elements in $\mathcal{C}'_{\phi(x)}$ naturally as those in $\mathcal{C}'_{\phi(x), \phi}$, which can be easily shown to define an $L_{\infty}[1]$-morphism. (See Definition \ref{varep} and Lemma \ref{vecpf}.)

We recall that an FOOO embedding is given by a pair $\left(\phi, \widetilde{\phi}\right),$ where $\phi: U \hookrightarrow U'$ is an embedding and $\widetilde{\phi} :  E \hookrightarrow E'$ is a bundle embedding, satisfying natural compatibility conditions with respect to the sections, together with tangent bundle condition: $ds'$ induces an isomorphism
\[
\left[ds'_{\phi(x)}\right] : \frac{T_{\phi(x)}U'}{\phi_*(T_xU)} \xrightarrow{\simeq} \frac{E'_{\phi(x)}}{\widetilde{\phi}(E_{x})},
\]
at each $x \in s^{-1}(0).$

We propose a new definition of chart embeddings: A chart morphism is called an \textit{embedding} if 
\begin{enumerate}
\item[--] $\phi$ is an (equivariant) embedding of manifolds,
\item[--] $\widehat{\phi}^{\mathrm{c}}_{x} : \mathcal{C}'_{\phi(x), \phi} \rightarrow \mathcal{C}_x$ is a quasi-isomorphism of $L_{\infty}[1]$-algebras for each $x.$
\end{enumerate}
We claim that this definition of embedding can be regarded as a generalization of the FOOO chart embedding (see Theorem~\ref{afec} for the precise statement):
\begin{thm}
An FOOO embedding of Kuranishi charts determines an embedding of $L_{\infty}$-Kuranishi charts. In particular, the tangent bundle condition yields the quasi-isomorphism condition between the $L_{\infty}[1]$-algebras.
\end{thm} 

A \textit{coordinate change} of $L_{\infty}$-Kuranishi charts provides a main example of embedding: For two points $p, q \in X$ with $\mathrm{Im}\psi_p \cap \mathrm{Im}\psi_q \neq \emptyset,$ we define their \textit{coordinate change} $\Phi_{pq} : \mathcal{U}_p \rightarrow \mathcal{U}_q$ by a tuple
\[
\Phi_{pq} := \left(U_{pq}, \phi_{pq}, \left\{\widehat{\phi}_{pq,x}\right\}\right),
\]
where $U_{pq} \subset U_p$ is an open submanifold, and
\[
\left(\phi_{pq}, \widehat{\phi}_{pq}\right) : \mathcal{U}_p|_{U_{pq}} \rightarrow \mathcal{U}_q
\]
is an embedding of $L_{\infty}$-Kuranishi charts from $\mathcal{U}_p|_{U_{pq}},$ that is, the chart restricted to $U_{pq}.$ They are required to satisfy:

\begin{enumerate}[label = (\roman*)]
\item $\Phi_{pp} = \mathrm{id}_{\mathcal{U}_p},$
\item $\psi_q \circ \phi_{pq} = \psi_p$ on $s_{p}^{-1}(0) \cap U_{pq},$
\item $\phi_{qr} \circ \phi_{pq} = \phi_{pr}$ on $s_p^{-1}(0) \cap \phi_{qr}^{-1}(U_{pq}) \cap U_{pr},$
\item $\psi_p\big(s_p^{-1}(0) \cap U_{pq}\big) =\mathrm{Im} \psi_p \cap\mathrm{Im} \psi_q,$
\end{enumerate}

A pair of the compact topological space $X$ and a collection of Kuranishi charts with coordinate changes
\[
\left(X, \widehat{\mathcal{U}}\right),
\]
where $ \widehat{\mathcal{U}}:=\left(\left\{\widehat{\mathcal{U}}_p\right\}, \left\{\Phi_{pq}\right\}\right),$ is called an $L_{\infty}$\textit{-Kuranishi atlas}. For technical reasons, we assume that $\max\limits_{p \in X} \dim U_p < \infty$ together with the compactness of $X$.

\begin{rem}
Contrary to the existing definition of Kuranishi spaces, the above cocycle condition does not include the bundle component. This reflects the flexibility of our construction, where we can take advantage of higher homotopical notions for $L_{\infty}[1]$-algebras. By virtue of the conceptual bypass provided by higher homotopy theory, the corresponding higher cocycle condition can be formulated and ensured to hold by making a choice of homotopy data. A more detailed discussion is given in \cite{Kim3}.
\end{rem}

\subsection{Definition of $L_{\infty}$-Kuranishi spaces and categorical structures}

Two atlases are said to be equivalent, denoted by $\widehat{\mathcal{U}} \sim \widehat{\mathcal{U}}'$ if we have an \textit{equality between the expansions of  atlases,}
\[
\widehat{\mathcal{U}}^0_1 \times \mathbb{R}^{m_1} = \widehat{\mathcal{U}}^0_2 \times \mathbb{R}^{m_2},
\]
for some $m_1, m_2 \geq 0$ and for the restrictions to some open subsets $\widehat{\mathcal{U}}^0 = \widehat{\mathcal{U}}|_{U^0 \subset U}$ and $\widehat{\mathcal{U}}^{'0} = \widehat{\mathcal{U}}'|_{U^{'0} \subset U'},$ whose precise meaning is given in Definition \ref{eqats}. By (iii) of Lemma \ref{gggadbbbb}, the relation $\sim$ defines an equivalence relation. 

Then we define an $L_{\infty}$-\textit{Kuranishi space} to be an equivalence class with respect to $\sim$
\[
\mathfrak{X} := \left(X, \left[\widehat{\mathcal{U}}\right]\right).
\]

Let $\mathfrak{X} = \left(X, \left[\widehat{\underline{\mathcal{U}}}\right]\right)$ and $\mathfrak{Y} = \left(Y,\left[\underline{\widehat{\mathcal{U}}}'\right]\right)$ be $L_{\infty}$-Kuranishi spaces.
We consider two atlases $\left(X, \widehat{\mathcal{U}}\right)$ and $\left(X', \widehat{\mathcal{U}}'\right)$ such that $\widehat{\mathcal{U}} \sim \underline{\widehat{\mathcal{U}}}$ and $\widehat{\mathcal{U}}' \sim \underline{\widehat{\mathcal{U}}}'$ with $\widehat{\mathcal{U}} = \left(\{\widehat{\mathcal{U}}_p\}, \left\{\Phi_{pq}\right\}\right) $ and $\widehat{\mathcal{U}}' = \left(\{\widehat{\mathcal{U}}'_{p'}\}, \left\{\Phi'_{p'q'}\right\}\right).$

A \textit{pre-morphism} is defined by the following tuple
\[
\overline{F} := \left(\widehat{\mathcal{U}}, \widehat{\mathcal{U}}', f, \{f_p\}, \left\{\widehat{f}_{p,x}\right\}\right).
\]
Here $f : X \rightarrow Y$ is a continuous map between the zero loci, while $\left\{\left(f_p, \left\{\widehat{f}_{p,x}\right\}\right)\right\}$ is a collection of chart morphisms. Then $\overline{F}$ is required to satisfy the following compatibilities with respect to the coordinate changes $\Phi_{pq}=\left(\phi_{pq}, \left\{\widehat{\phi}_{pq,x}\right\}\right)$:
For $p,q \in X$ with Im$\psi_p \cap\mathrm{Im} \psi_q \neq \emptyset,$ we require
\begin{enumerate}[label = (\roman*)]
\item $\phi'_{f(p)f(q)} \circ f_p = f_q \circ \phi_{pq}$ on $s_{p}^{-1}(0) \cap U_{pq},$
\item $\widehat{\phi}_{pq,x} \circ \widehat{f}_{q, \phi_{pq}(x)} = \widehat{f}_{p,x} \circ \widehat{\phi}'_{f(p)f(q),f_p(x)}$ for each $x \in s_p^{-1}(0) \cap U_{pq}$ up to $L_{\infty}[1]$-homotopy.
\end{enumerate}

We can define an equivalence relation on the set of pre-morphisms as in Definition \ref{defmoreq} and Lemma \ref{hgjgkg}. We define a \textit{morphism} from $\mathfrak{X} = \left(X, \left[\widehat{\mathcal{U}}\right]\right)$ to $\mathfrak{X}' = \left(X', \left[\widehat{\mathcal{U}}'\right]\right)$
by an equivalence class of a pre-morphism $\overline{F}$ from $\mathfrak{X}$ to $\mathfrak{X}':$
\[F := \left[\overline{F}\right] : \mathfrak{X} \rightarrow \mathfrak{X}'.\]

Let $\mathfrak{X} = \left(X, \left[\widehat{\mathcal{U}}\right]\right),$ $\mathfrak{X}' = \left(X', \left[\widehat{\mathcal{U}'}\right]\right),$ and $\mathfrak{X}'' = \left(X'', \left[\widehat{\mathcal{U}''}\right]\right)$ be Kuranishi spaces. Let $F : \mathfrak{X} \rightarrow \mathfrak{X}'$ and $G : \mathfrak{X}' \rightarrow \mathfrak{X}''$ be morphisms between them, represented by pre-morphisms
\begin{equation}\nonumber
\begin{cases}
\overline{F} = \left(\widehat{\mathcal{U}}, \widehat{\mathcal{U}'}, f, \left\{f_{p}\right\}, \left\{\widehat{f}_{p,x}\right\}\right),\\
\overline{G} = \left(\underline{\widehat{\mathcal{U}'}}, \widehat{\mathcal{U}''}, g, \left\{g_{f(p)}\right\}, \left\{\widehat{g}_{f(p),y}\right\}\right),
\end{cases}
\end{equation}
respectively with $[ \widehat{\mathcal{U}'}] = [\underline{\widehat{\mathcal{U}'}}]$.

In fact, we may assume that ${\widehat{\mathcal{U}'}} = \underline{\widehat{\mathcal{U}'}}$ and that $f_{p}$ is surjecive for each $p$; if not, we take the \textit{extension} of the pre-morphisms having equivalent charts as components (cf. (\ref{extfi}) and Definition \ref{afjfjggg}). We then define the \textit{composition} $G \circ F$ to be the following equivalence class:
\begin{equation}\label{morcom}
G \circ F := \left[\left(\widehat{\mathcal{U}}, \widehat{\mathcal{U}}'', g\circ f, \left\{{g}_{f(p)} \circ {f}_{p}\right\}, \left\{{\widehat{f}}_{p,x} \circ {\widehat{g}}_{f(p),f_p(x)}\right\} \right)\right].
\end{equation}

Then it turns out that the composition is well-defined and associative with the identity given by
\begin{equation}\label{idxm}
\mathrm{id}_{\mathfrak{X}}:= \left[\left(\widehat{\mathcal{U}}, \widehat{\mathcal{U}}, \mathrm{id}_X, \left\{\mathrm{id}_p\right\}, \left\{\widehat{\mathrm{id}}_{p,x}\right\}\right)\right]
\end{equation}
of each $\mathfrak{X} = \left(X, \left[\widehat{\mathcal{U}}\right]\right).$ We then obtain a category denoted by $\mathbf{Kur}$ that consists of:
\begin{equation}\nonumber
\begin{cases}
\text{Ob}(\textbf{Kur}) =\{L_{\infty}\text{-Kuranishi spaces}\}\\
\text{Mor}(\textbf{Kur}) = \{\text{Equivalence classes of pre-morphisms} \}.
\end{cases}
\end{equation}
Indeed, $\mathbf{Kur}$ contains $\mathbf{Man}$, the category of smooth manifolds as a subcategory, allowing us to treat Kuranishi spaces and smooth manifolds on equal footing.
\begin{thm}
$L_{\infty}$-Kuranishi spaces form a category that naturally admits an embedding from the category of smooth manifolds.
\end{thm}

\subsection{Outline of the paper}

The organization of this paper is as follows. In Section 2, we introduce the notion of $L_\infty$-Kuranishi charts. Section 3 defines $L_{\infty}$-Kuranishi spaces and establishes that they form a category with an appropriate notion of morphisms.  In Appendix A, we recall the theory of $L_{\infty}[1]$-homotopies and their homotopies developed in \cite{Kim2}. Finally, Appendix B studies V-algebras, providing the technical groundwork for the local $L_{\infty}[1]$-structures that appear in the earlier sections.

\subsection*{Acknowledgment} We are grateful to Sam Bardwell-Evans, Kenji Fukaya, Eunjung Jung, Adeel Khan, Young-Hoon Kiem, Jeongseok Oh, Yong-Geun Oh, Hiroshi Ohta, Kaoru Ono, and Hyeonjun Park for stimulating and fruitful discussions. This work was carried out while the author was affiliated with the Department of Mathematics at Pohang University of Science and Technology (POSTECH) and was supported by the BK21 FOUR program funded by the Ministry of Education (MOE), Korea (No. 4120240414885).

\section{$L_\infty$-Kuranishi charts}

In this section, we propose a new chart-level theory for Kuranishi spaces. We equip the base of each chart with a closed two-form, which induces a stratification structure and a presymplectic neighborhood around each zero point. This additional structure allows us to construct local $L_{\infty}[1]$-algebras and morphisms between two charts. Our notion of chart embedding generalizes the FOOO embedding, in the sense that their tangent bundle condition is reformulated as the quasi-isomorphism condition for the corresponding $L_{\infty}$-component. The prerequisite for this section is the content of Appendix B on $L_{\infty}[1]$-structures arising from V-algebras. We refer the reader to Appendix A for a brief introduction to $L_{\infty}[1]$-algebras and their homotopy theory.

\subsection{Definition of $L_\infty$-Kuranishi charts}

We present our definition of Kuranishi charts, emphasizing the key differences from existing approaches.

\begin{defn}[$L_\infty$-Kuranishi charts]\label{kurdef}
Let $X$ be a compact metrizable space. An $L_\infty$-\textit{Kuranishi chart} of $X$ is given by a tuple
\begin{equation}\nonumber
\mathcal{U} = (U, E, s, \Gamma, \psi),
\end{equation}
where
\begin{enumerate}
\item[--] $U = (U, \beta)$ consists of a smooth manifold together with a closed two-form $\beta \in Z^2(U).$
\item[--] $\pi : E \rightarrow U$ is a (finite rank) vector bundle.
\item[--] $s : U \rightarrow E$ is a smooth section.
\item[--] $\Gamma$ is a finite group acting on $U$ whose action restricts to the zero set of $s,$ i.e., $\Gamma \cdot s^{-1}(0) \subset s^{-1}(0).$
\item[--] $\psi : {s^{-1}(0)}/{\Gamma} \overset{\simeq}{\hookrightarrow} X$ is a homeomorphism onto its image.
\end{enumerate}
The \textit{dimension} of $\mathcal{U}$ is defined by $\dim \mathcal{U} := \dim U - \text{rk}\,E.$

We require that the chart $\mathcal{U}$ be endowed with the following structures:
\begin{itemize}
\item[--] $U$ is stratified by a decomposition into (not necessarily connected) submanifolds,
\begin{equation}\label{wstri}
U = \bigcup\limits_i \mathcal{S}_i,
\end{equation}
where each $\mathcal{S}_i$ is given by
\[
\mathcal{S}_i := \{x \in U \mid \text{rk} (\ker\beta_x) =i \}, \ 0 \leq i \leq \dim U.
\]
Here, each $\mathcal{S}_i$ is a submanifold equipped with a tubular neighborhood $N_i \supset \mathcal{S}_i$ and an associated projection $\pi_{i} : N_i \rightarrow \mathcal{S}_i.$

\item[--] To each $x \in s^{-1}(0),$ we assign:
\begin{enumerate}[label = (\roman*)]
\item A presymplectic open neighborhood $W_x$ of $x$ in $U$ with $W_x \simeq B^n,$
\item A local $L_{\infty}[1]$-algebra $\mathcal{C}_x,$
\end{enumerate}
which are described in detail below.
\end{itemize}

\noindent
(i) (\textit{The presymplectic open neighborhood $W_x$}) We associate an open contractible submanifold $W_x$ to each zero point $x \in s^{-1}(0).$ For a zero point $x \in s^{-1}(0) \cap \mathcal{S}_i$ for some $i$, we choose an open ball $\overset{\circ}{W}_x \subset \mathcal{S}_i$ containing $x.$ Let $\pi_i : N_i \rightarrow \mathcal{S}_i$ be the projection from the system of tubular neighborhoods described in Remark \ref{wsassm} below, restricting to $\pi_i : W_x \twoheadrightarrow \overset{\circ}{W}_x.$ For the inclusion ${\iota_x} : \overset{\circ}{W}_x {\hookrightarrow} U,$ we set $\beta|_{\overset{\circ}{W}_x} := \iota_x^*\beta.$ Note that
\[
d\beta|_{\overset{\circ}{W}_x} = d\iota_x^*\beta = \iota_x^* d\beta = 0,
\]
and that $\beta|_{\overset{\circ}{W}_x}$ has constant rank by construction. In other words, $\left(\overset{\circ}{W}_x, \beta|_{\overset{\circ}{W}_x}\right)$ is a presymplectic manifold. We then obtain another presymplectic manifold
\[
W_x = \left(W_x, \beta_{W_x}\right) := \left(\pi_i^{-1}\left(\overset{\circ}{W_x}\right), \pi^*_i\left(\beta|_{\overset{\circ}{W_x}}\right)\right),
\]
which we call the \textit{local presymplectic neighborhood of} $x \in s^{-1}(0).$ We write
\[
T\mathcal{F}_x := \ker \beta_{W_x}
\]
for the resulting regular foliation (i.e., one whose leaves all have the same dimension), and $T^*\mathcal{F}_x$ for its dual.

\begin{rem}\label{wsassm}
By \cite[Corollary 6.2]{KO}, there exists a residual subset of closed two-forms for which the decomposition (\ref{wstri}) holds. Moreover, the resulting stratification (\ref{wstri}) is a Whitney stratification (cf. \cite[Definition 6.5 \& Theorem 6.6]{KO}), which leads to the construction of \textit{stratified $L_{\infty}$-spaces.}
\end{rem}

\noindent
(ii) (\textit{The $L_{\infty}[1]$-algebra $\mathcal{C}_x$})
To each zero point $x \in s^{-1}(0),$ we associate a \textit{local $L_{\infty}[1]$-algebra}
\begin{equation}\nonumber 
\mathcal{C}_x := \overbrace{\bigwedge\nolimits^{-\bullet}\Gamma(E^*|_{W_x})}^{\text{Koszul}} \oplus \overbrace{\Omega^{\bullet + 1}_{\mathrm{aug}}(\mathcal{F}_x)}^{\text{de Rham}},
\end{equation}
consisting of two parts: a Koszul part and a de Rham part.

The Koszul part, $\bigwedge\nolimits^{-\bullet}\Gamma(E^*|_{W_x})$, is the Koszul complex
\[
0 \rightarrow \overbrace{\bigwedge\nolimits^{r}\Gamma(E^*|_{W_x})}^{\text{deg} = -r} \xrightarrow{\iota_{s|_{W_x}}} \cdots \xrightarrow{\iota_{s|_{W_x}}}  \overbrace{\Gamma(E^*|_{W_x})}^{\text{deg} = -1} \xrightarrow{\iota_{s|_{W_x}}}  \overbrace{C^{\infty}(W_x)}^{\text{deg} = 0} \rightarrow 0,
\]
with differential $l_1^{\mathrm{K}} := \iota_{s|_{W_x}}$ given by
\[
l_1^{\mathrm{K}} : a_1 \wedge \cdots \wedge a_m \mapsto \sum\limits_{i=1}^{m} (-1)^{i+1}a_i(s|_{W_x}) \cdot a_1 \wedge \cdots \wedge \widehat{a_i} \wedge \cdots \wedge a_m,
\]
and with all higher operations $l^{\mathrm{K}}_{k}$, $k \geq 2,$ set to zero.

The de Rham part, $\Omega^{\bullet + 1}_{\mathrm{aug}}(\mathcal{F}_x)$, is the augmented foliation de Rham complex shifted in degree by one, equipped with the $L_{\infty}[1]$-algebra structure $\{l^{\mathrm{dR}}_k\}_{k \geq 1}$ obtained from Definition \ref{obpoaug} and Proposition \ref{augomega}.

The $L_{\infty}[1]$-structure on $\mathcal{C}_x$ is then given by
\[
l_k : \mathcal{C}_x^{\otimes k} \rightarrow \mathcal{C}_x; \ l_k := l^{\mathrm{K}}_k \oplus l^{\mathrm{dR}}_k,
\]
where the direct sum notation indicates that the two components' operations are defined independently of one another. It is immediate that the family $\{l_k\}_{k \geq 1}$ satisfies the $L_{\infty}$-relations.
\end{defn}
\begin{lem}\label{lemdw}
For different choices of $\overset{\circ}{W}_x$, the resulting de Rham part $L_\infty[1]$-algebras are isomorphic.
\end{lem}

\begin{proof}
Given another choice $\overset{\circ}{W^{\prime}}_x$ in $\mathcal{S}_i,$ we may connect it to $\overset{\circ}{W'}_x$ using their contractibility. That is, there exists a smooth map
\[
A : [0,1] \times B^n \longrightarrow \mathcal{S}_i,
\]
satisfying
\[
A(0, B^n) = \overset{\circ}{W'}_x; \ A(1, B^n) = \overset{\circ}{W}_x.
\]
This yields a 1-parameter family of presymplectic forms of constant rank,
\[
\pi_x^* \left(\beta \big|_{A(t,B^n)}\right),
\]
interpolating between $\pi_x^*\left(\beta|_{\overset{\circ}{W'}_x}\right)$ and $\pi_x^*\left(\beta|_{\overset{\circ}{W}_x}\right).$ By Corollary \ref{indliso}, we obtain an $L_\infty[1]$-isomorphism
\[
\Omega^{\bullet+1} \left(\mathcal{F}_x \left( \pi_x^*(\beta|_{\overset{\circ}{W}_x'})\right)\right) \xrightarrow{\cong} \Omega^{\bullet+1} \left(\mathcal{F}_x \left(\pi_x^*(\beta|_{\overset{\circ}{W}_x})\right)\right),
\]
where $\Omega^{\bullet+1}\left(\mathcal{F}_x(\cdots)\right)$ denotes the foliation de Rham complex determined by the presymplectic form ($\cdots$).
\end{proof}

For later use, we fix the following auxiliary choices.
\begin{choice}\label{3choi}
For each $x \in s^{-1}(0),$ we choose:
\begin{enumerate}
\item[--] a decomposition $TW_x \simeq T\mathcal{F}_x \oplus G_x$ for some sub-vector bundle $G_x,$
\item[--] global coordinates on $W_x \simeq B^n$, written as $(y_1, \cdots, y_{n-k}, q_1, \cdots, q_k),$ where $\frac{\partial}{\partial q_1}, \cdots, \frac{\partial}{\partial q_k} \in \Gamma(T\mathcal{F}_x),$
\item[--] a global orthonormal frame $(e_1, \cdots, e_r)$ of the trivial vector bundle $E|_{W_x}.$
\end{enumerate}
\end{choice}

Any chart naturally restricts to an open subset of its base, provided the group action is compatible with that subset:

\begin{defn}[Open subchart]
Let $\mathcal{U}= (U, E, s, \Gamma, \psi)$ be an $L_{\infty}$-Kuranishi chart of $X$ and $U_0 \subset U$ an open subset with $\Gamma \cdot U_0 \subset U_0.$ Then the restricted tuple
\[
\mathcal{U}|_{U_0}= (U_0, E|_{U_{0}}, s|_{U_{0}}, \Gamma, \psi |_{(U_0 \cap s^{-1}(0))/\Gamma})
\]
canonically determines an $L_{\infty}$-Kuranishi chart, called the \textit{open subchart} of $\mathcal{U}$ on $U_0.$
\end{defn}

The chart morphisms we introduce shortly require the notion of completed algebras.
\begin{defn}[Completed algebras]\label{ladef}
Let $\phi : V \rightarrow W$ be a smooth map between manifolds. We consider an ideal of $C^{\infty}(W)$,
\begin{equation}\nonumber
I_{\phi} := \{f \in C^{\infty}(W) \mid f|_{\mathrm{Im}\,\phi} \equiv 0\}
\end{equation}
and denote
\[
C^{\infty}(W)^{(j)} := C^{\infty}(W)/I_{\phi}^j \cdot C_{\phi}^{\infty}(W) = C^{\infty}(W)/I_{\phi}^j.
\] We define the \textit{completion of} $\mathcal{C}_{x}$ by
\begin{equation}\nonumber
\mathcal{C}_{x, \phi} := \left(\bigwedge\nolimits^{-\bullet}\Gamma(E^*|_{W_x})\right)_{\phi} \oplus \Omega^{\bullet + 1}_{\mathrm{aug}, \phi}(\mathcal{F}_x),
\end{equation}
where the Koszul part is
\begin{equation}\label{kdzpt}
\left(\bigwedge\nolimits^{-\bullet}\Gamma(E^*|_{W_x})\right)_{\phi} :=  C^{\infty}_{\phi}(W_x) \otimes_{C^{\infty}(W_x)} \bigwedge\nolimits^{-\bullet}\Gamma(E^*|_{W_x}),
\end{equation}
and we take the inverse limit
\begin{equation}\label{ivlmt}
\begin{split}
C_{\phi}^{\infty}(W) &:= \lim_{\longleftarrow} C^{\infty}(W)^{(j)}\\
&= \lim\limits_{\longleftarrow} \left\{ \cdots \xrightarrow{p_{3,2}}
C^{\infty}(W)^{(2)} \xrightarrow{p_{2,1}}
C^{\infty}(W)^{(1)} \right\},
\end{split}
\end{equation}
where $p_{j+1,j} : C^{\infty}\left(W'_{\phi(x)}\right)^{(j+1)} \rightarrow  C^{\infty}\left(W'_{\phi(x)}\right)^{(j)}$ denotes the canonical projection for each $j \geq 1.$

This is the Koszul complex with completion,
\[
0 \rightarrow  \left(\bigwedge\nolimits^{r}\Gamma(E^*|_{W_x})\right)_{\phi} \xrightarrow{l_1^{\mathrm{K},\phi}} \cdots \xrightarrow{l_1^{\mathrm{K},\phi}} \Gamma(E^*|_{W_x})_{\phi}  \rightarrow C^{\infty}(W_x)_{\phi}  \rightarrow 0.
\]
Its $L_{\infty}[1]$-structure
\[
l^{\mathrm{K}, \phi}_k: {\left(\bigwedge\nolimits^{i}\Gamma(E^*|_{W_x})\right)_{\phi}}^{\otimes k} \rightarrow \left(\bigwedge\nolimits^{i-1}\Gamma(E^*|_{W_x})\right)_{\phi}
\]
is defined, for each $1 \leq i \leq r$, as follows. For $j\geq 1,$ $h \in C^{\infty}(W_x)^{(j)},$ and $a \in \bigwedge\nolimits^{-\bullet}\Gamma(E^*|_{W_x}),$ we set
\[
(j) : l^{\mathrm{K}, \phi}_1(h \otimes a) := [1]_{j-2} \otimes \iota_{s|_{W_x}}\left(\widetilde{h} a\right),
\]
where $\widetilde{h}$ is a representative in $C^{\infty}(W_x)$ satisfying $h = \widetilde{h} + I_{\phi}^j,$ and $[1]_{j-2} := 0$ by definition for $j \leq 2.$ All higher operations $l_{k}^{\mathrm{K}, \phi}$, $k \geq 2,$ are set to zero, so the $L_{\infty}$-relations for $\{l^{\mathrm{K}}_k\}_{k \geq 1}$ hold trivially.

The well-definedness of $l^{\mathrm{K}, \phi}_k$ requires verifying two conditions: (i) independence of the choice of representative $\widetilde{h},$ and (ii) compatibility with the choice of $(j)$. Both follow directly from the fact that $\iota_{s|_{W_x}}$ respects the restriction maps of sections; the verification is analogous to that in the proof of Lemma \ref{vlocp}, and we leave it to the reader.

The de Rham part $\Omega^{\bullet + 1}_{\mathrm{aug}, \phi}(\mathcal{F}_x)$ is the completed foliation de Rham complex \textit{with augmentation}, given by
\[
\Omega^{\bullet + 1}_{\mathrm{aug}, \phi}(\mathcal{F}_x) := \overbrace{\Omega^{\bullet + 1}(\mathcal{F}_x)_{\phi}}^{\deg \geq -1} \oplus \overbrace{\left(C^{\infty}(W_x)_{\mathcal{F}_x}\right)_{\phi}}^{\deg = -2},
\]
where
\[
\begin{cases}
\Omega^{\bullet + 1}(\mathcal{F}_x)_{\phi} := C_{\phi}^{\infty}(W_x) \otimes_{C^{\infty}(W_x)} \Omega^{\bullet + 1}(\mathcal{F}_x),\\
\big(C^{\infty}(W_x)_{\mathcal{F}_x}\big)_{\phi} : = \ker \big(l_1^{\mathrm{dR}} : \Omega^{-1}(\mathcal{F}_x)[1]_{\phi} \rightarrow \Omega^{0}(\mathcal{F}_x)[1]_{\phi}\big).
\end{cases}
\]
The de Rham part $L_{\infty}[1]$-structure $l_k^{\mathrm{dR}, \phi}$ is obtained by applying Proposition \ref{augomega} to the $L_{\infty}[1]$-algebra $\Omega^{\bullet +1}(\mathcal{F}_x)_{\phi}$ from Definition \ref{defldrc}.

Finally, $\mathcal{C}_{x,\phi}$ together with $\left\{l^{\phi}_k := l^{\mathrm{K}, \phi}_k \oplus l^{\mathrm{dR}, \phi}_k\right\}$ forms an $L_{\infty}[1]$-algebra, and the $L_{\infty}[1]$-relations can be verified in a straightforward manner.
\end{defn}

\begin{lem}
We have $C^{\infty}(W_x)$-module isomorphisms
\[
\begin{split}
\left(\bigwedge\nolimits^{-\bullet}\Gamma(E^*|_{W_x})\right)_{\phi} &\simeq \lim_{\longleftarrow} \left(C^{\infty}(W)^{(j)} \otimes_{C^{\infty}(W_x)} \bigwedge\nolimits^{-\bullet}\Gamma(E^*|_{W_x})\right),\\
\Omega^{\bullet + 1}(\mathcal{F}_x)_{\phi} &\simeq
\lim_{\longleftarrow} \left(C^{\infty}(W)^{(j)} \otimes_{C^{\infty}(W_x)} \Omega^{\bullet + 1}(\mathcal{F}_x)\right).
\end{split}
\]
That is, for the completions above, we may equivalently consider the inverse systems
\[
\cdots \xrightarrow{p_{3,2} \otimes \mathrm{id}_{(\cdots)}}
C^{\infty}(W)^{(2)} \otimes \bigwedge\nolimits^{-\bullet}\Gamma(E^*|_{W_x}) \xrightarrow{p_{2,1} \otimes \mathrm{id}_{(\cdots)}}
C^{\infty}(W)^{(1)} \otimes \bigwedge\nolimits^{-\bullet}\Gamma(E^*|_{W_x})
\]
and
\[
\cdots \xrightarrow{p_{3,2} \otimes \mathrm{id}_{(\cdots)}}
C^{\infty}(W)^{(2)} \otimes \Omega^{\bullet + 1}(\mathcal{F}_x) \xrightarrow{p_{2,1} \otimes \mathrm{id}_{(\cdots)}}
C^{\infty}(W)^{(1)} \otimes \Omega^{\bullet + 1}(\mathcal{F}_x).
\]
\end{lem}

\begin{proof}
We have isomorphisms
\[
\begin{split}
\left(\lim_{\longleftarrow} C^{\infty}(W)^{(j)}\right) \otimes \bigwedge\nolimits^{-\bullet}\Gamma(E^*|_{W_x}) &\simeq \lim_{\longleftarrow} \left(C^{\infty}(W)^{(j)} \otimes \bigwedge\nolimits^{-\bullet}\Gamma(E^*|_{W_x})\right),\\
\left(\lim_{\longleftarrow} C^{\infty}(W)^{(j)}\right) \otimes \Omega^{\bullet + 1}(\mathcal{F}_x) &\simeq \lim_{\longleftarrow} \left(C^{\infty}(W)^{(j)} \otimes \Omega^{\bullet + 1}(\mathcal{F}_x)\right),
\end{split}
\]
since $\Gamma(E^*|_{W_x})$ and $\Omega^{\bullet + 1}(\mathcal{F}_x)$ are flat, finitely presented $C^{\infty}(W_x)$-modules. (Indeed, they are free modules with finite bases.)
\end{proof}

Given a local algebra $\mathcal{C}_x,$ there is a natural map to its completion. For each $k \geq 1$, we define
\begin{equation}\label{varep}
\widehat{\varepsilon}_{\phi(x), \phi,k} : \mathcal{C}_x^{\otimes k} \rightarrow \mathcal{C}_{x, \phi}
\end{equation}
by
\begin{equation}\nonumber
\widehat{\varepsilon}_{\phi(x), \phi, k}\big((a_1, \xi_1), \cdots, (a_k,\xi_k)\big) := \begin{cases}
1 \otimes (a_1, \xi_1) = (1 \otimes a_1, 1 \otimes \xi_1) &\text{ if } k = 1,\\
0 &\text{ if } k \geq 2,
\end{cases}
\end{equation}
and set $\widehat{\varepsilon}_{\phi(x), \phi} := \left\{\widehat{\varepsilon}_{\phi(x), \phi,k}\right\}_{k \geq 1}.$

\begin{lem}\label{vecpf}
$\widehat{\varepsilon}_{\phi(x), \phi}$ is an $L_{\infty}[1]$-morphism.
\end{lem}

\begin{proof}
Since $\widehat{\varepsilon}_{\phi(x), \phi,k}$ vanishes for all $k \geq 2,$ it suffices to show
\[
\widehat{\varepsilon}_{\phi(x), \phi, 1}\big(l_k\big((a_1, \xi_1), \cdots, (a_k,\xi_k)\big)\big) = l^{\phi}_k\left(\widehat{\varepsilon}_{\phi(x), \phi, 1}(a_1, \xi_1), \cdots, \widehat{\varepsilon}_{\phi(x), \phi, 1}(a_k,\xi_k)\right)
\]
for each $k \geq 1,$ $a_i \in \bigwedge\nolimits^{-\bullet}\Gamma(E^*|_{W_x})$ and $\xi_i \in \Omega^{\bullet + 1}_{\mathrm{aug}}(\mathcal{F}_x), \ 1 \leq i \leq k.$

For the Koszul part, we have
\[
\begin{split}
\widehat{\varepsilon}_{\phi(x), \phi, 1}^{\mathrm{K}}\big(l_1^{\mathrm{K}}(a_1)\big) &= \widehat{\varepsilon}_{\phi(x), \phi, 1}^{\mathrm{K}}\big(\iota_{s|_{W_x}}(a_1)\big)\\ &= 1 \otimes \iota_{s|_{W_x}}(a_1) = l^{\mathrm{K}, \phi}_{1}(1 \otimes a_1) = l_1^{\mathrm{K}, \phi}\big(\widehat{\varepsilon}_{\phi(x), \phi, 1}^{\mathrm{K}}(a_1)\big),
\end{split}
\]
and for the de Rham part,
\[
\begin{split}
\widehat{\varepsilon}_{\phi(x), \phi, 1} \big( l^{\mathrm{dR}}_k(\xi_1,&\dots,\xi_k) \big)
= 1 \otimes l^{\mathrm{dR}}_k (\xi_1, \dots, \xi_k) = 1 \otimes l^{\mathrm{dR}}_k (\widetilde{1} \cdot \xi_1, \dots, \widetilde{1} \cdot \xi_k)\\
&= l^{\mathrm{dR},\phi}_k (1 \otimes \xi_1, \dots, 1 \otimes \xi_k)  = l^{\mathrm{dR},\phi}_k \big( \widehat{\varepsilon}_{\phi(x), \phi, 1}(\xi_1),\dots,\widehat{\varepsilon}_{\phi(x), \phi, 1}(\xi_k)\big).
\end{split}
\]
\end{proof}

The acyclicity of the completed $L_{\infty}[1]$-algebras is inherited by the completions of sections whose vanishing order at the image is 1. The following result states this precisely and is used to obtain a Poincar\'{e}-type theorem for completions in Corollary \ref{corpoinc}.
 
\begin{lem}\label{lemab}
Write the section $s$ in the orthonormal frame $\{e_m\}$ of $\Gamma(E)$ as
\[
s = \sum_m s_m e_m,
\]
and suppose that $s_m \in I_\phi \setminus I_\phi^2$ for each $m$. Then, in the context of Definition \ref{ladef}, we have:
\begin{enumerate}[label = (\roman*)]
\item Suppose the cohomology vanishes:
\[
H^i\left(\bigwedge\nolimits^{-\bullet}\Gamma(E^*|_{W_x})\right) = 0
\]
for $i \leq -1.$ Then the cohomology of the completion also vanishes:
\[
H^i\left(\left(\bigwedge\nolimits^{-\bullet}\Gamma(E^*|_{W_x})\right)_{\phi}\right) = 0
\]
for the same range of $i$.
\item
We have
\[
H^i\left(\left( \Omega^{\bullet +1}(\mathcal{F}_x) \right)_{\phi}\right)=0
\]
for $i \geq 0.$
\end{enumerate}
\end{lem}
\begin{proof}
\begin{enumerate}[label = (\roman*)]
\item
Fix $j \geq 1,$ and consider
\[
\left(\bigwedge\nolimits^{i+1}\Gamma(E^*|_{W_x})\right)_{\phi} \xrightarrow{d_{\phi}^{(i+1)}}  \left(\bigwedge\nolimits^{i}\Gamma(E^*|_{W_x})\right)_{\phi} \xrightarrow{d_{\phi}^{(i)}} \left(\bigwedge\nolimits^{i-1}\Gamma(E^*|_{W_x})\right)_{\phi},
\]
given by $d_{\phi}^{(i)}(h \otimes a) := [1]_{j-2} \otimes \iota_{s|_{W_x}}\left(\widetilde{h}a\right)$ for $h \in C^{\infty}_{\phi}(W_x)^{(j)} := C^{\infty}(W_x)/I_{\phi}^j,$ $a \in \bigwedge\nolimits^{i}\Gamma\left(E^*|_{W_x}\right),$ and $\widetilde{h} \in C^{\infty}(W_x)$ such that $\left[\widetilde{h}\right]_j = h.$ For a kernel element $\sum_{l}h_l \otimes a_l \in \ker d_{\phi}^{(i)},$ we have
\[
0 = d^{(i)}_{\phi} \left(\sum_{l}h_l \otimes a_l\right) = \sum_l [1]_{j-2} \otimes  \iota_{s|_{W_x}}(\widetilde{h}_la_l) = [1]_{j-2} \otimes   \iota_{s|_{W_x}}\left(\sum_l\widetilde{h}_la_l\right),
\]
so that $\sum_l\widetilde{h}_la_l \in \ker \iota_{s|_{W_x}}.$ By hypothesis, there exists $b \in \bigwedge\nolimits^{i-1}\Gamma(E^*|_{W_x})$ such that $\iota_{s|_{W_x}}(b) = \sum_l\widetilde{h}_la_l.$ Then
\[
\begin{split}
d^{(i-1)}_{\phi}([1]_{j} \otimes b) = [1]_{j-2} \otimes  \iota_{s|_{W_x}}\left(\widetilde{1} \cdot b\right) &=  [1]_{j-2} \otimes \iota_{s|_{W_x}}(b) =  [1]_{j-2} \otimes \sum_l\widetilde{h}_la_l \\
&= \sum_l [1]_{j-2} \otimes \widetilde{h}_la_l = \sum_l h_l \otimes a_l,
\end{split}
\]
that is, $ \sum_l h_l \otimes \xi_l \in \mathrm{Im}\,d_{\phi}^{(i-1)}.$

For a different choice of $j,$ compatibility is verified as follows.

\[
C^{\infty}(W_x)/I_{\phi}^{j+1} \xrightarrow{p_{j+1,j}} C^{\infty}(W_x)/I_{\phi}^{j}.
\]

Consider $h \in C^{\infty}(W_x)^{(j)}$, $h' \in C^{\infty}(W_x)^{(j+1)}$ such that $h = p_{j+1,j}(h') = h' + I_\phi^j / I_\phi^{j+1},$ with representatives $\widetilde{h}, \widetilde{h}' \in C^{\infty}(W_x)$ satisfying $\widetilde{h} + I_\phi^j = {h}, \ \widetilde{h}' + I_\phi^{j+1} = h'.$
Then $\widetilde{h} = \widetilde{h}' + \widetilde{g}$ for some $\widetilde{g} \in I_\phi^j$.

There exist $b, b' \in \bigwedge^i \Gamma(E^*|_{W_x})$ such that
\begin{equation}\label{afhlkfsd}
\iota_{s|W_x}(b) = \sum_l \widetilde{h}_l a_l, \quad \text{and} \quad
\iota_{s|W_x}\left(b'\right) = \sum_l \widetilde{h}'_l a_l
= \sum_l \widetilde{h}_l a_l + \sum_l \widetilde{g}_l a_l.
\end{equation}
For $[1]_j \in C^\infty(W_x)^{(j)},$ we have
\[
(j): d_\phi^{(i)} ([1]_{j} \otimes b) = [1]_{j-2} \otimes \iota_{s|W_x}\left(\widetilde{1} \cdot b\right) = [1]_{j-2} \otimes \sum_l \widetilde{h}_l a_l
= \sum_l [1]_{j-2} \otimes \widetilde{h}_l a_l = \sum\limits_l h_l \otimes a_l,
\]
and
\[
(j+1): d^{(i)}_{\phi} \left([1]_{j+1} \otimes b'\right)
= [1]_{j-1} \otimes \iota_{s|W_x}\left(\widetilde{1} \cdot b'\right)
= [1]_{j-1} \otimes \sum_l \widetilde{h}_l' a_l
= \sum_l h'_l \otimes a_l,
\]
for the choices of $j$ and $j+1,$ respectively.
On the other hand, taking the difference of the two equations in (\ref{afhlkfsd}), we obtain
\[
\sum_l \widetilde{g}_l a_l = \iota_{s|_{W_x}}(b-b') = \sum_m s_m|_{W_x} \cdot c_{l_1,\ldots,l_n}
\, e^*_{l_1} \wedge \cdots \wedge \widehat{e^*_{l_m}} \wedge \cdots \wedge e^*_{l_n}
\]
for some $c_{l_1 \cdots l_n}.$

From the condition $s_m|_{W_x} \in I_\phi \setminus I_\phi^2$ together with $\widetilde{g}_l \in I_\phi^j$, it follows that $c_{i_1,\ldots,i_n} \in I_\phi^j$ for all $i_1,\ldots,i_n.$

Compatibility then follows from
\[
\begin{split}
(j-1): &\ p_{j-1,j-2}\,d^{(i)}_{\phi}\big([1]_{j-1} \otimes (b - b')\big)\\
&= p_{j-1,j-2} \Big( [1]_{j-3} \otimes \sum\limits_{(\cdots)} c_{i_1,\ldots,i_n} \,
e^*_{i_1} \wedge \cdots \wedge \widehat{e^*_{i_m}} \wedge \cdots \wedge e^*_{i_n} \Big)\\
&= \sum\limits_{(\cdots)} p_{j-1,j-2}\Big( [c_{i_1,\ldots,i_n}]_{j-3} \otimes e^*_{i_1} \wedge \cdots \wedge
\widehat{e^*_{i_m}} \wedge \cdots \wedge e^*_{i_n} \Big) = 0
\end{split}
\]
for each $j.$

\item
Consider
\[
\left(\Omega^{i-1}(\mathcal{F}_x)[1] \right)_{\phi} \xrightarrow{d_{\phi}^{(i-1)}}\left(\Omega^{i}(\mathcal{F}_x)[1] \right)_{\phi} \xrightarrow{d_{\phi}^{(i)}} \left(\Omega^{i +1}(\mathcal{F}_x)[1] \right)_{\phi},
\]
given by $d_{\phi}^{(i)}(h \otimes \xi) := [1]_{j-1} \otimes d^{(i)}_{\mathcal{F}_x}\left(\widetilde{h}\xi\right)$ for $h \in C^{\infty}_{\phi}(W_x)^{(j)} := C^{\infty}(W_x)/I_{\phi}^j,$ $\xi \in \Omega^i(\mathcal{F}_x)[1],$ and $\widetilde{h} \in C^{\infty}(W_x)$ such that $[\widetilde{h}]_j = h.$ For a kernel element $\sum_{l}h_l \otimes \xi_l \in \ker d_{\phi}^{(i)},$ we have
\[
0 = d^{(i)}_{\phi} \left(\sum_{l}h_l \otimes \xi_l\right) = \sum_l [1]_{j-1} \otimes d_{\mathcal{F}_x}^{(i)}\left(\widetilde{h}_l\xi_l\right) = [1]_{j-1} \otimes  d_{\mathcal{F}_x}^{(i)}\left(\sum_l\widetilde{h}_l\xi_l\right),
\]
so that $\sum_l\widetilde{h}_l\xi_l \in \ker d_{\mathcal{F}_x}^{(i)}.$ By the Poincar\'{e} lemma for foliations, there exists $\eta \in \Omega^{i-1}(\mathcal{F}_x)[1]$ such that $d_{\mathcal{F}_x}^{(i-1)}(\eta) = \sum_l\widetilde{h}_l\xi_l.$ Then
\[
\begin{split}
d^{(i-1)}_{ \phi}([1]_j \otimes \eta) = [1]_{j-1} \otimes d^{(i-1)}_{\mathcal{F}_x}\left(\widetilde{1} \cdot \eta\right) &=  1 \otimes d^{(i-1)}_{\mathcal{F}_x}(\eta) =  [1]_{j-1} \otimes \sum_l\widetilde{h}_l\xi_l\\
&= [1]_{j-1} \otimes \sum_l\widetilde{h}_l\xi_l = \sum_l h_l \otimes \xi_l,
\end{split}
\]
that is, $ \sum_l h_l \otimes \xi_l \in \mathrm{Im}\,d_{\phi}^{(i-1)}.$ For a different $j,$ compatibility is verified as follows. First, we choose primitive forms as in the proof of Lemma \ref{pcrlem}, and set
\[
\zeta := \int_0^1 \left( p^*\left(\sum \left(\widetilde{h}_l \xi_l\right) - \sum \left(h_l \xi_l\right)\right)\right) dt
= \int_0^1 \left( p^*\left(\sum \widetilde{g}_l \xi_l\right)\right) dt,
\]
where $p : W \times [0,1] \rightarrow W$ is the projection onto the first component.

As in (i), for $h \in C^{\infty}(W_x)^{(j)}$, $h' \in C^{\infty}(W_x)^{(j+1)}$ with $h = p_{j+1,j}(h') = h' + I_\phi^j / I_\phi^{j+1},$ and respective representatives $\widetilde{h}, \widetilde{h}' \in C^{\infty}(W_x),$ there exists $\widetilde{g} \in I_\phi^j$ with  $\widetilde{h} = \widetilde{h}' + \widetilde{g}.$ Since $\widetilde{g}_l \in I_{\phi}^j$, we have
\[
\sum \left( p^*\left(\widetilde{g}_l \xi_l\right) \right)|_{I_{m\phi}} = 0,
\quad \forall t \in (0,1],
\]
for the map $W \rightarrow W$ given by
\[
(y_1,\ldots,y_m,\xi_1,\ldots,\xi_n) \mapsto (y_1,\ldots,y_m, t \xi_1,\ldots,t \xi_n),
\]
which is injective. Thus $\zeta|_{\mathrm{Im}_{\phi}} = 0$.
Moreover, $\zeta$ can be written as
\[
\zeta = \sum \widetilde{f}_l \zeta_l,
\]
for some $\widetilde{f}_l \in I_\phi^{\bar{j}}$ and $\zeta_l \in \Omega^{i}(\mathcal{F}_x)[1]$.

Compatibility with respect to $j$ is then verified by
\[
\begin{split}
(j) : p_{j,j-1} \big([1]_{j} \otimes (\eta - \eta') \big) = p_{j,j-1}([1]_{j} \otimes \zeta)
&= p_{j,j-1}\Big([1]_{j} \otimes \sum_l \widetilde{f}_l \zeta_l \Big) \\
&= \sum_l p_{j,j-1}\Big( [\widetilde{f}_l]_{j} \otimes \zeta_l \Big)
= 0
\end{split}
\]
for each $j.$

\end{enumerate}
\end{proof}

\begin{cor}\label{corpoinc}
$\Omega^{\bullet +1}_{\mathrm{aug},\phi}(\mathcal{F}_x)$ is an acyclic $L_{\infty}[1]$-algebra.
\end{cor}

\begin{proof}
By the Poincar\'{e} lemma for foliations, $H^*\big(\Omega^{\bullet +1}(\mathcal{F}_x)\big) = 0$ for all $* \geq 0.$ The preceding lemma then gives $H^*\big(\Omega^{\bullet +1}(\mathcal{F}_x)_{\phi}\big) = 0$ for all $* \geq 0.$ Since $\Omega^{\bullet +1}_{\mathrm{aug},\phi}(\mathcal{F}_x)$ agrees with $\Omega^{\bullet +1}(\mathcal{F}_x)_{\phi}$ in degrees $\geq -1,$ it follows that $H^*\big(\Omega^{\bullet +1}_{\mathrm{aug},\phi}(\mathcal{F}_x)\big) = 0$ in that range as well. The remaining cohomology, $H^*\big(\Omega^{\bullet +1}_{\mathrm{aug},\phi}(\mathcal{F}_x)\big)$ for $* \leq -1,$ vanishes by the definition of the augmentation.
\end{proof}

\begin{rem}\label{auglmoloc}
Note that Lemma \ref{auglmo}, on the existence of morphisms between augmented $L_{\infty}[1]$-algebras, continues to hold when either the domain or the target of the $L_{\infty}[1]$-morphism is replaced by its completion.
\end{rem}

\subsection{Special cases}

In the context of Definition \ref{ladef} on local $L_{\infty}[1]$-algebras, several special cases are worth discussing.

First, for surjective $\phi$ the completion takes a particularly simple form.
\begin{lem}\label{surjcp}
If $\phi$ is surjective, then
\[
\mathcal{C}_{\phi(x),\phi} \simeq \mathcal{C}_{\phi(x)}.
\]
\end{lem}

\begin{proof}
We have $I_\phi=\{0\},$ so that
\[
C_\phi^\infty\bigl(W_{\phi(x)}\bigr) =  \lim\limits_{\longleftarrow} C^\infty\bigl(W_{\phi(x)}\bigr) /I_{\phi}^n \simeq C^\infty\bigl(W_{\phi(x)}\bigr),
\]
and hence
\[
\mathcal{C}_{\phi(x),\phi}
= C_\phi^\infty\bigl(W_{\phi(x)}\bigr) \otimes \mathcal{C}_{\phi(x)}
\simeq C^\infty\bigl(W_{\phi(x)}\bigr) \otimes \mathcal{C}_{\phi(x)} \simeq \mathcal{C}_{\phi(x)}.
\]
\end{proof}

Second, we consider completions for open subcharts. Let $o : U \hookrightarrow U'$ be an open inclusion, and let $\mathcal{U} := \mathcal{U}'|_U$ denote the open subchart on $U.$ A sketch of the proof of the following lemma is given at the end of Subsection \ref{rtfo3kc}.

\begin{lem}[Completion at an open embedding with open subchart data]\label{itavsteai}
In the situation above, there exists an $L_\infty[1]$-quasi-isomorphism
\[
\widehat{o}_x : \mathcal{C}'_{o(x),o} \simeq \mathcal{C}_x.
\]
\end{lem}

Finally, we define the expansion of a chart.
\begin{defn}[Expansion of a chart]\label{expvs}
Let $\mathcal{U} = (U, E, s, \Gamma, \psi)$ be a Kuranishi chart on $X$ as in Definition \ref{kurdef}, and let $V$ be a finite-dimensional vector space. From $\mathcal{U}$ we construct another chart on $X$, called the \textit{expansion of $\mathcal{U}$ by $V$},
\[
\mathcal{U} \times V := (U \times V, E \times V, s \times \mathrm{id}_V, \Gamma, \psi),
\]
consisting of the following data:
\begin{itemize}
\item[--] $U \times V$ is the base equipped with the closed two-form $\pi^* \beta$, where $\pi: U \times V \to U$ denotes projection onto the $U$-component.
\item[--] $E \times V \to U \times V$ is the vector bundle obtained in the natural way from $E \to U.$
\item[--] $s \times \text{id}_V: U \times V \longrightarrow E \times V$ is the section given by $(y,v) \mapsto \big(s(y), v\big).$
\item[--] $\Gamma$ is the finite group acting only on the $U$-component of $U \times V.$
\item[--] $\psi : {(s \times \text{id}_V)^{-1}(0) }/{\Gamma}\simeq {s^{-1}(0)}/{\Gamma} \xhookrightarrow{\psi} X$ is the homeomorphism coinciding with the $\psi$ of $\mathcal{U}.$
\item[--] $W_{(x,0)} := W_x \times V$ is the open neighborhood of the zero point $(x,0).$
\item[--] The local $L_\infty[1]$-algebra at $(x,0) \in (s \times \mathrm{id}_V)^{-1}(0)$ is given by
\[
\mathcal{C}^V_{(x,0)} := \bigwedge\nolimits^{-\bullet} \Gamma\left(\left(\pi^* E \oplus V\right)^*\big|_{W_{(x,0)}}\right) \oplus \Gamma_{\mathrm{aug}}\left(\bigwedge\nolimits^{\bullet+1}\left(\pi^*T\mathcal{F} \oplus V\right)^*\big|_{W_{(x,0)}}\right).
\]
\item[--] The $L_{\infty}[1]$-operations $\left\{l_{(x,0),k}^V: {\mathcal{C}^V}_{(x,0)}^{\otimes k} \rightarrow \mathcal{C}_{(x,0)}^V\right\}_{k \geq 1}$ on $\mathcal{C}^V_{(x,0)}$,
given by $l_{(x,0), k}^{V, \mathrm{K}} \oplus l_{(x,0), k}^{V, \mathrm{dR}},$ with each component given by
\[
\begin{cases}
&\begin{split}
l_{(x,0), k}^{V, \mathrm{K}}&\big((a_1, w_1^*) , \cdots , (a_k, w_k^*)\big)\\
&\quad \quad := \begin{cases}
\iota_{s \times \mathrm{id}_V}(a_1,w^*_1) = \left\{\iota_s(a_1)(x), \iota_{v}w_1^*(x,v)\right\}_{(x,v) \in W_x \times V} &\text{ if } k =1,\\
0 &\text{ if } k \geq 2,
\end{cases}
\end{split}\\
&l^{V,\mathrm{dR}}_{(x,0),k} \big((\xi_1, \tau_1), \cdots , (\xi_k, \tau_k)\big) := \left(l^{\mathrm{dR}}_{x,k} \left(\xi_1, \cdots,\xi_k \right), l^V_{k}(\tau_1, \cdots, \tau_k)\right).
\end{cases}
\]
Here, we define
\[
 l^V_{k}(\tau_1, \cdots, \tau_k) := \begin{cases}
\sum_i\tau_1(v_i) & \text{with respect to a fixed basis } \{v_i\} \text{ of } V, \text{ if } k = 1,\\
0  & \text{ if } k \geq 2,
\end{cases}
\]
for $a_i \in \Gamma\left(\pi^*E^*\big|_{W_{(x,v)}}\right), \ \xi_i \in \pi^*\left(\Omega^{\bullet +1}\left(\mathcal{F}\big|_{W_{(x,v)}}\right)\right)$, and $w^*_i, \tau_i \in \Gamma\left(V^*|_{W_{(x,v)}}\right)$. It follows immediately that the family $\{l_{(x,0),k}^V\}$ forms an $L_\infty[1]$-algebra. Moreover, we have
\[
H^*\left(\Gamma\big(\bigwedge\nolimits^{\bullet+1}(\pi^*T\mathcal{F} \oplus V)^*\big|_{W_{(x,0)}}\big), l^{V, \mathrm{dR}}_{(x,0),1}\right) \simeq H^*\big(\Omega^{\bullet + 1}(\mathcal{F}_x), l^{\mathrm{dR}}_{x,1}\big),
\]
which can be verified as follows: since $l^{\mathrm{dR}}_{x,k}$ and $l^V_k$ are defined independently, one may apply the K\"{u}nneth theorem together with the fact that the cohomology computed by $l_1$ on an affine space is trivial. We then add the augmentation to the de Rham part, and equip it with an $L_{\infty}[1]$-structure using Proposition \ref{augomega}.

\end{itemize}
\end{defn}

\begin{rem}
Example \ref{dcotopwx} shows that when two charts differ only in the choice of the open neighborhood $W_x,$ as is the case here, the resulting structures are independent of that choice.
\end{rem}

From this point forward, we shall often suppress the adjective $L_\infty$- from $L_\infty$-Kuranishi structures and simply write Kuranishi structures, unless it is necessary for clarity.

\subsection{Morphisms of Kuranishi charts}

Let
\[
\mathcal{U} = (U, E, s, \Gamma, \psi) \text{ and  } \mathcal{U}' = (U', E', s', \Gamma', \psi')
\]
be Kuranishi charts on topological spaces $X$ and $Y,$ respectively, and suppose we are given a continuous map $f : X \rightarrow Y.$

\begin{defn}[Morphism of $L_\infty$-Kuranishi charts]\label{ourchrtmor}
A \textit{morphism of Kuranishi charts} $\Phi : \mathcal{U} \rightarrow \mathcal{U}'$ is a pair $\Phi = \left(\phi, \widehat{\phi}\right)$, where:
\begin{enumerate}
\item[--] $\phi : U \rightarrow U'$ is a $\left(\Gamma, \Gamma'\right)$-equivariant map of smooth manifolds that need not respect the closed two-forms,
\item[--] $\widehat{\phi} = \left\{\widehat{\phi}_x :  \mathcal{C}'_{\phi(x)} \rightarrow \mathcal{C}_x \right\}_{x \in s^{-1}(0)}$ is a family of $L_{\infty}[1]$-morphisms,
\end{enumerate}
satisfying the following conditions:
\begin{enumerate}[label = (\roman*)]
\item $\psi' \circ \phi = f \circ \psi$ on $s^{-1}(0).$
\item $\phi(W_x) \subset W'_{\phi(W_x)}$ for each $x \in s^{-1}(0).$
\item $\widehat{\phi}_x$ factors through $\mathcal{C}'_{\phi(x), \phi}$ for each $x \in s^{-1}(0);$ that is, $\widehat{\phi}_x = \widehat{\phi}_x^{\mathrm{c}} \circ \widehat{\varepsilon}_{\phi(x), \phi}$ for some $L_{\infty}[1]$-morphism $\widehat{\phi}^{\mathrm{c}}_x :  \mathcal{C}'_{\phi(x), \phi} \rightarrow \mathcal{C}_x.$
\end{enumerate}
\end{defn}

Consider an element $ \left(\overline{a}^{ k}, \overline{\xi}^{ k}\right) := \left({a}_1 \otimes \cdots \otimes a_k, {\xi}_1 \otimes \cdots \otimes \xi_k\right) \in \left(\mathcal{C}'_{\phi(x)}\right)^{\otimes k}.$ From condition (iii) above, together with the definition of the canonical map $\widehat{\varepsilon}_{\phi(x),\phi}:\mathcal C'_{\phi(x)} \rightarrow \mathcal C_{\phi(x),\phi},$ it follows immediately that if $\textrm{supp}\left({a}_i,{\xi}_i\right)\subset W'_{\phi(x)} \setminus \mathrm{Im}\, \phi$ for some $i,$ then $\widehat{\varepsilon}_{\phi(x),\phi,k}\left(\overline{a}^k,\overline{\xi}^k\right)=0.$ Hence we obtain the following lemma:

\begin{lem}
If $\mathrm{supp}({a}_i, {\xi}_i) \subset W'_{\phi(x)} \setminus \mathrm{Im}\, \phi$ for some $i,$ then $\widehat{\phi}_{x,k}\left(\overline{a}^k,\overline{\xi}^k\right)=0.$
\end{lem}

Conversely, we also have the following.
\begin{lem}\label{pcpc}
Let
\[
\widehat{\kappa}_x : \mathcal C'_{\phi(x)} \rightarrow \mathcal C_x
\]
be an $L_{\infty}[1]$-morphism satisfying
\[
\widehat{\kappa}_{x,k}\left(\overline{a}^k,\overline{\xi}^k\right)=0
\]
for all $k$ and all $\left(\overline{a}^k,\overline{\xi}^k\right) = \left({a}_1 \otimes \cdots \otimes a_k, {\xi}_1 \otimes \cdots \otimes \xi_k\right)$ such that $\mathrm{supp}\left({a}_i, {\xi}_i\right)\subset W'_{\phi(x)} \setminus \mathrm{Im}\,\phi$ for some $i.$ Then there exists a map
\[
\widehat{\kappa}_x^{\mathrm{c}} :
\mathcal C'_{\phi(x),\phi} \rightarrow \mathcal C_x
\]
such that
\[
\widehat{\kappa}_x
=
\widehat{\kappa}_x^{\mathrm{c}}
\circ
\widehat{\varepsilon}_{\phi(x),\phi}.
\]
\end{lem}

\begin{proof}
$\widehat{\kappa}_x^{\mathrm{c}}$ is well defined once we view $\widehat{\kappa}_x$ as taking inputs modulo homogeneous elements $\left(\overline{a}^k,\overline{\xi}^k\right)$ with $\mathrm{supp}\left({a}_i, {\xi}_i\right)\subset W'_{\phi(x)} \setminus \mathrm{Im}\,\phi$ for some $i.$ The claim then follows from the construction of the completion $\mathcal C'_{\phi(x),\phi}.$
\end{proof}

\begin{rem}
\begin{enumerate}[label=(\roman*)]
\item In Definition \ref{ourchrtmor}, $\left(\Gamma, \Gamma'\right)$-equivariance implicitly presupposes a choice of group homomorphism $g : \Gamma \rightarrow \Gamma'.$ Note that condition (i) of Definition \ref{ourchrtmor} is well defined precisely because of this equivariance.
\item By condition (i) of Definition \ref{ourchrtmor}, the zero points of $s$ map into the zero set of $s'$ under a morphism of Kuranishi charts; that is, $\phi\left(s^{-1}(0)\right) \subset s'^{-1}(0).$
\end{enumerate}
\end{rem}

\begin{exam}[Different choices for the open neighborhood $W_x$]\label{dcotopwx}
Let $\mathcal{U}$ and $\mathcal{U}'$ be Kuranishi charts that are identical \textit{except} for the choice of open neighborhoods $W_x \underset{\text{open}} {\subset} W'_x$ at a zero point $x \in s^{-1}(0)$. Then Kuranishi charts are independent of such choices, in the following sense: there exists a morphism of charts $\Phi: \mathcal{U} \rightarrow \mathcal{U}'$ given by
\begin{itemize}
    \item[--] $\phi: U \rightarrow U$, the identity map;
    \item[--] $\widehat{\phi}^{\mathrm{c}}_x: \mathcal{C}'_{\phi(x),\phi} \xrightarrow{\simeq} \mathcal{C}_x$, the quasi-isomorphism from Lemma \ref{itavsteai}.
\end{itemize}
\end{exam}

\begin{defn}[Embedding of $L_\infty$-Kuranishi charts]\label{ourcemb}
Let $\mathcal{U} = (U, E, s, \Gamma, \psi)$ and $\mathcal{U}' = (U', E', s', \Gamma', \psi')$
be Kuranishi charts of $X.$ A morphism of charts $\Phi = (\phi, \widehat{\phi}) : \mathcal{U} \rightarrow \mathcal{U}'$ is called an \textit{embedding} if
\begin{enumerate}
\item[--] $\phi : U \hookrightarrow U'$ is a $(\Gamma, \Gamma')$-equivariant embedding of smooth manifolds, and
\item[--] $\widehat{\phi}^{\mathrm{c}}_{x} : \mathcal{C}'_{\phi(x), \phi} \rightarrow \mathcal{C}_x$ is a quasi-isomorphism of $L_{\infty}[1]$-algebras for each $x.$
\end{enumerate}
When an embedding $\Phi = \left(\phi, \widehat{\phi}\right)$ is given, we implicitly fix, for each $x,$ a retraction between the Euclidean balls,
\begin{equation}\nonumber
\pi_x : W_{\phi(x)}' \twoheadrightarrow \phi(W_x),
\end{equation}
restricting to the identity map $\mathrm{id}_{\phi(W_x)}$ on $\phi(W_x) \subset W'_{\phi(x)}.$
\end{defn}

\begin{defn}[Open embedding of $L_\infty$-Kuranishi charts]\label{ouremb}
An embedding of charts $\Phi = \left(\phi, \widehat{\phi}\right) : \mathcal{U} \rightarrow \mathcal{U}'$ is called \textit{open} if $\dim \mathcal{U} = \dim \mathcal{U}'$
\end{defn}

\begin{exam}[Chart morphism from an expansion]\label{cmfae}
Let $\mathcal{U} \times V$ be an expansion of a Kuranishi chart $\mathcal{U}$ (cf. Subsection \ref{expvs}). We consider a morphism of charts $P : \mathcal{U}\times V \to \mathcal{U}$, consisting of:
\begin{itemize}[leftmargin=*]
\item[--] the projection $\pi : U \times V\to U$ onto the $U$-component, which restricts to the isomorphism
\[
\pi|_{(s \times \text{id}_V)^{-1}(0)}: (s \times \text{id}_V)^{-1}(0) \simeq s^{-1}(0);
\]
\item[--] the $L_{\infty}[1]$-algebra morphism $\widehat{\pi}^{\mathrm{c}}_{(x,0)} : \mathcal{C}_{x,\pi} \simeq \mathcal{C}_x \to \mathcal{C}^V_{(x,0)}$, defined by
\[
\widehat\pi^{\mathrm{c}}_{(x,0),k}\bigl((a_1,\xi_1),\dots,(a_k,\xi_k)\bigr)
  :=\begin{cases}
    (\pi^*a_1,\pi^*\xi_1), & \text{ if } k=1,\\
    0, & \text{ if } k \geq 2.
  \end{cases}
\]
We denote $\widehat\pi_{(x,0)} := \widehat\pi_{(x,0)}^{\mathrm{c}} \circ \widehat{\varepsilon}_{x,\pi}.$

\item[--]  the $L_{\infty}[1]$-algebra morphism $\widehat{\pi}^{-1}_{(x,0)} :
    \mathcal{C}^{\mathrm{V}}_{(x,0)} \rightarrow \mathcal{C}_{x}$
defined by
\[
    \widehat{\pi}_{(x,0),k}^{-1}\left(\left(\left(a_1,v_1\right),\left(\xi_1, w_1\right)\right),\dots,\left(\left(a_k,v_k\right),\left(\xi_k, w_k\right)\right) \right) :=
    \begin{cases}
        \left(a_1, \xi_1\right) &\text{ if } k = 1\\
        0 &\text{ if } k \geq 2.
    \end{cases}
\]
By the fact that the $L_{\infty}[1]$-structure on $\mathcal{C}^{\mathrm{V}}_{(x,0)}$ is separately for each factor (with or without $\mathrm{V}$-factor), it is straightforward to verify that $ \widehat{\pi}_{(x,0),1}^{-1}$ and $ \widehat{\pi}_{(x,0),1}^{-1}$ satisfy the $L_{\infty}[1]$-relation,
and we omit it.

\end{itemize}
\end{exam}

It immediately follows that
$\widehat{\pi}^{-1}_{(x,0)} \circ \widehat{\pi}_{(x,0)} = \mathrm{id}_{\mathcal{C}_{x}}$. Note, however, that
$\widehat{\pi}_{(x,0)} \circ \widehat{\pi}^{-1}_{(x,0)} \neq
\mathrm{id}_{\mathcal{C}^{\mathrm{V}}_{(x,0)}}$ in general. Instead ,we have the following lemma whose proof is postponed to the end of Subsection \ref{rtfo3kc}.
\begin{lem}\label{phxv}
$\widehat{\pi}_{(x,0)}$ and $\widehat{\pi}^{-1}_{(x,0)}$ are
quasi-isomorphisms. Moreover, they are $L_{\infty}[1]$-homotopy inverse to each other.
\end{lem}

We conclude this subsection with a typical application of the preceding lemma.
\begin{exam}\label{gageeefffe}
Given a chart morphism
\[
\Phi: \mathcal{U} \rightarrow \mathcal{U}',
\]
there exists another morphism
\[
\Phi': \mathcal{U} \times V \hookrightarrow \mathcal{U}'
\]
extending $\Phi,$ that is, $\Phi'|_{\mathcal{U} \times \{0\}} \equiv \Phi,$ with the property that the base component map
\[
\phi' : U \times V \rightarrow U
\]
is surjective, for a choice of $V$ of sufficiently large dimension.
By Example \ref{dcotopwx}, the local base map
\[
\phi'|_{W_{(x,0)}}: W_{(x,0)} \rightarrow W'_{(\phi(x),0)},
\]
for each $(x,0) \in (s \times \mathrm{id}_V)^{-1}(0)$, may likewise be assumed surjective, modulo quasi-isomorphic changes of the local $L_{\infty}[1]$-algebra, by taking $W_{(x,0)}$ sufficiently large. The $L_{\infty}[1]$-component map $\widehat{\phi}'_{(x,0)}$ at $(x,0)$ can then be chosen as the composition
\[
\widehat{\phi}'_{(x,0)}: \mathcal{C}'_{\phi'(x,0),\phi'} = \mathcal{C}'_{\phi'(x,0)} =  \mathcal{C}'_{\phi(x)} \xrightarrow{\widehat{\varepsilon}'_{\phi(x),\phi}} \mathcal{C}'_{\phi(x), \phi} \xrightarrow{\widehat{\phi}_x} \mathcal{C}_x \xrightarrow[\simeq]{\widehat{\pi}_{(x,0)}} \mathcal{C}^V_{(x,0)},
\]
using Lemmata \ref{vecpf}, \ref{surjcp}, and \ref{phxv}.
\end{exam}

\subsection{Relation to FOOO Kuranishi charts}\label{rtfo3kc}
This subsection reviews FOOO Kuranishi charts and embeddings, and shows how they relate to the $L_{\infty}[1]$-Kuranishi chart theory introduced above. We begin by recalling the definition of an FOOO Kuranishi chart.

\begin{defn}[FOOO Kuranishi charts]\label{deflkur}
Let $X$ be a compact metrizable space. We call a tuple $\mathscr{U} := (U, E, s, \Gamma, \psi)$ an \textit{FOOO Kuranishi chart} of $X$ if the following conditions are satisfied:
\begin{enumerate}[label = (\roman*)]
\item[--] $U$ is a simply connected orbifold.
\item[--] $E$ is a vector bundle of finite rank over $U.$
\item[--] $s : U \rightarrow E$ is a smooth section.
\item[--] $\Gamma$ is a finite group acting on $U,$ preserving $s^{-1}(0).$
\item[--] $\psi : {s^{-1}(0)}/{\Gamma} \overset{\simeq}{\hookrightarrow} X$ is a homeomorphism onto its image.
\end{enumerate}
\end{defn}

\begin{defn}\label{foemb}
Given two FOOO Kuranishi charts $\mathcal{U}$ and $\mathcal{U}'$ of $X,$ an \textit{FOOO embedding} $\Phi := \left(\phi, \widetilde{\phi}\right) : \mathcal{U} \hookrightarrow \mathcal{U}'$ consists of:
\begin{enumerate}
\item[--] an orbifold embedding $\phi : U \hookrightarrow U',$
\item[--] a vector bundle embedding $\widetilde{\phi} : E \hookrightarrow E',$
\end{enumerate}
required to satisfy the following conditions:
\begin{enumerate}
\item[(i)] $\widetilde{\phi} \circ s = s' \circ \phi,$
\item[(ii)] $\psi' \circ \phi = \psi$ on $s^{-1}(0),$
\item[(iii)](\textit{Tangent bundle condition}) $ds'$ induces an isomorphism
\begin{equation}\label{tancond}
\left[ds'_{\phi(x)}\right] : \frac{T_{\phi(x)}U'}{\phi_*(T_xU)} \xrightarrow{\simeq} \frac{E'_{\phi(x)}}{\widetilde{\phi}(E_{x})},
\end{equation}
at each $x \in s^{-1}(0).$
\end{enumerate}
\end{defn}

\begin{defn}[FOOO Kuranishi space]\label{fo3ks}
Let $X$ be a compact metrizable space. An \textit{FOOO Kuranishi structure} $\widehat{\mathscr{U}}$ on $X$ assigns, by definition, to each point $p \in X$ an FOOO Kuranishi chart $\mathscr{U}_p := (U_p, E_p, s_p, \Gamma_p, \psi_p),$ and, to a pair of points $p, q \in X$ with $p \in \mathrm{Im}\,\psi_q,$ the following data:
\begin{enumerate}[label=(\roman*)]
\item[--] an open subset $U_{pq} \subset U_p,$
\item[--] an FOOO embedding of the same virtual dimension (called the \textit{coordinate change}) $\Phi_{pq}=(\phi_{pq}, \widetilde{\phi}_{pq})$ from $\mathscr{U}_p|_{U_{pq}}$  to  $ \mathscr{U}_q,$
\end{enumerate}
satisfying the compatibility conditions
\begin{enumerate}[label = (\roman*)]
\item $\Phi_{pr}|_{U_{pqr}} = \Phi_{qr} \circ \Phi_{pq}|_{U_{pqr}}$ for $q \in \mathrm{Im}\, \psi_p, \ r \in \psi_q \big(s_q^{-1}(0) \cap U_{qr}\big),$
\item $\Phi_{pp} = (U_p, \mathrm{id}_{{p}}, \widehat{\mathrm{id}}_{p, x}),$
\item $\psi_p\big(s_p^{-1}(0) \cap U_{pq}\big) =\mathrm{Im}\, \psi_p \cap \mathrm{Im}\, \psi_q,$
\end{enumerate}
where $U_{pqr} := \phi_{pq}^{-1}(U_{qr}) \cap U_{pr}.$
An \textit{FOOO Kuranishi space} is defined by the pair $(X, \widehat{\mathscr{U}}).$
\end{defn}

Three points are worth noting here. First, the coordinate changes are required to satisfy the compatibility conditions \textit{only on the zero locus.} Second, the cocycle condition is imposed \textit{on the base maps alone}, and not on the $L_{\infty}[1]$-component. The reason for this is that $L_{\infty}$-compatibilities always hold (cf.\ Remark \ref{whnli} (iii)); this point is revisited in \cite{Kim3}. Third, we consider pairs $(p,q)$ satisfying $\mathrm{Im}\,\psi_p \cap \mathrm{Im}\,\psi_q \neq \emptyset$, whereas in the FOOO setting coordinate changes are defined only for pairs with $p \in \mathrm{Im}\,\psi_q$. Together, these three points yield the flexibility we seek to achieve.

FOOO Kuranishi charts, as defined in \cite{FOOO1} and \cite{FOOO2}, can be regarded as special cases of our construction in the following sense. Given an FOOO Kuranishi chart $\mathscr{U} = (U, E, s, \Gamma, \psi)$, we equip its base $U$ with the zero presymplectic form $\omega_U \equiv 0$. Note that in this case $T\mathcal{F} = TU$.

Let $\mathscr{U} = (U, E, s, \Gamma, \psi)$ and $\mathscr{U}' = (U', E', s', \Gamma', \psi')$ be FOOO Kuranishi charts, regarded as Kuranishi charts in our sense as above, and let $\left(\phi, \widetilde{\phi}\right) : \mathscr{U} \rightarrow \mathscr{U}'$ be an FOOO embedding between them.

\begin{lem}[Additional conditions]\label{addcond}
We add two further conditions to the definition of FOOO embeddings, alongside conditions (i), (ii), and (iii) of Definition \ref{foemb}. Before proceeding, we write $E^c$ for a complement of $\widetilde{\phi}(E)$ in $E'$, and $p^c: E' \twoheadrightarrow E^c$ for the canonical projection. We additionally require:
\begin{enumerate}[label=(\roman*), start=4]
\item $p^c(s')|_{\phi(U)} \equiv 0.$
\item (After fixing a local trivialization,) the tangent bundle condition
\begin{equation}\label{tancond4}
\left[d\left(s'|_{W_x}\right)_y\right] : \frac{T_{\phi(x)}W'_{\phi(x)}}{\phi_*(T_yW_x)} \xrightarrow{\simeq} \frac{E'_{\phi(y)}}{\widetilde{\phi}(E_{y})}
\end{equation}
holds for all $x \in s^{-1}(0)$ and every $y \in W_x$ (not merely for $x$ itself).
\end{enumerate}
Then conditions (iv) and (v) always hold by taking $W_x$ sufficiently small, if necessary.
\end{lem}
\begin{proof}
\begin{enumerate}[label=(\roman*), start=4]
\item This condition is satisfied by (i) of Definition \ref{foemb}.
\item The linearization (with a choice of local trivialization of $E$ over $W_x$) being an isomorphism is an open condition in $x \in W_x$. Hence, by taking $W_x$ smaller if necessary, one can ensure that $\left[d_ys'|_{W_x}\right]$ is an isomorphism for all $y \in W_x$.
\end{enumerate}
\end{proof}

Suppose we are given Kuranishi charts $\mathcal{U}$ and $\mathcal{U}'$ determined as explained above. We now show that our definition of an open embedding is the correct generalization of the FOOO Kuranishi chart embedding, in the following sense:

\begin{thm}\label{afec}
An FOOO embedding determines an embedding of Kuranishi charts in the sense of Definition \ref{ourcemb}. In particular, the tangent bundle condition yields the quasi-isomorphism condition between the $L_{\infty}[1]$-algebras.
\end{thm}

\begin{proof}
The base component is set to be the smooth embedding $\phi : U \hookrightarrow U'.$ For the $L_{\infty}$-component, we first define a quasi-isomorphic $L_{\infty}[1]$-morphism $\widehat{\eta}_x := \left\{\widehat{\eta}_{x,k}\right\}_{k \geq 1}$ and then take its homotopy inverse $\widehat{\phi}_{x} := \left\{\widehat{\phi}_{x,k}\right\}_{k \geq 1}.$

\subsubsection*{(Preparatory constructions)} To define $\widehat{\eta}_{x},$ we require some preliminary steps. We first choose a projection
\[
\pi : U' \twoheadrightarrow \phi(U)
\]
restricting to the identity map $\mathrm{id}_{\phi(U)}$ on $\phi(U) \subset U',$ whose existence is guaranteed by the embedding property of $\phi.$

Note that the embedding $\widetilde{\phi} : E \hookrightarrow E'$ naturally induces another bundle embedding
\[
\widetilde{\phi} : \bigwedge\nolimits^{-\bullet}E^* \hookrightarrow \bigwedge\nolimits^{-\bullet}E^{'*}.
\]
By abuse of notation, we denote this embedding by the same symbol $\widetilde{\phi}.$

Consider an embedding
\begin{equation}\nonumber
\widetilde{i} : (\phi^{-1})^*T^*U \hookrightarrow T^*({\mathrm{Im}\phi}),
\end{equation}
which, after composing with the inclusion $T^*(\mathrm{Im}\phi) \hookrightarrow T^*U',$ yields a symplectic submanifold
\begin{equation}\nonumber
\widetilde{i}\big((\phi^{-1})^*(T^*U)\big) \subset T^*U'
\end{equation}
with respect to the standard symplectic structures on $T^*U$ and $T^*U'.$ Such an $\widetilde{i}$ always exists, and allows us to define a map of sections,
\begin{equation}\label{wtimap}
\begin{split}
\overline{(\cdot)} : \Gamma\left(\bigwedge\nolimits^{\bullet +1} T^*U\right) &\rightarrow \Gamma\left(\bigwedge\nolimits^{\bullet +1}T^*U'\right),\\
\xi &\mapsto \overline{\xi} := \pi^* \circ \widetilde{i} \circ {(\phi^{-1})}^*(\xi).
\end{split}
\end{equation}
The maps $\phi$, $\widetilde{i},$ and $\pi$ give rise to the following commutative diagram:
\begin{equation}\label{34diag}
\begin{tikzcd}
TT^*U \arrow{r}{((\phi^{-1})^*)_*} \arrow{d} & T(\phi^{-1})^*T^*U \arrow{r}{\widetilde{i}_*} \arrow{d} & TT^*(\mathrm{Im}\phi) \arrow{r}{(\pi^*)_*} \arrow{d} & TT^*U' \arrow{d}\\
T^*U \arrow{r}{(\phi^{-1})^*} \arrow{d} & (\phi^{-1})^*T^*U \arrow{r}{\widetilde{i}} \arrow{d} & T^*(\mathrm{Im}\phi) \arrow{r}{\pi^*} \arrow{d} & T^*U' \arrow{d} \\
U \arrow{r}{\phi} & \mathrm{Im}\phi \arrow{r}{=} & \mathrm{Im}\phi & U' \arrow{l}{\pi},\\
\end{tikzcd}
\end{equation}
where all vertical arrows are given by the bundle projection maps.
Since the top horizontal arrows are bundle maps, we have
\begin{equation}\label{diagaft}
(\pi^*)_* \circ \widetilde{i}_* \circ \left((\phi^{-1})^*\right)_*|_{T^*U} = \pi^* \circ \widetilde{i} \circ (\phi^{-1})^*.
\end{equation}
Observe that (\ref{34diag}) further induces a commutative diagram arising from the corresponding V-algebras. (See Definition \ref{Voronov1} for the definition of V-algebras.)
\begin{equation}\label{34diag2}
\begin{aligned}
&\begin{tikzcd}
\lim\limits_{\longleftarrow}\frac{\Gamma(\bigwedge^{\bullet +1}TT^*U)}{I^n \cdot \Gamma(\bigwedge^{\bullet +1}TT^*U)} \arrow{r}{((\phi^{-1})^*)_*} \arrow{d}{\Pi} & \lim\limits_{\longleftarrow}\frac{\Gamma(\bigwedge^{\bullet +1}T(\phi^{-1})^*T^*U)}{I^n \cdot \Gamma(\bigwedge^{\bullet +1}T(\phi^{-1})^*T^*U)} \arrow{d} \cdots\\
\Gamma\left(\bigwedge^{\bullet + 1}T^*U\right) \arrow{r}{(\phi^{-1})^*} & \Gamma\left(\bigwedge^{\bullet + 1}(\phi^{-1})^*T^*U\right) \cdots
\end{tikzcd}
\\
&\quad \quad \quad \quad \quad \begin{tikzcd}
\cdots \arrow{r}{\widetilde{i}_*} & \lim\limits_{\longleftarrow}\frac{\Gamma(\bigwedge^{\bullet +1}TT^*(\text{Im}\phi))}{I^n \cdot \Gamma(\bigwedge^{\bullet +1}TT^*(\text{Im}\phi))}  \arrow{r}{(\pi^*)_*} \arrow{d} & \lim\limits_{\longleftarrow}\frac{\Gamma(\bigwedge^{\bullet +1} TT^*U')}{I^n \cdot \Gamma(\bigwedge^{\bullet +1} TT^*U')}  \arrow{d}{\Pi' }
\\
\cdots \arrow{r}{\widetilde{i}} & \Gamma\left(\bigwedge^{\bullet + 1}T^*(\text{Im}\phi)\right) \arrow{r}{\pi^*} & \Gamma\left(\bigwedge^{\bullet + 1}T^*U'\right),
\end{tikzcd}
\end{aligned}
\end{equation}
so that
\begin{equation}\nonumber
\Pi' \circ (\pi^*)_* \circ \widetilde{i}_* \circ \left({(\phi^{-1})}^*\right)_* = \pi^* \circ \widetilde{i} \circ {(\phi^{-1})}^* \circ \Pi.
\end{equation}
In the upper-left component, the ideal $I$ consists of functions vanishing on the zero section $T^*U$, that is,
\[
I := \left\{ f \in C^{\infty}(TT^*U) \mid f|_{T^*U} \equiv 0\right\}.
\]
The remaining ideals are defined analogously as the corresponding ideals of 
\[ C^{\infty}\left(\left(\phi^{-1}\right)^*T^*U\right), C^{\infty}\left(T^*\left(\text{Im}\phi\right)\right), \text{ and } C^{\infty}\left(T^*U'\right).
\] 
We denote all of these by the same symbol $I,$ by abuse of notation. The induced maps on the top horizontal line (written with the same notation as in (\ref{34diag})) exist due to the fact that $\left(\phi^{-1}\right)_*, \widetilde{i}_*,$ and $\left(\pi^*\right)_*$ in (\ref{34diag}) are bundle maps.
The two Poisson structures
\begin{equation}
\begin{cases}
P &= \sum\limits_{\alpha} \frac{\partial}{\partial q^{\alpha}} \wedge \frac{\partial}{\partial p_{\alpha}} \in \lim\limits_{\longleftarrow}\frac{\Gamma(\bigwedge^{\bullet +1}TT^*U)}{I^n \cdot \Gamma(\bigwedge^{\bullet +1}TT^*U)},\\
P' &= \sum\limits_{\alpha'} \frac{\partial}{\partial q^{'\alpha'}} \wedge \frac{\partial}{\partial p'_{\alpha'}} \in \lim\limits_{\longleftarrow}\frac{\Gamma(\bigwedge^{\bullet +1} TT^*U')}{I^n \cdot \Gamma(\bigwedge^{\bullet +1} TT^*U')}
\end{cases}
\end{equation}
are induced from the zero presymplectic structures on $U$ and $U'$, as in (\ref{pico}), respectively.

\subsubsection*{({Definition of} $\widehat{\eta}_{x}$)} Using the maps in diagrams (\ref{34diag}) and (\ref{34diag2}), and considering the completion at the image of the embedding $W_x \hookrightarrow W'_{\phi(x)},$ we define
\begin{equation}\nonumber
\begin{split}
\widehat{\eta}_{x,k} &: \left(\Gamma\left(\bigwedge\nolimits^{-\bullet}E^*|_{W_x}\right) \oplus \Gamma \left(\bigwedge\nolimits^{\bullet +1}T^*U|_{W_x}\right)\right)^{\otimes k}\\
& \quad \quad \quad \quad \quad \quad \quad \quad \rightarrow \left(\Gamma\left(\bigwedge\nolimits^{-\bullet}E^{'*}|_{W'_{\phi(x)}}\right)\right)_{\phi} \oplus \left(\Gamma\left(\bigwedge\nolimits^{\bullet +1}T^*U'|_{W'_{\phi(x)}}\right)\right)_{\phi}
\end{split}
\end{equation}
by
\begin{equation}\label{eta1k}
\widehat{\eta}_{x,k}\big((a_1, \xi_1), \cdots, (a_k, \xi_k)\big) := \begin{cases}
\big(1 \otimes (\overline{a_1}, 0), 1 \otimes \overline{\xi_1}\big) &\text{ if } k = 1,\\
(0, 0) &\text{ if } k \geq 2,\\
\end{cases}
\end{equation}
for $a_i \in \Gamma\left(\bigwedge\nolimits^{-\bullet}\left(E^*|_{W_x}\right)\right),$ $\xi_i \in \Gamma\left(\bigwedge\nolimits^{\bullet +1}T^*U|_{W_x}\right), \ i = 1, \cdots, k,$
where we denote
\begin{equation}
\begin{cases}
\overline{a_i} := \pi^* \circ \widetilde{\phi}(a_i)\in \Gamma\left(\bigwedge\nolimits^{-\bullet}E^{'*}|_{W'_{\phi(x)}}\right),\\
\overline{\xi_i} := \pi^* \circ \widetilde{i} \circ \left(\phi^{-1}\right)^*(\xi_i) \in \Gamma\left(\bigwedge\nolimits^{\bullet +1}T^*U'|_{W'_{\phi(x)}}\right).
\end{cases}
\end{equation}
Here, the \textit{pullback by} $\pi$,
\[
\pi^* : \Gamma\left(\bigwedge\nolimits^{-\bullet}E^{'*}|_{\mathrm{Im}\phi}\right) \rightarrow \Gamma\left(\bigwedge\nolimits^{-\bullet}E^{'*}|_{W'_{\phi(x)}}\right)
\]
is defined by
\[
\pi^*(b)(u) := \pi^*\left(b(u|_{\mathrm{Im}\phi})\right)
\]
for $u \in \Gamma\left(\bigwedge\nolimits^{-\bullet}E'|_{W'_{\phi(x)}}\right),$ inductively on the degree of $b \in \Gamma\left(\bigwedge\nolimits^{-\bullet}E^{'*}|_{\mathrm{Im}\phi}\right).$

\begin{lem}\label{pwgjbynn}
$\widehat{\eta}_{x}$ is an $L_{\infty}[1]$-morphism.
\end{lem}

\begin{proof}
We must show
\begin{equation}\nonumber
l'_k \big(\widehat{\eta}_{x,1}(a_1, \xi_1), \cdots, \widehat{\eta}_{x,1}(a_k, \xi_k) \big)
= \widehat{\eta}_{x,1} \big( l_k\big((a_1, \xi_1), \cdots, (a_k, \xi_k)\big) \big).
\end{equation}
If $k = 1,$ we have
\begin{equation}\nonumber
\begin{split}
l'_1\big(\widehat{\eta}_{x,1}(a, \xi)\big) &= \big(1 \otimes l^{'\mathrm{K}}_1(\overline{a}, 0), 1 \otimes l^{'\mathcal{F}}_1(\overline{\xi})\big) = \big(1 \otimes (\iota_{s'}\overline{a},0), 1 \otimes d_{\mathcal{F}'}\overline{\xi}\big)\\
&\overset{*}{=} \big(1 \otimes(\overline{\iota_{s}a},0), 1 \otimes d_{\mathcal{F}'}\overline{\xi}\big) = \big(\big(1 \otimes\overline{l_1^K(a)}), \big(1 \otimes l_1^{\mathcal{F}'}(\overline{\xi})\big)\big) \\
&=\big(1 \otimes\overline{l_1^K(a)}, 1 \otimes \overline{l_1^{\mathcal{F}}(\xi)}\big)  = \widehat{\eta}_{x,1}\big(l_1^K(a), l_1^{\mathcal{F}}(\xi)\big) = \widehat{\eta}_{x,1}\big( l_1 (a, \xi) \big).
\end{split}
\end{equation}
Assume $a$ is homogeneous and write $a = a_1 \wedge \cdots \wedge a_l$ as a product of degree-1 elements. The equality $*$ then follows from the fact that the operations $\iota_s$ and $\iota_{s'}$ respect the restriction maps:
\begin{equation}\label{isp1}
\begin{split}
\iota_{s'}\overline{a} &= \iota_{s'}\left(\pi^*\big(\widetilde{\phi}(a)\big)\right) = \pi^*\big(\widetilde{\phi}(a)\big)(s') = \pi^* \left( \widetilde{\phi}(a)(s'|_{\mathrm{Im}\phi})\right)\\ &= \pi^*\left(\widetilde{\phi}(a_1) \wedge \cdots \wedge \widetilde{\phi}(a_l) (s'|_{\mathrm{Im}\phi})\right)\\ &=\pi^*\left(\sum\limits_i(-1)^{i+1}\widetilde{\phi}(a_i) (s'|_{\mathrm{Im}\phi}) \cdot \widetilde{\phi}(a_1) \wedge \cdots \wedge \widehat{\widetilde{\phi}(a_i)} \wedge \widetilde{\phi}(a_l)\right),
\end{split}
\end{equation}
where we use
\begin{equation}\nonumber
\widetilde{\phi}(a_i)\left(s'|_{\mathrm{Im}\phi}\right) = a_i\left(\widetilde{\phi}^{-1}\left(s'|_{\mathrm{Im}\phi}\right)\right) = a_i(s)
\end{equation}
by the bundle map property of $\widetilde{\phi},$ so that (\ref{isp1}) further equals
\begin{equation}\nonumber
\begin{split}
&= \pi^* \left(\sum_i(-1)^{i+1}a_i(s) \cdot \widetilde{\phi}(a_1) \wedge \widehat{\widetilde{\phi}(a_i)} \wedge \cdots \wedge \widetilde{\phi}(a_l)\right) \\
&= \pi^* \circ \widetilde{\phi} \left(\sum_i(-1)^{i+1}a_i(s) \cdot a_1 \wedge \widehat{a_i} \wedge \cdots \wedge a_l\right) \\
&= \pi^* \circ \widetilde{\phi} (\iota_s a) = \overline{\iota_s a}.
\end{split}
\end{equation}

If $k \geq 2,$ we have, for $a_i \in \Gamma\left(\bigwedge\nolimits^{-\bullet}E^*|_{W_x}\right),$ $\xi_i\in \Gamma(T^*U|_{W_x}),$ and $1 \leq i \leq k,$
\begin{equation}\label{pagjgggshs}
\begin{split}
l'_k \big( \widehat{\eta}_{x,1}(a_1, \xi_1), &\cdots, \widehat{\eta}_{x,1}(a_k, \xi_k) \big) = l'_k \left((1 \otimes (\overline{a_1}, 0), 1 \otimes \overline{\xi_1}), \cdots, (1 \otimes (\overline{a_k}, 0), 1 \otimes \overline{\xi_k})\right)\\
&=\left(l^{\prime \mathrm{K}}_k \left( 1 \otimes (\overline{a_1}, 0), \cdots, 1 \otimes (\overline{a_k},0) \right), l^{\prime \mathcal{F}}_k(1 \otimes \overline{\xi_1}, \cdots, 1\otimes \overline{\xi_k})\right)\\
&=0\\
&= \widehat{\eta}_{x,1}\big(0,0\big) = \widehat{\eta}_{x,1}\left(l^{K}_k(a_1, \cdots, a_k), l^{\mathcal{F}}_k(\xi_1, \cdots, \xi_k)\right)\\
&= \widehat{\eta}_{x,1} \big( l_k \big((a_1, \xi_1), \cdots, (a_k, \xi_k) \big)\big).\\
\end{split}
\end{equation}
In particular, we use (ii) and (iii) of Lemma \ref{liiifdr}, which show that $l^{\mathcal{F}}_{k \geq 2}$ and $l^{\prime \mathcal{F}}_{k \geq 2}$ vanish, so that the resulting $L_{\infty}[1]$-algebra coincides with the usual (augmented) de Rham complex.

\end{proof}

We denote the induced $L_{\infty}[1]$-morphism (still denoted by $\widehat{\eta}_{x}$) from Lemma \ref{auglmo} and Remark \ref{auglmoloc} by
\begin{equation}\label{etxcc}
\widehat{\eta}_x : \mathcal{C}_x^{\otimes k} \rightarrow \mathcal{C}'_{\phi(x), \phi}.
\end{equation}

\begin{lem}\label{etqi}
$\widehat{\eta}_x$ is a quasi-isomorphism.
\end{lem}
\begin{proof}
Since $\widehat{\eta}_{x,1}$ is injective, we consider the short exact sequence:
\begin{equation}\nonumber
0 \rightarrow \mathcal{C}_x \xrightarrow{\widehat{\eta}_{x,1}} \mathcal{C}'_{\phi(x),\phi} \rightarrow \frac{\mathcal{C}'_{\phi(x),\phi}}{\widehat{\eta}_{x,1}(\mathcal{C}_x)} \rightarrow 0.
\end{equation}
To show that $\widehat{\eta}_{x,1}$ is a quasi-isomorphism, it suffices to prove the acyclicity of the quotient chain complex
\begin{equation}\label{epwhe}
\begin{split}
\frac{\mathcal{C}'_{\phi(x),\phi}}{\widehat{\eta}_{x,1}(\mathcal{C}_x)} &=
\frac{\left(\bigwedge^{-\bullet}\Gamma(E^{'*}|_{W'_{\phi(x)}})\right)_{\phi} \times \Omega_{\mathrm{aug},\phi}^{\bullet+1}(U'|_{W'_{\phi(x)}})}{\widehat{\eta}_{x,1}\left(\bigwedge^{-\bullet}\Gamma(E^*|_{W_{x}}) \times \Omega_{\mathrm{aug},\phi}^{\bullet+1}(U|_{W_{x}})\right)} \\
&\simeq \frac{\left(\bigwedge^{-\bullet}\Gamma(E^{'*}|_{W'_{\phi(x)}})\right)_{\phi}}{\widehat{\eta}^{\mathrm{K}}_{x,1}\left(\bigwedge^{-\bullet}\Gamma(E^*|_{W_{x}})\right)} \times \frac{\Omega_{\mathrm{aug},\phi}^{\bullet+1}(U'|_{W'_{\phi(x)}})}{\widehat{\eta}_{x,1}^{\mathrm{dR}}\left(\Omega_{\mathrm{aug},\phi}^{\bullet+1}(U|_{W_{x}})\right)},
\end{split}
\end{equation}
which is further implied by the acyclicity of each component. Here, $\widehat{\eta}^{\mathrm{K}}_{x,1}$ and $\widehat{\eta}_{x,1}^{\mathrm{dR}}$ denote the Koszul and de Rham parts of $\widehat{\eta}_{x,1}$ in (\ref{etxcc}), respectively.

We make a remark on the Koszul part of (\ref{epwhe}). Recall that, by the definition of the pullback of dual sections, the $C^{\infty}\left(W'_{\phi(x)}\right)$-module structure on 
\[
\pi^* \circ \widetilde{\phi}\left(\bigwedge\nolimits^{-\bullet}\Gamma\left(E^*\big|_{W_{x}}\right)\right)
\]
is given, for $h \in C^{\infty}\left(W'_{\phi(x)}\right)$, by
\[
\begin{split}
h \colon \pi^* \circ \widetilde{\phi}\left(\bigwedge\nolimits^{-\bullet}\Gamma\left(E^*\big|_{W_{x}}\right)\right) &\longrightarrow \pi^* \circ \widetilde{\phi}\left(\bigwedge\nolimits^{-\bullet}\Gamma\left(E^*\big|_{W_{x}}\right)\right), \\
\pi^* \circ \widetilde{\phi}(a) &\longmapsto \pi^*\left(h\big|_{\operatorname{Im}\phi}\right) \cdot \pi^* \circ \widetilde{\phi}(a),
\end{split}
\]
whereas the $C^{\infty}\left(W'_{\phi(x)}\right)$-module structure on $\bigwedge\nolimits^{-\bullet}\Gamma\left({E'}^*\big|_{W'_{\phi(x)}}\right)$ is given by standard scalar multiplication. It is straightforward to show that these two module structures are incompatible. Therefore, the inclusion
\begin{equation}\label{b3wgw}
\pi^* \circ \widetilde{\phi}\left(\bigwedge\nolimits^{-\bullet}\Gamma\left(E^*\big|_{W_{x}}\right)\right) \hookrightarrow \bigwedge\nolimits^{-\bullet}\Gamma\left({E'}^*\big|_{W'_{\phi(x)}}\right)
\end{equation}
can be understood solely as an $\mathbb{R}$-linear map.

Similarly, we have an $\mathbb{R}$-linear inclusion
\[
\begin{split}
\pi^* \circ \widetilde{\phi}\left(\bigwedge\nolimits^{-\bullet}\Gamma\left(E^*\big|_{W_{x}}\right)\right) &\longrightarrow C^{\infty}_{\phi}\left(W'_{\phi(x)}\right) \otimes_{C^{\infty}\left(W'_{\phi(x)}\right)} \bigwedge\nolimits^{-\bullet}\Gamma\left({E'}^*\big|_{W'_{\phi(x)}}\right), \\
\pi^* \circ \widetilde{\phi}(a) &\longmapsto 1 \otimes (\text{the image of } a \text{ under } \eqref{b3wgw}),
\end{split}
\]
making it meaningful to consider the quotient $\mathbb{R}$-module
\[
\frac{\left(\bigwedge^{-\bullet}\Gamma(E^{'*}|_{W'_{\phi(x)}})\right)_{\phi}}{\widehat{\eta}^{\mathrm{K}}_{x,1}\left(\bigwedge^{-\bullet}\Gamma(E^*|_{W_{x}})\right)} := \frac{C^{\infty}_{\phi}\left(W'_{\phi(x)}\right) \otimes_{C^{\infty}\left(W'_{\phi(x)}\right)} \bigwedge\nolimits^{-\bullet}\Gamma\left({E'}^*\big|_{W'_{\phi(x)}}\right)}{\pi^* \circ \widetilde{\phi}\left(\bigwedge\nolimits^{-\bullet}\Gamma\left(E^*\big|_{W_{x}}\right)\right)}.
\]

\subsubsection*{(The de Rham part)} The de Rham part $L_{\infty}[1]$-morphism in (\ref{epwhe}) is automatically a quasi-isomorphism, as it is a map between acyclic complexes (namely, augmented de Rham complexes) (cf.\ Corollary \ref{corpoinc}).

\subsubsection*{(The Koszul part)} Observe that there exists a decomposition
\begin{equation}\nonumber
\begin{split}
\bigwedge\nolimits^{-\bullet}\Gamma\left(E^{'*}|_{W'_{\phi(x)}}\right)_{\phi} &= C_{\phi}^{\infty}\left(W'_{\phi(x)}\right) \otimes \bigwedge\nolimits^{-\bullet}\Gamma\left(E^{'*}|_{W'_{\phi(x)}}\right)\\
&\simeq \bigoplus\limits_{p,q}\left(C_{\phi}^{\infty}(W'_{\phi(x)}) \otimes \bigwedge^{p,q}\Gamma(E^{'*}|_{W'_{\phi(x)}})\right),
\end{split}
\end{equation}
where we denote the $(p,q)$-component by
\begin{equation}\label{baohshgg}
\bigwedge^{p,q}\Gamma\left(E^{'*}|_{W'_{\phi(x)}}\right) := \bigwedge\nolimits^{p}\Gamma\left(\pi^*\widetilde{\phi}(E^*|_{W_x})\right) \otimes \bigwedge\nolimits^{q}\Gamma\left(E^c\right).
\end{equation}
Here $E^c$ denotes the vector bundle given by the complement,
\begin{equation}\nonumber
E^{'}|_{W'_{\phi(x)}} \simeq \pi^*\widetilde{\phi}\left(E|_{W_x}\right) \oplus E^c,
\end{equation}
and similarly for the dual bundle,
\begin{equation}\nonumber
E^{'*}|_{W'_{\phi(x)}} \simeq \pi^*\widetilde{\phi}\left(E^*|_{W_x}\right) \oplus E^c.
\end{equation}
By abuse of notation, we write $E^c$ for both cases. The section $s' \in \Gamma\left(E^{'}|_{W'_{\phi(x)}}\right)$ decomposes accordingly as
\begin{equation}\nonumber
s' =: s'_{\phi} \oplus s'_c.
\end{equation}
Let $\mathrm{rk}\,E' = k$, $\mathrm{rk}\,\widetilde{\phi}(E) = m$, and $\mathrm{rk}\,E^c = r$ be the ranks of the vector bundles. We then obtain a double complex by noting that the differential decomposes as
\begin{equation}\label{qgjggjgjg}
\iota_{s'} = \iota_{s'_{\phi}} + (-1)^p\iota_{s'_c}
\end{equation}
when applied to the $(p,q)$-component.

As a consequence, we obtain a double complex
\[
\frac{\bigwedge^{-\bullet}\Gamma\left(E^{'*}|_{W'_{\phi(x)}}\right)_{\phi}}{\widehat{\eta}^{\mathrm{K}}_{x,1}\left(\bigwedge^{-\bullet}\Gamma\left(E^*|_{W_{x}}\right)\right)} = \bigoplus\limits_{p \leq 0, q < 0} \bigwedge^{p,q}\Gamma\left(E^{'*}|_{W'_{\phi(x)}}\right)_{\phi} \times \bigoplus\limits_{p \leq 0}\frac{\bigwedge^{p,0}\Gamma\left(E^{'*}|_{W'_{\phi(x)}}\right)_{\phi}}{\widehat{\eta}^{\mathrm{K}}_{x,1}\big(\bigwedge\nolimits^{p}\Gamma\left(E^*|_{W_{x}}\right)\big)}
\]
illustrated in the following diagram:

\[
\begin{tikzcd}[column sep=0.75cm, row sep=0.25cm]
{} &
0 &
0 &
\cdots &
0 &
{} \\
0 \arrow[r] &
\frac{\Gamma(\wedge^{m,0}(E^{\prime*}|_{W'_{\phi(x)}}))_{\phi}}{\widehat{\eta}^{\mathrm{K}}_{x,1}(\wedge^m \Gamma\left(E^*|_{W_x}\right))}
  \arrow[r,"\iota_{s_{\phi}}"]
  \arrow[u] &
\frac{\Gamma(\wedge^{m-1,0}(E^{\prime*}|_{W'_{\phi(x)}}))_{\phi}}{\widehat{\eta}^{\mathrm{K}}_{x,1}(\wedge^m\Gamma\left(E^*|_{W_x}\right))}
  \arrow[r,"\iota_{s_{\phi}}"] \arrow[u] &
\cdots
  \arrow[r,"\iota_{s_{\phi}}"] &
\overbrace{\frac{\Gamma(\wedge^{0,0}(E^{\prime*}|_{W'_{\phi(x)}}))_{\phi}}{\widehat{\eta}^{\mathrm{K}}_{x,1}(\wedge^m \Gamma\left(E^*|_{W_x}\right))}}^{\text{bideg} = (0,0)}
  \arrow[r] \arrow[u] &
0 \\[2ex]
{} &
\vdots \arrow[u, "\iota_{s_c}"] &
\vdots \arrow[u, "\iota_{s_c}"] &
\cdots &
\vdots \arrow[u, "\iota_{s_c}"] &
{} \\
0 \arrow[r] &
\overbrace{\Gamma(\wedge^{m,r-1}(E^{\prime*}|_{W'_{\phi(x)}}))_{\phi}}^{\text{bideg} = (-m,-r+1)}
  \arrow[r,"\iota_{s_{\phi}}"]
  \arrow[u,"\iota_{s_c}"] &
\Gamma(\wedge^{m-1,r-1}(E^{\prime*}|_{W'_{\phi(x)}}))_{\phi}
  \arrow[r,"\iota_{s_{\phi}}"]
  \arrow[u,"\iota_{s_c}"] &
\cdots
  \arrow[r,"\iota_{s_{\phi}}"] &
\Gamma(\wedge^{0,r-1}(E^{\prime*}|_{W'_{\phi(x)}}))_{\phi}
  \arrow[r]
  \arrow[u,"\iota_{s_c}"] &
0 \\[2ex]
0 \arrow[r] &
\overbrace{\Gamma(\wedge^{m,r}(E^{\prime*}|_{W'_{\phi(x)}}))_{\phi}}^{\text{bideg} = (-m,-r)}
  \arrow[r,"\iota_{s_{\phi}}"]
  \arrow[u,"\iota_{s_c}"] &
\overbrace{\Gamma(\wedge^{m-1,r}(E^{\prime*}|_{W'_{\phi(x)}}))_{\phi}}^{\text{bideg} = (-m+1,-r)}
  \arrow[r,"\iota_{s_{\phi}}"]
  \arrow[u,"\iota_{s_c}"] &
\cdots
  \arrow[r,"\iota_{s_{\phi}}"] &
\Gamma(\wedge^{0,r}(E^{\prime*}|_{W'_{\phi(x)}}))_{\phi} \arrow[r]
  \arrow[u,"\iota_{s_c}"] &
0 \\
{} &
0 \arrow[u] &
0 \arrow[u] &
\cdots &
0 \arrow[u] &
{}
\end{tikzcd}
\]

By a standard argument in homological algebra, the acyclicity of this bounded double complex follows from that of each column and row complex. Thus, for the acyclicity of the Koszul part quotient complex, it suffices to show the acyclicity of each column complex:
\begin{equation}\label{diccpx}
\begin{split}
\mathcal{D}_i : 0 \rightarrow \bigwedge^{-i,-r}\Gamma\left(E^{'*}|_{W'_{\phi(x)}}\right)_{\phi} &\xrightarrow{\iota_{s'_c}} \bigwedge^{-i,-r+1}\Gamma\left(E^{'*}|_{W'_{\phi(x)}}\right)_{\phi} \xrightarrow{\iota_{s'_c}} \cdots\\
&\cdots \xrightarrow{\iota_{s'_c}} \frac{\bigwedge^{-i,0}\Gamma\left(E^{'*}|_{W'_{\phi(x)}}\right)_{\phi}}{\widehat{\eta}^{\mathrm{K}}_{x,1}\left(\bigwedge^{-i}\Gamma\left(E^*|_{W_{x}}\right)\right)} \rightarrow 0, \quad i = 0, \cdots, m.
\end{split}
\end{equation}

The acyclicity is now a consequence of the following lemma and Lemma~\ref{lemab}:

\begin{lem}\label{bjadgohaa}
Under the tangent bundle condition, each column complex $\mathcal{D}_i$ in \eqref{diccpx} is acyclic.
\end{lem}

\begin{proof}
Express $s'_c$ in the orthonormal frame $\{e'_1, \dots, e'_r\}$ from Choice~\ref{3choi} as
\begin{equation}\label{ssce}
s'_c = \sum_{j=1}^r {s'}^j_c e'_j, \quad {s'}^j_c \in C^{\infty}\left(W'_{\phi(x)}\right).
\end{equation}

\begin{claim}
The tuple $\left({s'}^1_c, \dots, {s'}^r_c\right)$ from \eqref{ssce} is a \textit{regular sequence}; that is, for each $1 \leq i \leq r$, the element ${s'}^i_c$ is not a zero-divisor in the quotient ring $\frac{C^{\infty}\left(W'_{\phi(x)}\right)}{\langle {s'}^1_c, \dots, {s'}^{i-1}_c\rangle}$. This property is independent of the choices of $\{e'_j\}$ and $\{ {s'}^j_c \}$.
\end{claim}

\begin{proof}
Suppose ${s'}^i_c$ is a zero-divisor in $\frac{C^{\infty}\left(W'_{\phi(x)}\right)}{\langle {s'}^1_c, \dots, {s'}^{i-1}_c\rangle}$. Then there exists $b_i \in C^{\infty}\left(W'_{\phi(x)}\right)$ such that $b_i {s'}_c^i \in \langle {s'}^1_c, \dots, {s'}^{i-1}_c\rangle$; in other words, it can be expressed as
\begin{equation}\label{zdld}
b_i {s'}_c^i = \sum_{j=1}^{i-1} b_j {s'}_c^j
\end{equation}
for some $b_j \in C^{\infty}\left(W'_{\phi(x)}\right)$ ($j = 1, \dots, i-1$). 

From the full rank condition of the matrix $\left\{ d c_x \left( \frac{\partial}{\partial x_l} \right) \right\}_l$, by shrinking $W'_{\phi(x)}$ if necessary, we can choose local coordinates that extend $\left\{ c_1, \dots, c_r \right\}$ around the origin. Then \eqref{zdld} implies that
\[
b_i\big|_{\{c_1 = \dots = c_{i-1} = 0\}} = 0.
\]
By applying Taylor's theorem with respect to these coordinates (again shrinking $W'_{\phi(x)}$ if necessary), it follows that
\[
b_i \in \langle c_1, \dots, c_{i-1} \rangle,
\]
which contradicts the hypothesis. Since $i$ is arbitrary, $\left(c_1, \dots, c_r\right)$ forms a regular sequence.
\end{proof}

We now show the acyclicity of $\mathcal{D}_0$,
\begin{equation}\nonumber
\mathcal{D}_0 : 0 \rightarrow \overbrace{\left(\bigwedge\nolimits^{r}\Gamma\left(E^{c}\right)\right)_{\phi} \xrightarrow{\iota_{{s'}_c}} \cdots \xrightarrow{\iota_{{s'}_c}} \Gamma\left(E^c\right)_{\phi}}^{\deg < 0} \xrightarrow{\iota_{{s'}_c}} \overbrace{\frac{\Gamma\left(\bigwedge^{0,0}(E^{'*}|_{W'_{\phi(x)}})\right)_{\phi}}{\widehat{\eta}^{\mathrm{K}}_{x,1}\Gamma\left(\bigwedge\nolimits^{0}\left(E^*|_{W_{x}}\right)\right)}}^{\deg = 0} \rightarrow 0.
\end{equation}
The proof for $\mathcal{D}_{i>0}$ is essentially identical, so we omit it.

For the case $\deg < 0$, we first consider
\begin{equation}\label{wes}
\cdots \rightarrow \bigwedge\nolimits^{i+1}\Gamma\left(E^c|_{W'_{\phi(x)}}\right) \xrightarrow{\iota_{{s'}_c}^{(i+1)}} \bigwedge\nolimits^{i}\Gamma\left(E^c|_{W'_{\phi(x)}}\right) \xrightarrow{\iota_{{s'}_c}^{(i)}} \bigwedge\nolimits^{i-1}\Gamma\left(E^c|_{W'_{\phi(x)}}\right) \rightarrow \cdots,
\end{equation}
where $\ker \iota^{(i)}_{{s'}_c} = \mathrm{Im}\,\iota_{{s'}_c}^{(i+1)}$ follows from standard homological algebra under the condition that $s'_c = \langle {s'}^1_c, \cdots, {s'}_c^r \rangle$ is a regular sequence. (See, for example, Chapter 17 of \cite{Eisenbud}.)

Note that $s'_c$, the section of the complement, satisfies ${s'}_{c}^{m} \in I_\phi \setminus I_\phi^2$ for each $m$ (when written in the orthonormal frame of $\Gamma(E^c)$ as $s'_c = \sum_m {s'}^m_{c} e_m$) by the tangent bundle condition and condition (iv) in Lemma \ref{addcond}. By Lemma \ref{lemab}, we draw the same conclusion as (\ref{wes}) for the completed complex:
\begin{equation}\label{wes2}
\cdots \rightarrow \Big(\bigwedge\nolimits^{i+1}\Gamma\left(E^c|_{W'_{\phi(x)}}\right)\Big)_{\phi} \xrightarrow{\iota_{{s'}_c}^{(i+1)}} \Big(\bigwedge\nolimits^{i}\Gamma\left(E^c|_{W'_{\phi(x)}}\right)\Big)_{\phi} \xrightarrow{\iota_{{s'}_c}^{(i)}} \Big(\bigwedge\nolimits^{i-1} \Gamma\left(E^c|_{W'_{\phi(x)}}\right)\Big)_{\phi}
\rightarrow \cdots.
\end{equation}

For the $\deg = 0$ case, we have
\begin{equation}\label{cse}
\begin{split}
\frac{\Gamma\left(\bigwedge^{0,0}\left({E'}^*\big|_{W'_{\phi(x)}}\right)\right)_{\phi}}{\widehat{\eta}^{\mathrm{K}}_{x,1}\left(\Gamma\left(\bigwedge^{0}E^*\big|_{W_{x}}\right)\right)} 
&= \frac{C^{\infty}\left(W'_{\phi(x)}\right)_{\phi}}{\widehat{\eta}^{\mathrm{K}}_{x,1}\big(C^{\infty}(W_x)\big)} \\ 
&\simeq \frac{C^{\infty}_{\phi}\left(W'_{\phi(x)}\right)}{\pi^*\left(C^{\infty}\left(\phi(W_x)\right)\right)} \\
&\overset{*}{\simeq} \left\{ h \in C^{\infty}_{\phi}\left(W'_{\phi(x)}\right) \;\Big|\; h\big|_{\operatorname{Im}\phi} = 0 \right\},
\end{split}
\end{equation}
where $*$ follows from the following lemma:

\begin{lem}\label{lem12}
There exists an $\mathbb{R}$-isomorphism:
\begin{equation}\nonumber
\frac{\bigwedge^{p,0}\Gamma\left(E^{'*}|_{W'_{\phi(x)}}\right)}{\pi^* \circ \widetilde{\phi}\left(\bigwedge^{p}\Gamma(E^*|_{W_{x}})\right)} \xrightarrow{\; \simeq \;} \left\{a' \in \bigwedge\nolimits^{p,0}\Gamma\left(E^{'*}|_{W'_{\phi(x)}}\right) \;\Big|\; a'|_{\mathrm{Im}\,\phi} \equiv 0\right\}.
\end{equation}
In particular, if $p = 0$, then we have
\[
\frac{C^{\infty}\left(W'_{\phi(x)}\right)}{\pi^* \left(C^{\infty}(\phi(W_x))\right)} \xrightarrow{\; \simeq \;}  \left\{ h \in C^{\infty}\left(W'_{\phi(x)}\right) \;\Big|\; h\big|_{\operatorname{Im}\phi} = 0 \right\}.
\]
\end{lem}

\begin{proof}
Consider the $\mathbb{R}$-linear map
\begin{equation}\nonumber
\begin{split}
\kappa : \bigwedge^{p,0}\Gamma\left(E^{'*}|_{W'_{\phi(x)}}\right) &\longrightarrow \left\{a' \in \bigwedge\nolimits^{p,0}\Gamma\left(E^{'*}|_{W'_{\phi(x)}}\right) \;\Big|\; a'|_{\mathrm{Im}\,\phi} \equiv 0\right\},\\
a' &\longmapsto a' - \pi^*\left(a'|_{\mathrm{Im}\,\phi}\right),
\end{split}
\end{equation}
which is well-defined because $\left(a' - \pi^*\left(a'|_{\mathrm{Im}\,\phi}\right)\right)\Big|_{\mathrm{Im}\,\phi} = a'|_{\mathrm{Im}\,\phi} - a'|_{\mathrm{Im}\,\phi} = 0$.
Then $\kappa$ is surjective. Its kernel consists of all elements of the form $\pi^*(b)$ for some section $b$ over $\operatorname{Im}\phi \cap W'_{\phi(x)}$, which can be rewritten as $b = \widetilde{\phi}(b')$ for some $b' \in \bigwedge\nolimits^{p}\Gamma\left(E^*|_{W_{x}}\right)$ by the embedding property of $\phi$ and $\widetilde{\phi}$. That is,
\[
\ker(\kappa) = \pi^* \circ \widetilde{\phi}\left(\bigwedge\nolimits^{p}\Gamma(E^*|_{W_{x}})\right).
\]
\end{proof}

It remains to show the surjectivity of the $\mathbb{R}$-module morphism
\begin{equation}\label{gecwp}
\Gamma\left(E^c\big|_{W'_{\phi(x)}}\right)_{\phi} \xrightarrow{\ \iota_{{s'}_c}\ } \left\{ h \in C^{\infty}_{\phi}\left(W'_{\phi(x)}\right) \;\Big|\; h\big|_{\operatorname{Im}\phi} = 0 \right\},
\end{equation}
given by
\[
h \otimes a \longmapsto h \cdot \iota_{s'_c}(a)
\]
for $h \in C^{\infty}_{\phi}\left(W_{\phi(x)}'\right)$ and $a \in \Gamma(E^c|_{W'_{\phi(x)}})$, after the identification \eqref{cse}. Also, observe that $h \cdot \iota_{s'_c}(a)$ vanishes on $\operatorname{Im}\phi$, as we have $\iota_{{s'}_c}(a)\big|_{\operatorname{Im}\phi} = a(s'_c)\big|_{\operatorname{Im}\phi} = 0$ by condition (iv) in Lemma~\ref{addcond}; hence, \eqref{gecwp} is well-defined. Similarly to the $\deg < 0$ case, we first show the surjectivity of \eqref{gecwp}.

At each $(y_1, \dots, y_m) \in \phi(W_x)$, we consider the restricted section
\[
\underline{s'_c} := s'_c\Big|_{\pi^{-1}(y_1, \dots, y_m)},
\]
that is,
\[
\begin{split}
W'_{\phi(x)} \supset \pi^{-1}(y_1, \dots, y_m) &\longrightarrow \mathbb{R}^r, \\
\underline{s'_c} \colon (w_1, \dots, w_r) &\longmapsto \left(\underline{s'_c}_1(\vec{y}, \vec{w}), \dots, \underline{s'_c}_r(\vec{y}, \vec{w})\right).
\end{split}
\]
By the conditions in Lemma~\ref{addcond} and the inverse function theorem, there exists a local inverse
\begin{equation}\label{spmo}
\underline{s'_c}^{-1} \colon \left(\underline{s'_c}_1, \dots, \underline{s'_c}_r\right) \longmapsto (w_1, \dots, w_r)
\end{equation}
defined on a smaller open neighborhood $\overset{\circ}{W}'_{\phi(x)} \subset W'_{\phi(x)}$. An important technical point here is that we may assume $\overset{\circ}{W}'_{\phi(x)}$ coincides with $W'_{\phi(x)}$, by virtue of Example~\ref{dcotopwx}, without loss of generality.

Note that any $h \in \left\{ h \in C^{\infty}_{\phi}\left(W'_{\phi(x)}\right) \;\Big|\; h\big|_{\operatorname{Im}\phi} = 0 \right\}$ can be written as
\begin{equation}\label{hhy}
h = h(y_1, \dots, y_m)\left[\left[w_1, \dots, w_r\right]\right]
\end{equation}
in the coordinates $(y_1, \dots, y_m)$ of $\phi(W_x)$ and the normal-direction coordinates $(w_1, \dots, w_r)$ in $W'_{\phi(x)}$. Using the local inverse \eqref{spmo}, we can substitute
\[
w_i = \underline{s'_c}^{-1}_i\left(\underline{s'_c}_1(\vec{y}, \vec{w}), \dots, \underline{s'_c}_r(\vec{y}, \vec{w})\right)
\]
into \eqref{hhy} to obtain
\begin{equation}\label{heee}
\begin{split}
h &= h\left(y_1, \dots, y_m\right)\left[\left[\underline{s'_c}_1(\vec{y}, \vec{w}), \dots, \underline{s'_c}_r(\vec{y}, \vec{w})\right]\right] \\
&\overset{(1)}{=} h(y_1, \dots, y_m, \vec{0}) + \sum_{1 \le |\vec{\alpha}| < k} \frac{1}{\vec{\alpha}!} \frac{\partial^{\vec{\alpha}} h\left(\vec{y}, \vec{\underline{s'_c}}\right)}{\partial \underline{s'_c}^{\vec{\alpha}}}\bigg|_{\vec{\underline{s'_c}} = 0} \cdot \underline{s'_c}^{\vec{\alpha}} + \sum_{|\vec{\alpha}| = k} \underline{s'_c}^{\vec{\alpha}} A_{\vec{\alpha}}\left(\vec{y}, \vec{\underline{s'_c}}\right) \\
&\overset{(2)}{=} \sum_{1 \le |\vec{\alpha}| < k} \frac{1}{\vec{\alpha}!} \frac{\partial^{\vec{\alpha}} h\left(\vec{y}, \vec{\underline{s'_c}}\right)}{\partial \underline{s'_c}^{\vec{\alpha}}}\bigg|_{\vec{\underline{s'_c}} = 0} \cdot \underline{s'_c}^{\vec{\alpha}} + \sum_{|\vec{\alpha}| = k} \underline{s'_c}^{\vec{\alpha}} A_{\vec{\alpha}}\left(\vec{y}, \vec{\underline{s'_c}}\right)
\end{split}
\end{equation}
for some formal power series $A_{\vec{\alpha}}\left(\vec{y}, \vec{\underline{s'_c}}\right)$ in $\underline{s'_c}$ with smooth coefficients depending on $\vec{y}$, satisfying $A_{\vec{\alpha}}(\vec{y}, 0) = 0$. Here, equalities $(1)$ and $(2)$ follow from Taylor's formula with integral remainder at order $k$ and the assumption $h\big|_{\operatorname{Im}\phi} = 0$, respectively. Moreover, we use the multi-index notation $\underline{s'_c}^{\vec{\alpha}} := \underline{s'_c}_1^{\alpha_1} \dots \underline{s'_c}_r^{\alpha_r}$ for each multi-index $\vec{\alpha} = (\alpha_1, \dots, \alpha_r)$.

It is then clear that $h \in \operatorname{Im}\,\iota_{\underline{s'_c}}$, so that $\iota_{\underline{s'_c}}$ is surjective: the last line of \eqref{heee} shows that $h$ can be expressed as a linear combination of $\{\underline{s'_c}^{\vec{\alpha}}\}$ over $C^{\infty}\left(W'_{\phi(x)}\right)$. Recalling the expression \eqref{ssce}, we can always choose an element of $\Gamma\left(E^c\big|_{W'_{\phi(x)}}\right)$ written in the basis $\{{e'}^*_1, \dots, {e'}^*_r\}$ (cf. Choice~\ref{3choi}) with appropriate coefficients.

For a multi-index $\vec{\alpha} = (\alpha_1, \dots, \alpha_r)$, we denote
\[
m(\vec{\alpha}) := \min \{ i : \alpha_i \neq 0, \ 1 \leq i \leq r \}
\]
and
\[
\vec{\alpha}' := \vec{\alpha} - (0, \dots, \overbrace{1}^{m(\vec{\alpha})}, \dots, 0).
\]
Then we define
\[
\tau(h) := \sum_{\vec{\alpha}} \left( \frac{1}{\vec{\alpha}!} \frac{\partial^{\vec{\alpha}} h}{\partial \underline{s'_c}^{\vec{\alpha}}}\bigg|_{\vec{\underline{s'_c}} = 0} \cdot \underline{s'_c}^{\vec{\alpha}'} \right) {e'}^*_{m(\vec{\alpha})} + \sum_{|\vec{\alpha}| = k} \underline{s'_c}^{\vec{\alpha}'} A_{\vec{\alpha}}\left(\vec{y}, \vec{\underline{s'_c}}\right) {e'}^*_{m(\vec{\alpha})},
\]
satisfying
\[
h = \iota_{\underline{s'_c}}\left(\tau(h)\right).
\]
Therefore, the map $\iota_{\underline{s'_c}}$ is surjective, thereby establishing Lemma~\ref{bjadgohaa}.
\end{proof}

This finishes the proof of Lemma~\ref{etqi}.
\end{proof}

This concludes the proof of Theorem~\ref{afec}.
\end{proof}

In fact, Theorem~\ref{afec} has a version without completion:

\begin{thm}[The uncompleted case]\label{oqqtittt}
In the setting of Theorem~\ref{afec}, the FOOO embedding also induces a quasi-isomorphism to the target $L_{\infty}[1]$-algebra without completion:
\[
\widehat{\eta}_x \colon \mathcal{C}_x \rightarrow \mathcal{C}'_{\phi(x)}.
\]
\end{thm}

\begin{proof}[Sketch of proof]
The proof is essentially the same as that of Lemma~\ref{etqi}. We define $\widehat{\eta}_x = \{\widehat{\eta}_{x,k}\}_{k \geq 1}$ by
\begin{equation}\label{eta1k}
\widehat{\eta}_{x,k}\big((a_1, \xi_1), \dots, (a_k, \xi_k)\big) :=
\begin{cases}
\left( (\overline{a_1}, 0), \overline{\xi_1} \right) & \text{if } k = 1, \\[1ex]
(0, 0) & \text{if } k \geq 2,
\end{cases}
\end{equation}
for $a_i \in \Gamma\left(\bigwedge\nolimits^{-\bullet}\left(E^*\big|_{W_x}\right)\right)$ and $\xi_i \in \Gamma\left(\bigwedge\nolimits^{\bullet +1}T^*U\big|_{W_x}\right)$ ($i = 1, \dots, k$), where we denote
\[
\begin{cases}
\overline{a_i} := \pi^* \circ \widetilde{\phi}(a_i) \in \Gamma\left(\bigwedge\nolimits^{-\bullet}{E'}^*\big|_{W'_{\phi(x)}}\right), \\
\overline{\xi_i} := \pi^* \circ \widetilde{i} \circ (\phi^{-1})^*(\xi_i) \in \Gamma\left(\bigwedge\nolimits^{\bullet +1}T^*U'\big|_{W'_{\phi(x)}}\right).
\end{cases}
\]

In the same way as in the proof of Lemma~\ref{pwgjbynn}, we can verify that $\widehat{\eta}_x$ is an $L_{\infty}[1]$-morphism. Furthermore, as in Lemma~\ref{etqi}, the de Rham part of the $L_{\infty}[1]$-morphism in \eqref{epwhe} is automatically a quasi-isomorphism, as it is a map between acyclic complexes (namely, augmented de Rham complexes) (cf.\ Corollary~\ref{corpoinc}).

For the Koszul part, we show the acyclicity of the quotient complex
\[
\frac{\mathcal{C}'_{\phi(x),\phi}}{\widehat{\eta}_{x,1}(\mathcal{C}_x)} \simeq \frac{\bigwedge^{-\bullet}\Gamma\left({E'}^*\big|_{W'_{\phi(x)}}\right)}{\widehat{\eta}^{\mathrm{K}}_{x,1}\left(\bigwedge^{-\bullet}\Gamma\left(E^*\big|_{W_{x}}\right)\right)} \times \frac{\Omega_{\mathrm{aug},\phi}^{\bullet+1}\left(U'\big|_{W'_{\phi(x)}}\right)}{\widehat{\eta}_{x,1}^{\mathrm{dR}}\left(\Omega_{\mathrm{aug},\phi}^{\bullet+1}\left(U\big|_{W_{x}}\right)\right)},
\]
and, in particular, the acyclicity of its Koszul part double complex
\[
\frac{\bigwedge^{-\bullet}\Gamma\left({E'}^*\big|_{W'_{\phi(x)}}\right)}{\widehat{\eta}^{\mathrm{K}}_{x,1}\left(\bigwedge^{-\bullet}\Gamma\left(E^*\big|_{W_{x}}\right)\right)} = \bigoplus_{p \leq 0, q < 0} \bigwedge^{p,q}\Gamma\left({E'}^*\big|_{W'_{\phi(x)}}\right) \times \bigoplus_{p \leq 0}\frac{\bigwedge^{p,0}\Gamma\left({E'}^*\big|_{W'_{\phi(x)}}\right)}{\widehat{\eta}^{\mathrm{K}}_{x,1}\left(\bigwedge^{p}\Gamma\left(E^*\big|_{W_{x}}\right)\right)},
\]
with respect to the decompositions \eqref{baohshgg} and \eqref{qgjggjgjg}, while each of its column complexes $\mathcal{D}_i$ is given by
\[
\begin{split}
\mathcal{D}_i \colon 0 \longrightarrow \bigwedge^{-i,-r}\Gamma\left({E'}^*\big|_{W'_{\phi(x)}}\right) &\xrightarrow{\ \iota_{\underline{s'_c}}\ } \bigwedge^{-i,-r+1}\Gamma\left({E'}^*\big|_{W'_{\phi(x)}}\right) \xrightarrow{\ \iota_{\underline{s'_c}}\ } \dots \\
&\dots \xrightarrow{\ \iota_{\underline{s'_c}}\ } \frac{\bigwedge^{-i,0}\Gamma\left({E'}^*\big|_{W'_{\phi(x)}}\right)}{\widehat{\eta}^{\mathrm{K}}_{x,1}\left(\bigwedge^{-i}\Gamma\left(E^*\big|_{W_{x}}\right)\right)} \longrightarrow 0, \quad i = 0, \dots, m.
\end{split}
\]

We show the acyclicity of $\mathcal{D}_0$, namely,
\[
\mathcal{D}_0 \colon 0 \longrightarrow \overbrace{\Gamma\left(\bigwedge^{r}E^c\big|_{W'_{\phi(x)}}\right) \xrightarrow{\ \iota_{\underline{s'_c}}\ } \dots \xrightarrow{\ \iota_{\underline{s'_c}}\ } \Gamma\left(E^c\big|_{W'_{\phi(x)}}\right)}^{\deg < 0} \xrightarrow{\ \iota_{\underline{s'_c}}\ } \overbrace{\frac{\Gamma\left(\bigwedge^{0,0}\left({E'}^*\big|_{W'_{\phi(x)}}\right)\right)}{\widehat{\eta}^{\mathrm{K}}_{x,1}\left(\Gamma\left(\bigwedge^{0}E^*\big|_{W_{x}}\right)\right)}}^{\deg = 0} \longrightarrow 0,
\]
whereas the proof for $\mathcal{D}_{i>0}$ is essentially identical and thus omitted.

For the $\deg < 0$ case, we first consider
\[
\dots \longrightarrow \bigwedge^{i+1}\Gamma\left(E^c\big|_{W'_{\phi(x)}}\right) \xrightarrow{\ \iota_{\underline{s'_c}}^{(i+1)}\ } \bigwedge^{i}\Gamma\left(E^c\big|_{W'_{\phi(x)}}\right) \xrightarrow{\ \iota_{\underline{s'_c}}^{(i)}\ } \bigwedge^{i-1}\Gamma\left(E^c\big|_{W'_{\phi(x)}}\right) \longrightarrow \dots,
\]
where $\ker \iota_{\underline{s'_c}}^{(i)} = \operatorname{Im}\,\iota_{\underline{s'_c}}^{(i+1)}$ follows from standard homological algebra under the condition that $s'_c = \langle {s'}^1_c, \dots, {s'}^r_c \rangle$ is a regular sequence.

For the $\deg = 0$ case, we have natural isomorphisms
\begin{equation}\label{csee}
\begin{split}
\frac{\Gamma\left(\bigwedge^{0,0}\left({E'}^*\big|_{W'_{\phi(x)}}\right)\right)}{\widehat{\eta}^{\mathrm{K}}_{x,1}\left(\Gamma\left(\bigwedge^{0}E^*\big|_{W_{x}}\right)\right)} 
&= \frac{C^{\infty}\left(W'_{\phi(x)}\right)}{\widehat{\eta}^{\mathrm{K}}_{x,1}\big(C^{\infty}(W_x)\big)} \\ 
&\simeq \frac{C^{\infty}\left(W'_{\phi(x)}\right)}{\pi^* \circ (\phi^{-1})^*\left(C^{\infty}(W_x)\right)} \\
&\simeq \frac{C^{\infty}\left(W'_{\phi(x)}\right)}{\pi^*\left(C^{\infty}\left(\phi(W_x)\right)\right)} \\
&\overset{*}{\simeq} \left\{ h \in C^{\infty}\left(W'_{\phi(x)}\right) \;\Big|\; h\big|_{\operatorname{Im}\phi} = 0 \right\},
\end{split}
\end{equation}
where $*$ follows by an argument analogous to Lemma~\ref{lem12}.

Then we consider the $\mathbb{R}$-module morphism
\begin{equation}\label{gecwp}
\Gamma\left(E^c\big|_{W'_{\phi(x)}}\right) \xrightarrow{\ \iota_{\underline{s'_c}}\ } \left\{ h \in C^{\infty}\left(W'_{\phi(x)}\right) \;\Big|\; h\big|_{\operatorname{Im}\phi} = 0 \right\},
\end{equation}
given by
\[
a \longmapsto \iota_{\underline{s'_c}}(a)
\]
after the identification \eqref{csee}. Also, observe that $\iota_{\underline{s'_c}}(a)$ vanishes on $\operatorname{Im}\phi$, as we have $\iota_{\underline{s'_c}}(a)\big|_{\operatorname{Im}\phi} = a(\underline{s'_c})\big|_{\operatorname{Im}\phi} = 0$ by condition (iv) in Lemma~\ref{addcond}; hence, \eqref{gecwp} is well-defined.

The remainder of the proof is devoted to showing the surjectivity of \eqref{gecwp}, which follows the same logic as in Lemma~\ref{etqi}, but relies on Taylor's theorem for smooth functions. Specifically, at each $(y_1, \dots, y_m) \in \phi(W_x)$, we consider the restricted section
\[
\underline{s'_c} := s'_c\Big|_{\pi^{-1}(y_1, \dots, y_m)},
\]
and its local inverse
\begin{equation}\label{spmo}
\underline{s'_c}^{-1} \colon \left(\underline{s'_c}_1, \dots, \underline{s'_c}_r\right) \longmapsto (w_1, \dots, w_r)
\end{equation}
defined on a smaller open neighborhood $\overset{\circ}{W}'_{\phi(x)} \subset W'_{\phi(x)}$, which can be taken to be equal to $W'_{\phi(x)}$ by virtue of Example~\ref{dcotopwx}, without loss of generality.

Since any $h \in \left\{ h \in C^{\infty}\left(W'_{\phi(x)}\right) \;\Big|\; h\big|_{\operatorname{Im}\phi} = 0 \right\}$ can be written in terms of the coordinates $(y_1, \dots, y_m)$ of $\phi(W_x)$ and the normal-direction coordinates $(w_1, \dots, w_r)$ in $W'_{\phi(x)}$ as
\begin{equation}\label{hhy}
h = h(y_1, \dots, y_m, w_1, \dots, w_r),
\end{equation}
we can substitute
\[
w_i = \underline{s'_c}^{-1}_i\left(\underline{s'_c}_1(\vec{y}, \vec{w}), \dots, \underline{s'_c}_r(\vec{y}, \vec{w})\right)
\]
into \eqref{hhy} to obtain
\begin{equation}\label{heee}
\begin{split}
h &= h\left(y_1, \dots, y_m, \underline{s'_c}_1(\vec{y}, \vec{w}), \dots, \underline{s'_c}_r(\vec{y}, \vec{w})\right) \\
&\overset{(1)}{=} h(y_1, \dots, y_m, \vec{0}) + \sum_{1 \le |\vec{\alpha}| < k} \frac{1}{\vec{\alpha}!} \frac{\partial^{\vec{\alpha}} h}{\partial \underline{s'_c}^{\vec{\alpha}}}\bigg|_{\vec{\underline{s'_c}} = 0} \cdot \underline{s'_c}^{\vec{\alpha}} + \sum_{|\vec{\alpha}| = k} \underline{s'_c}^{\vec{\alpha}} A_{\vec{\alpha}}\left(\vec{y}, \vec{\underline{s'_c}}\right) \\
&\overset{(2)}{=} \sum_{1 \le |\vec{\alpha}| < k} \frac{1}{\vec{\alpha}!} \frac{\partial^{\vec{\alpha}} h}{\partial \underline{s'_c}^{\vec{\alpha}}}\bigg|_{\vec{\underline{s'_c}} = 0} \cdot \underline{s'_c}^{\vec{\alpha}} + \sum_{|\vec{\alpha}| = k} \underline{s'_c}^{\vec{\alpha}} A_{\vec{\alpha}}\left(\vec{y}, \vec{\underline{s'_c}}\right)
\end{split}
\end{equation}
for some smooth coefficient functions $A_{\vec{\alpha}}\left(\vec{y}, \vec{\underline{s'_c}}\right)$ satisfying $A_{\vec{\alpha}}(\vec{y}, 0) = 0$. Here, equalities $(1)$ and $(2)$ follow from Taylor's formula with integral remainder at order $k$ and the assumption $h\big|_{\operatorname{Im}\phi} = 0$, respectively.

It is then clear that $h \in \operatorname{Im}\,\iota_{\underline{s'_c}}$, so that $\iota_{\underline{s'_c}}$ is surjective: the last line of \eqref{heee} shows that $h$ can be expressed as a linear combination of $\{\underline{s'_c}^{\vec{\alpha}}\}$ over $C^{\infty}\left(W'_{\phi(x)}\right)$, analogously to the proof of Theorem~\ref{afec}.
\end{proof}

We now provide the proofs of Lemmata \ref{itavsteai} and \ref{phxv}.

\begin{proof}[Proof of Lemma \ref{itavsteai}]
For the Koszul part $\widehat{o}^{\mathrm{K}}_x$, we can regard the open embedding of the subchart as an FOOO embedding where the bundle embedding is given by the identity map on the fibers. The tangent bundle condition holds trivially, as it reduces to an isomorphism between zero vector spaces. Then Theorem~ref{afec} implies that the quasi-isomorphism condition holds.

For the de Rham part $\widehat{o}^{\mathrm{dR}}_x$, we consider a linear map
\[
\widehat{o}^{\mathrm{dR}}_{x,k} : \Omega^{\bullet + 1}\left(\mathcal{F}'_{o(x),k}\right)^{\otimes k}_o \rightarrow  \Omega^{\bullet + 1}\left(\mathcal{F}_x\right),
\]
defined by
\[
\widehat{o}^{\mathrm{dR}}_{x,k} \left( h_1 \otimes \xi'_1, \ldots, h_k \otimes \xi'_k \right) =
\begin{cases}
o^* \widetilde{h}_1 \cdot o^* \xi'_1 & k = 1, \\
0 & k \geq 2,
\end{cases}
\]
for $h_1 \in C^{\infty}\left(W'_{o(x)}\right)^{(j)}$ and its representative $\widetilde{h}_1 \in C^{\infty}\left(W'_{o(x)}\right).$ Note that this is well-defined; for a different choice of representative $\widetilde{h}'_1,$ we have $\widetilde{h}_1 - \widetilde{h}'_1 \in I^j_{o},$ and hence $o^*\left(\widetilde{h}_1 - \widetilde{h}'_1\right) = 0.$

To show the compatibility with respect to the choices of $j,$ consider $h \in C^\infty\left(W_x\right)^{(j)}$ and $h' \in C^\infty\left(W_x\right)^{(j+1)}$ satisfying $h = p_{j+1,j}\left(h'\right) = h' + I_o^j / I_o^{j+1},$ together with their representatives $\widetilde{h}$ and $\widetilde{h}' \in C^\infty\left(W_x\right)$ with $\widetilde{h} + I_o^{j} = h$ and $\widetilde{h}' + I_o^{j+1} = h',$ respectively. Then we have $\widetilde{h} = \widetilde{h}' + \widetilde{g}$ for some $\widetilde{g} \in I_o^{j}.$

The compatibility can now be verified as follows:
\[
\widehat{o}_{x,1}\left(h_1 \otimes \xi_1'\right) 
= o^* \widetilde{h}_1 \cdot o^* \xi_1' 
= o^*\left(\widetilde{h}' + \widetilde{g}\right) \cdot o^* \xi_1' 
= o^* \widetilde{h}_1' \cdot o^* \xi_1' 
= \widehat{o}_{x,1}\left(h_1' \otimes \xi_1'\right).
\]

We claim that $\widehat{o}_x$ is an $L_\infty[1]$-morphism. Indeed, we have
\begin{equation}\nonumber
\begin{split}
\widehat{o}_{x,1}&\left(l_k\left(h_1 \otimes \xi'_1, \cdots, h_k \otimes \xi'_k\right)\right) = \widehat{o}_{x,1}\left([1] \otimes l_k\left(\widetilde{h}_1 \xi'_1, \cdots, \widetilde{h}_k \xi_k\right)\right)\\
&= o^*1 \cdot o^{*}\left(l_k\left(\widetilde{h}_1 \xi'_1, \cdots, \widetilde{h}_k \xi'_k\right)\right) \stackrel{*}= l_k\left(\left(o^{*}\widetilde{h}_1\right) \cdot o^* \xi'_1, \cdots, \left(o^* \widetilde{h}_k\right) \cdot o^{*} \xi'_k\right)\\ 
&= l_k\left(\widehat{o}_{x,1}\left(h_1 \otimes \xi'_1\right), \cdots, \widehat{o}_{x,1}\left(h_k \otimes \xi'_k\right)\right).
\end{split}
\end{equation}
Here, the equality $*$ follows from the fact that the Nijenhuis-Schouten bracket respects restrictions to smaller \textit{open} subsets.

We now define our de Rham part $L_{\infty}[1]$-morphism,
\[
\widehat{o}_x : \Omega^{\bullet + 1}_{\mathrm{aug},o}\left(\mathcal{F}'_{o(x)}\right) \rightarrow  \Omega^{\bullet + 1}_{\mathrm{aug}}\left(\mathcal{F}_x\right),
\]
still denoted by $\widehat{o}_x,$ to be the induced morphism between the augmented $L_{\infty}[1]$-algebras of Lemma \ref{auglmo}. This is necessarily a quasi-isomorphism since it is a map between acyclic $L_{\infty}[1]$-algebras.
\end{proof}

\begin{proof}[Proof of Lemma \ref{phxv}] 
Since $\widehat{\pi}^{-1}_{(x,0)} \circ \widehat{\pi}_{(x,0)} = 
\mathrm{id}_{\mathcal{C}_{x}}$, it suffices to show that
$\widehat{\pi}_{(x,0)}$ is a quasi-isomorphism. Moreover, Corollary
\ref{anhp} implies that
$\widehat{\pi}_{(x,0)} \circ \widehat{\pi}^{-1}_{(x,0)}$ and
$\mathrm{id}_{\mathcal{C}_{(x,0)}}$ are homotopic to each other.

We first recall that we have $\widehat\pi_{(x,0)} = \widehat\pi^{\mathrm{c}}_{(x,0)}$ by the surjectivity of $\pi.$ In fact, it is straightforward to verify that $\widehat\pi^{\mathrm{c}}_{(x,0)}$ respects the $L_\infty[1]$-operations:
\[
\begin{split}
l_k^V&\left(\widehat\pi^{\mathrm{c}}_{(x,0),1}\left(a_1,\xi_1\right),\cdots,\widehat\pi^{\mathrm{c}}_{(x,0),1}\left(a_k,\xi_k\right)\right) = l^V_k\left(\left(\pi^*a_1,\pi^*\xi_1\right),\cdots,\left(\pi^*a_k,\pi^*\xi_k\right)\right)\\
& = l_k\left(\left(a_1,\xi_1\right),\cdots,\left(a_k,\xi_k\right)\right) = l^{\mathrm{K}}_k\left(a_1,\cdots, a_k \right) \oplus l^{\mathrm{dR}}_k\left(\xi_1, \cdots, \xi_k\right)\\
&= \widehat\pi^{\mathrm{c}}_{(x,0),1}\left(l_k^{\mathrm{K}}\left(a_1,\cdots,a_k\right),l_k^{\mathrm{dR}}\left(\xi_1,\cdots,\xi_k\right)\right) =\widehat\pi^{\mathrm{c}}_{(x,0),1}\left(l_k\left(\left(a_1,\xi_1\right),\cdots,\left(a_k,\xi_k\right)\right)\right).
\end{split}
\]

Since $\widehat{\pi}_{(x,0),1}$ is injective, it suffices to show that the quotient complex
\[
\frac{\mathcal{C}_{(x,0)}^V}{\widehat{\pi}^{\mathrm{c}}_{(x,0),1}(\mathcal{C}_x)}
\]
is acyclic by the quasi-isomorphism property of $\widehat{\pi}^{\mathrm{c}}_{(x,0)}.$ The proof is essentially identical to that of Theorem~\ref{afec}, once we observe that we obtain an FOOO embedding
\[
I: \mathcal{U} \hookrightarrow \mathcal{U} \times V; \quad I = \left(i, \widetilde{i}\right),
\]
where 
\[
i: U \hookrightarrow U \times V; \quad y \mapsto (y, 0)
\]
is the natural inclusion and 
\[
\widetilde{i}: E \hookrightarrow E \times V; \quad (y, a) \mapsto \big(\big((y,0), a\big), 0\big)
\]
is the natural bundle embedding. We remark that it satisfies the tangent bundle condition; that is, we have an isomorphism
\[
\left(V \simeq\right) \frac{T_{(x,0)}(U \times V)}{i_*(T_x U \times \{0\})} \xrightarrow{\; \simeq \; } \frac{(E \times V)|_{(x,0)}}{\widetilde{i}(E_x)} \quad (\simeq V)
\]
induced from $ds^V \simeq ds \oplus \mathrm{id}_V.$ Thus, we can apply Theorem \ref{oqqtittt}.
\end{proof}

\begin{lem}\label{cochqim}
Suppose that the base $U$ is contractible. For an embedding $\Phi : \mathcal U \rightarrow \mathcal U'$ induced from an $\mathrm{FOOO}$ embedding (cf. Theorem~\ref{afec}) with contractible bases $U$ and $U',$ the $L_{\infty}[1]$-morphism
\[
\widehat{\varepsilon}_{\phi(x),\phi} : \mathcal{C}'_{\phi({x})} \rightarrow \mathcal{C}'_{\phi(x),\phi}
\]
for each ${x} \in s^{'-1}(0)$ is a quasi-isomorphism.
\end{lem}

\begin{proof}
We first note that for $m := \dim U' - \dim U,$ there exists a bundle embedding pair $\left(\widetilde{\phi}, \overline{\widetilde{\phi}}\right)$ with the following properties:
\begin{enumerate}[label = (\roman*)]
\item $\widetilde{\phi} : U \times \mathbb{R}^m \rightarrow U'$ is an open embedding diffeomorphism,
\item $\widetilde{\phi}|_{U \times \{0\}} \equiv \phi,$
\item $\overline{\widetilde{\phi}} : E \oplus \mathbb{R}^m|_{U} \hookrightarrow E'$ is a bundle embedding,
\end{enumerate}
which is guaranteed by the contractibility of $U$ and $U'.$ Observe that the pair $\left(\widetilde{\phi}, \overline{\widetilde{\phi}}\right)$ determines an FOOO embedding satisfying the tangent bundle condition:
\[
\left[d_{({x},0)}(s \times \mathrm{id}_{\mathbb{R}^m})\right] : \frac{T_{\widetilde{\phi}({x},0)}U}{\widetilde{\phi}_{*}\left(T_{({x},0)}\left(U \times \mathbb{R}^m\right)\right)} \xrightarrow{\; \simeq \;} \frac{E'_{\widetilde{\phi}({x},0)}}{\overline{\widetilde{\phi}}\left(E_{x} \times \mathbb{R}^m\right)}
\]
for each $x \in s^{-1}(0).$ Then Theorem~\ref{afec} implies the quasi-isomorphism property of $\widehat{\widetilde{\phi}}^{\mathrm{c}}_{x}$ for each $x.$

By the construction of $\widehat{\pi}_{(x,0)}^{\mathrm{c}}, \widehat{\phi}^{\mathrm{c}}_{x},$ and $\widehat{\widetilde{\phi}}^{\mathrm{c}}_{x},$ it is straightforward to show that we have the commutative (up to $L_{\infty}[1]$-homotopy) diagram
\begin{equation}
\begin{tikzcd}
\mathcal{C}_{({x},0)}^{\mathbb{R}^m} \arrow{r}{\widehat{\pi}^{\mathrm{c},\mathrm{K}}_{(x,0)}, \ \simeq} \arrow{d}[swap]{\widehat{\widetilde{\phi}}^{\mathrm{c},{\mathrm{K}},-1}_{x}, \ \simeq} & \mathcal{C}_{x} \arrow{r}{\widehat{\phi}_{x}^{\mathrm{c},{\mathrm{K}},-1}, \ \simeq} & \mathcal{C}_{\phi(x),\phi} \\
\mathcal{C}_{\widetilde{\phi}(x,0), \widetilde{\phi}} \arrow{r}{=} & \mathcal{C}_{\widetilde{\phi}(x,0)} \arrow{r}{=} & \mathcal{C}_{\phi(x)} \arrow{u}[swap]{\widehat{\varepsilon}^{\mathrm{K}}_{\phi(x),\phi}}
\end{tikzcd}
\end{equation}
consisting solely of Koszul part morphisms. Here, $\widehat{\widetilde{\phi}}^{\mathrm{c},{\mathrm{K}},-1}_{x}$ is a quasi-isomorphism by Theorem~\ref{oqqtittt}. Consequently, since all $L_{\infty}[1]$-morphisms in the diagram other than $\widehat{\varepsilon}^{\mathrm{K}}_{\phi(x),\phi}$ are quasi-isomorphisms, so is $\widehat{\varepsilon}^{\mathrm{K}}_{\phi(x),\phi}$.

Finally, the de Rham part morphism is also a quasi-isomorphism, as it is an $L_{\infty}[1]$-morphism between acyclic algebras.
\end{proof}

\begin{cor}\label{crtbc}
For an embedding of Kuranishi charts with the tangent bundle condition, the $L_{\infty}$-component $\widehat{\phi}_x = \widehat{\phi}^{\mathrm{c}}_x \circ \widehat{\varepsilon}_{\phi(x), \phi} $ for each $x \in s^{-1}(0)$ is a quasi-isomorphism.
\end{cor}

\section{The category of $L_\infty$-Kuranishi spaces}

In this section, we introduce the notion of \emph{$L_\infty$-Kuranishi spaces}. An advantage of working with this $L_\infty$-version is that the set of $L_\infty$-Kuranishi spaces forms a category. In this regard, we define the morphisms between them, which are essentially given by a collection of compatible chart morphisms.

\subsection{$L_{\infty}$-Kuranishi atlases}

We can cover the underlying topological space with $L_{\infty}$-Kuranishi charts, provided that they satisfy certain compatibility conditions. Before presenting the definition of a Kuranishi space, we first introduce the notion of an $L_{\infty}$\textit{-Kuranishi atlas.} Here, the term \textit{atlas} should not be confused with its usage in other contexts in the literature, such as in \cite{MW1}.

\begin{defn}[$L_\infty$-Kuranishi atlases]
\label{kstr}
Let $X$ be a compact metrizable space. We say that $\widehat{\mathcal{U}}$ is a \textit{Kuranishi atlas on} $X$ if for each $p \in X,$ there exists a neighborhood $V_p$ of $p$ in $X,$ a Kuranishi chart $\widehat{\mathcal{U}}_p = (U_p, E_p, s_p, \Gamma_p, \psi_p)$  and \textit{contractible} $U_p$ for each $p$, a homeomorphism $\psi_p : s_p^{-1}(0)/\Gamma_p \simeq V_p,$ and if $V_p \cap V_p \neq \emptyset,$ we require that there exist an open subchart $\mathcal{U}_{pq}$ of $\mathcal{U}_p$ and an embedding of charts
\[
\Phi_{pq} = (\phi_{pq}, \widehat{\phi}_{pq}) : \mathcal{U}_p|_{U_{pq}} \hookrightarrow \mathcal{U}_{q}
\]
over $\mathrm{id}_X : X \rightarrow X,$ called \textit{coordinate changes} with the following properties:
\begin{enumerate}[label = (\roman*)]
\item $\Phi_{pp} = \mathrm{id}_{\mathcal{U}_p},$
\item $\psi_q \circ \phi_{pq} = \psi_p$ on $s_{p}^{-1}(0) \cap U_{pq},$
\item $\phi_{qr} \circ \phi_{pq} = \phi_{pr}$ on $\phi_{pq}^{-1}(U_{qr}) \cap U_{pr},$
\item $\psi_p\left(s_p^{-1}(0) \cap U_{pq}\right) =\mathrm{Im} \psi_p \cap\mathrm{Im} \psi_q.$
\end{enumerate}
In this situation, we call $\widehat{\mathcal{U}} = \left(\left\{\mathcal{U}_p\right\}, \left\{\Phi_{pq}\right\}\right)$ a \textit{Kuranishi atlas on} $X$ and $\left\{\Phi_{pq}\right\}_{p,q}$ its \textit{coordinate changes}. We assume that our Kuranishi atlas $\left(X, \widehat{\mathcal{U}}\right)$ satisfies $\max\limits_{p \in X}U_p < \infty$ together with the compactness of $X.$
\end{defn}

\begin{rem}\label{whnli}
\begin{enumerate}[label = (\roman*)]
\item We adopt a convention $\Phi_{pq}$ for the coordinate change (from $\mathcal{U}_p$ to $\mathcal{U}_q$) that differs from FOOO's $\Phi_{qp}$, as it appears to be more convenient for our purpose of developing data with index numbers greater than two.
\item Compare this definition with Definition \ref{fo3ks}, where the coordinate changes are defined for pairs $(p,q)$ with $p \in \mathrm{Im}\,\psi_q$. Our version is more symmetrical in this regard.
\item The cocycle condition for the $L_{\infty}$-component is provided in \cite{Kim3} under the title of higher cocycle conditions. The reason it is not explicitly given in Definition \ref{kstr} is that it can always be achieved once some choices of higher homotopy data are made (cf.\ \cite[Definition 5.4 and Theorem 5.5]{Kim3}).
\end{enumerate}
\end{rem}

\begin{exam}[{Smooth manifolds}]\label{man}
Manifolds are Kuranishi spaces endowed with a Kuranishi atlas 
\[
\widehat{\mathcal{U}}^{\mathrm{man}} = \left(\left\{\mathcal{U}^{\mathrm{man}}_p\right\}, \left\{\Phi_{pq}\right\}\right) = \left(\left\{\left(U_p, E_p, s_p, \Gamma_p, \psi_p\right)\right\}, \left\{\left(U_{pq}, \phi_{pq}, \left\{\widehat{\phi}_{pq,x}\right\}\right)\right\}\right)
\] 
of the following restrictive type:
\begin{itemize}
\item[--] $U_p = (U_p, \beta)$ is the pair of a Euclidean space $\mathbb{R}^n$ of fixed dimension $n$ for all $p$ and the zero form $\beta = 0.$ Here, the isotropy group $\Gamma_x$ is trivial at each $x \in U_p.$
\item[--] $E_p = U_p \times \{0\} \simeq U_p$ is the zero-rank vector bundle.
\item[--] $s_p : U_p \xrightarrow{\simeq} E_p$ is the zero section.
\item[--] $\Gamma_p$ is the trivial group action.
\item[--] $\psi_p : s^{-1}_p(0) \simeq U_p \hookrightarrow \mathbb{R}^n$ is the manifold coordinate chart.
\item[--] $x \in W_x \subset U_p$ is an open ball $\simeq B^n.$
\item[--] $T\mathcal{F}_x = TU_p|_{W_x}$ is the total tangent bundle,
\item[--] $\mathcal{C}_{p,x} := \overset{\text{Koszul}}{C^{\infty}(W_x)} \oplus \overset{\text{de Rham}}{\Omega_{\mathrm{aug}}^{\bullet +1}(W_x)}$ is given by the direct sum of the one component trivial chain complex $C^{\infty}(W_x)$ and the augmented de Rham chain complex $\Omega_{\mathrm{aug}}^{\bullet +1}(W_x)$ with the $L_{\infty}[1]$-algebra $\{l^{\mathrm{man}}_k\}_{k \geq 1}$ with $l^{\mathrm{man}}_{k \geq 2} = 0$ (see Lemma \ref{liiifdr} (ii)). In other words, $\mathcal{C}_{p,x}$ is only a chain complex whose cohomolgy is isomorphic to $C^{\infty}(W_x).$

Let $\mathcal{U}_p$ and $\mathcal{U}_q$ be Kuranishi charts at $p$ and $q,$ respectively. The coordinate change $\Phi_{pq} := \left(U_{pq}, \phi_{pq}, \left\{\widehat{\phi}_{pq,x}\right\}\right) : \mathcal{U}_p \rightarrow \mathcal{U}_q$ is given by:
\begin{enumerate}
\item[--] $U_{pq} := \psi_p^{-1}\left(\mathrm{Im} \psi_p \cap \mathrm{Im} \psi_q\right).$
\item[--] $\phi_{pq} : U_{pq} \rightarrow U_q$ is the usual coordinate change for manifolds
\begin{equation}\nonumber
\phi_{pq} := \psi_q^{-1} \circ \psi_p\big|_{U_{pq}},
\end{equation}
which is an open embedding.
\item[--] $\widehat{\phi}_{pq, x}^{\mathrm{c}}: \left(\mathcal{C}'_{\phi_{pq}(x)}\right)_{\phi_{pq}} \rightarrow \mathcal{C}_x$ at each $x \in s^{-1}_p(0) \cap U_{pq} = U_{pq}$ is an isomorphism constructed as follows with $\widehat{\phi}_{pq, x} = \widehat{\phi}_{pq, x}^{\mathrm{c}} \circ \widehat{\varepsilon}_{\phi_{pq}(x), \phi_{pq}}.$

\noindent
(\textit{Construction of }$\widehat{\phi}^{\mathrm{c}}_{pq, x}$) Since $\phi_{pq}$ is an \textit{open} topological embedding, we can apply Lemma \ref{itavsteai}. As a consequence, we obtain a \textit{chain} isomorphism
\[
\widehat{\phi}_{pq, x}^{\mathrm{c}} := \widehat{\phi}^{\mathrm{K}, \mathrm{c}}_{pq,x} \oplus \widehat{\phi}^{\mathrm{dR}, \mathrm{c}}_{pq,x},
\]
where each component is given by
\[
\begin{split}
\widehat{\phi}^{\mathrm{K}, \mathrm{c}}_{pq,x} :  C^{\infty}\left(W'_{\phi_{pq}(x)}\right)_{\phi_{pq}} &\rightarrow C^{\infty}\left(W_x\right),\\
\widehat{\phi}^{\mathrm{dR}, \mathrm{c}}_{pq,x} : \Omega_{\mathrm{aug}}^{\bullet + 1}\left(W'_{\phi_{pq}(x)}\right)_{\phi_{pq}} &\rightarrow \Omega_{\mathrm{aug}}^{\bullet +1}(W_x).
\end{split}
\]
We first define the de Rham part $\widehat{\phi}^{\mathrm{dR}, \mathrm{c}}_{pq,x}$ by considering, for each $j \geq 1,$
\begin{equation}\label{pqiso}
\begin{split}
(j) : C^{\infty}\left(W'_{\phi_{pq}(x)}\right)^{(j)} \otimes &\Omega^{\bullet + 1}\left(W'_{\phi_{pq}(x)}\right) \rightarrow \Omega^{\bullet +1}(W_x)\\
h \otimes \xi &\mapsto \phi_{pq}^*\left(\widetilde{h}|_{\mathrm{Im}\phi_{pq}}\right) \cdot \phi_{pq}^*\left(\xi|_{\mathrm{Im}\phi_{pq}}\right),
\end{split}
\end{equation}
where $\widetilde{h} \in C^{\infty}\left(W'_{\phi_{pq}(x)}\right)$ is a representative of $h.$ It is easy to see that this map is well-defined. Its inverse is given by
\begin{equation}\label{obobw}
\begin{split}
\Omega^{\bullet +1}(W_x) &\rightarrow \Omega^{\bullet + 1}\left(W'_{\phi_{pq}(x)}\right)_{\phi_{pq}}\\
\xi &\mapsto 1 \otimes \pi_{pq}^*(\phi_{pq}^{-1})^*\xi.
\end{split}
\end{equation}
We observe that the compositions of the above two maps are given by, for each $j \geq 1,$
\[
\begin{split}
(j) : h \otimes \xi &\overset{(\ref{pqiso})}{\mapsto} \phi_{pq}^*\left(\widetilde{h}|_{\mathrm{Im}\phi_{pq}}\right) \cdot \phi_{pq}^*\xi|_{\mathrm{Im}\phi_{pq}}\\ &\overset{(\ref{obobw})}{\mapsto} [1]_j \otimes \pi_{pq}^*(\phi_{pq}^{-1})^*\left( \phi_{pq}^*\left(\widetilde{h}|_{\mathrm{Im}\phi_{pq}}\right) \cdot \phi_{pq}^*(\xi|_{\mathrm{Im}\phi_{pq}}) \right)\\ & \quad = [1]_j \otimes \widetilde{h}\xi = [\widetilde{h}]_j \otimes \xi =  h \otimes \xi,
\end{split}
\]
and
\[
\begin{split}
\xi &\overset{(\ref{obobw})}{\mapsto} 1 \otimes \pi_{pq}^*(\phi_{pq}^{-1})^*\xi\\ 
&\overset{(\ref{pqiso})}{\mapsto} \phi_{pq}^*\left( \pi_{pq}^*(\phi_{pq}^{-1})^*\xi|_{\mathrm{Im}\phi_{pq}} \right) =  \phi_{pq}^*(\phi_{pq}^{-1})^*\xi = \xi.
\end{split}
\]
The chain map properties are verified as follows. For each $j \geq 1,$ we have
\[
\begin{split}
d(h \otimes \xi) = [1]_{j-1} \otimes d\left(\widetilde{h}\xi\right) \overset{(\ref{pqiso})}{\mapsto} &\phi_{pq}^*\left( d\left(\widetilde{h}\xi\right)|_{\mathrm{Im}\phi_{pq}}\right) \overset{*}{=} d\left(\phi_{pq}^* \left(\widetilde{h}\xi\right)|_{\mathrm{Im}\phi_{pq}}\right)\\ &= d\left(\phi_{pq}^* \left(\widetilde{h}|_{\mathrm{Im}\phi_{pq}}\right) \cdot \phi_{pq}^*\left(\xi|_{\mathrm{Im}\phi_{pq}}\right)\right),
\end{split}
\]
where $*$ is a consequence of the fact that $\phi_{pq}$ is an open embedding. For the opposite direction, we have
\[
d \xi \overset{(\ref{obobw})}{\mapsto} 1 \otimes \pi_{pq}^*\left(\phi_{pq}^{-1}\right)^*d\xi = d\left(1 \otimes \pi_{pq}^*\left(\phi_{pq}^{-1}\right)^*\xi\right).
\]
Furthermore, Lemma \ref{auglmo} and Remark \ref{auglmoloc} lead to its augmented version, which is obviously an isomorphism, and we denote by:
\begin{equation}\nonumber
\widehat{\phi}_{pq,x}^{\mathrm{dR},\mathrm{c}} : \mathcal{C}_{q,\phi_{pq}(x),\phi_{pq}} = \Omega_{\mathrm{aug}, \phi_{pq}}^{\bullet+1}\left(W'_{\phi_{pq}(x)}\right) \xrightarrow{\simeq} \Omega_{\mathrm{aug}}^{\bullet+1}(W_x) = \mathcal{C}_{p,x}.
\end{equation}
Finally, note that the Koszul part $\widehat{\phi}^{\mathrm{K}, \mathrm{c}}_{pq,x}$ can be taken the same as ${\bullet  = -1}$ component of (\ref{pqiso}), which is an isomorphism by the above discussion.
\end{enumerate}
\end{itemize}
\end{exam}

\begin{exam}[{Smooth manifolds with closed two-forms}]
Let $(M,\beta)$ be a smooth manifold equipped with a closed $2$-form $\beta$. When understood as a Kuranishi space, it can be described by a collection of local charts
\[
\left\{(U_p,\beta|_{U_p})\right\}_{p\in M}
\]
and the coordinate changes among them.

More precisely, for each point $p\in M$ we set up the data 
\[
\mathcal{U}_p = (U_p, E_p, s_p, \Gamma_p, \psi_p),
\]
where
\begin{itemize}[label=\textbullet]
  \item[--] $U_p \subset M$ is an open neighborhood of $p$ equipped with the restriction $\beta|_{U_p}$.
  \item[--] $E_p = U_p\times \{0\}\cong U_p$ is the zero-rank vector bundle over $U_p$.
  \item[--] $s_p : U_p \to E_p$ is the zero section.
  \item[--] $\Gamma_p$ is the trivial group action.
  \item[--] $\psi_p : s_p^{-1}(0) = U_p \hookrightarrow M$ is the obvious embedding.
\item[--] $x \in W_x \subset U_p$ is an open ball $\simeq B^n$ assigned to each $x \in s_p^{-1}(0) = U_p.$
\item[--] $T\mathcal{F}_x = \ker(\beta|_{W_x})$ is the foliation tangent bundle.
\item[--] $\mathcal{C}_{p,x} := \overset{\text{Koszul}}{C^{\infty}(W_x)} \oplus \overset{\text{de Rham}}{\Omega_{\mathrm{aug}}^{\bullet +1}(\mathcal{F}_{p,x})}$ is the direct sum of the one component trivial chain complex $C^{\infty}(W_x)$ and the augmented foliation de Rham complex $\Omega_{\mathrm{aug}}^{\bullet +1}(\mathcal{F}_{p,x})$ equipped with the $L_{\infty}[1]$-algebra structure in Example \ref{gtyemb}. Observe that the cohomology of $\mathcal{C}_{p,x}$ is isomorphic to $C^{\infty}(W_x).$
\end{itemize}

Let $\mathcal{U}_p$ and $\mathcal{U}_q$ be Kuranishi charts at $p$ and $q,$ respectively. The coordinate change $\Phi_{pq} := \left(U_{pq}, \phi_{pq}, \left\{\widehat{\phi}_{pq,x}\right\}\right) : \mathcal{U}_p \rightarrow \mathcal{U}_q$ is given by:

\begin{equation}\nonumber
\Phi_{pq} : \mathcal{U}_p \to \mathcal{U}_q
\end{equation}
is given by
\begin{itemize}
\item[--] $U_{pq} \subset U_p$ is an open subset given by
\begin{equation}\nonumber
    U_{pq} := \psi_p^{-1}(\psi_p(U_p) \cap \psi_q(U_q)).
\end{equation}
\item[--] $\phi_{pq} : U_{pq} \to U_q$ is the usual coordinate change for manifolds,
\begin{equation}\nonumber
\phi_{pq} := \psi_q^{-1} \circ \psi_p|_{U_{pq}},
\end{equation}
which is an open embedding.
\item[--] $\widehat{\phi}_{pq,x} :  \mathcal{C}_{q,\phi_{pq}(x),\phi_{pq}} \rightarrow  \mathcal{C}_{p,x}$ at each $x \in s^{-1}_p(0) \cap U_{pq} = U_{pq}$ is an $L_{\infty}[1]$-isomorphism, given as follows with $\widehat{\phi}_{pq, x} = \widehat{\phi}_{pq, x}^{\mathrm{c}} \circ \widehat{\varepsilon}_{\phi_{pq}(x), \phi_{pq}}.$ 

\noindent
(\textit{Construction of }$\widehat{\phi}_{pq, x}^{\mathrm{c}}$) Since $\beta|_{U_{pq}} = \phi_{pq}^*(\beta|_{U_q})$ and $\dim U_p = \dim U_q$, we have
\begin{equation}\nonumber
T\mathcal{F}_{p,x} \simeq \phi_{pq}^*T\mathcal{F}_{q,\phi_{pq}(x)},
\end{equation}
which amounts to identifying open subchart data in the setting of Lemma \ref{itavsteai}. Note that the $L_{\infty}[1]$-algebra depends on the choice of splitting $T\mathcal{U}_p|_{W_x} = T\mathcal{F}_{p,x} \oplus G_{p,x}$; however, it only makes isomorphic differences by Lemma \ref{liiifdr} (iv). Then by Lemma \ref{itavsteai}, we obtain a \textit{chain} isomorphism
\begin{equation}\nonumber
\Omega^{\bullet+1}(\mathcal{F}_{\phi_{pq}(x)})_{\phi_{pq}} \xrightarrow{\simeq} \Omega^{\bullet+1}(\mathcal{F}_x),
\end{equation}
similarly to the manifold case (cf. Example \ref{man}). Furthermore, Lemma \ref{auglmo} and Remark \ref{auglmoloc} yield its augmented version, which is clearly an isomorphism, denoted by
\[
\widehat{\phi}_{pq,x}^{\mathrm{dR}, \mathrm{c}} :  \Omega_{\mathrm{aug}, \phi_{pq}}^{\bullet+1}(\mathcal{F}_{\phi_{pq}(x)}) \xrightarrow{\simeq} \Omega_{\mathrm{aug}}^{\bullet+1}(\mathcal{F}_x).
\]
Then with the same Koszul part $\widehat{\phi}^{\mathrm{K}, \mathrm{c}}_{pq,x} :  C^{\infty}\left(W'_{\phi_{pq}(x)}\right)_{\phi_{pq}} \rightarrow C^{\infty}\left(W_x\right)$ as the maifolds case, we obtain the chain isomorphism $\widehat{\phi}_{pq,x}^{\mathrm{c}} :=  \widehat{\phi}_{pq,x}^{\mathrm{K}, \mathrm{c}} \oplus \widehat{\phi}_{pq,x}^{\mathrm{dR}, \mathrm{c}}.$
\end{itemize}
We leave it as an exercise for the reader to verify that the above data satisfy all the axioms in Definitions \ref{kurdef} and \ref{kstr}. In other words, $(M,\beta)$ determines a Kuranishi atlas in our sense, whose special cases include smooth manifolds (cf.\ Example \ref{man}) and symplectic manifolds (with nondegenerate closed two-forms).
\end{exam}

\subsection{$L_{\infty}$-Kuranishi spaces}
Kuranishi atlases are not suitable for our purpose of achieving categorical structures. Instead, we propose a more well-behaved notion, which we call \textit{Kuranishi spaces,} defined by allowing some ambiguity in the choices of local charts.

\begin{defn}[Expanded atlases]

Given a Kuranishi atlas $\widehat{\mathcal{U}}$ and a nonnegative number $m,$ we define the \textit{expanded atlas} of $\widehat{\mathcal{U}}$ by
\[
\widehat{\mathcal{U}} \times \mathbb{R}^{{m}}
:= 
\Big(
\left\{\mathcal{U}_p \times \mathbb{R}^{{m}} \right\}_p,\,
\left\{ \left(U_{pq} \times \mathbb{R}^{{m}}, \phi_{pq}^{\mathbb{R}^{{m}}}, \left\{\phi_{pq,x}^{\mathbb{R}^{{m}}} \right\}\right)\right\}_{p,q}
\Big),
\]
where each component is given by:
\begin{itemize}
    \item[--] $\mathcal{U}_p \times \mathbb{R}^{{m}}$ is the expanded chart for each $p \in X$.

    \item[--] $U_{pq} \times \mathbb{R}^{{m}}$ is an open subset of $U_p \times \mathbb{R}^{{m}}$.
    
    \item[--] $\phi_{pq}^{\mathbb{R}^{{m}}} : U_{pq} \times \mathbb{R}^{{m}} \to U_q \times \mathbb{R}^{{m}}$ is the base coordinate change given by
    \[
    \phi_{pq}^{\mathbb{R}^{{m}}} := \phi_{pq} \times \mathrm{id}_{\mathbb{R}^{{m}}}.
    \]
   \item[--] $\widehat{\phi}_{pq,x}^{\mathbb{R}^{m}} : \mathcal{C}_{q,\,(\phi_{pq}(x),0)}^{\mathbb{R}^{m}}
\to \mathcal{C}_{p,(x,0)}$, for each $x \in s_{p}^{-1}(0)$, is the $L_{\infty}[1]$-coordinate change given by the composition
\[
\mathcal{C}^{\mathbb{R}^{m}}_{q,\,(\phi_{pq}(x),0)}
\xrightarrow{\widehat{\pi}_{q, (\phi_{pq}(x),0)}^{-1}, \simeq}
\mathcal{C}_{q,\,\phi_{pq}(x)}
\xrightarrow{\widehat{\phi}_{pq,x}, \simeq}
\mathcal{C}_{p,x} \xrightarrow{\widehat{\pi}_{p, (x,0)}, \simeq} \mathcal{C}^{\mathbb{R}^{m}}_{p, (x,0)},
\]
where the $L_{\infty}[1]$-quasi-isomorphisms $\widehat{\pi}_{q, (\phi_{pq}(x),0)}^{-1}$ and $\widehat{\pi}_{p, (x,0)}$ are as defined in Example \ref{cmfae} and Lemma \ref{phxv}. We obtain a decomposition
$\widehat{\phi}_{pq,x}^{\mathbb{R}^{m}} = \widehat{\phi}_{pq,x}^{\mathbb{R}^{m}, \mathrm{c}} \circ \widehat{\varepsilon}_{(\phi_{pq}(x),0), \phi^{\mathbb{R}^m}_{pq}}$ for an $L_{\infty}[1]$-morphism $\widehat{\phi}_{pq,x}^{\mathbb{R}^{m},\mathrm{c}} : \mathcal{C}_{q,\,(\phi_{pq}(x),0),\phi_{pq}^{\mathbb{R}^{m}}}^{\mathbb{R}^{m}}
\to \mathcal{C}_{p,(x,0)}$, from the fact that $\widehat{\pi}_{q, (\phi_{pq}(x),0)}^{-1}$ trivially induces a quasi-isomorphism between the completions
\[
\mathcal{C}^{\mathbb{R}^{m}}_{q,\,(\phi_{pq}(x),0), \phi_{pq}^{\mathbb{R}^{m}}} \rightarrow \mathcal{C}_{q,\,\phi_{pq}(x),\phi_{pq}}
\]
that is compatible with $\widehat{\pi}_{q, (\phi_{pq}(x),0)}^{-1}$, $\widehat{\varepsilon}_{(\phi_{pq}(x),0), \phi^{\mathbb{R}^m}_{pq}}$, and $\widehat{\varepsilon}_{\phi_{pq}(x), \phi_{pq}}$.
\end{itemize}
\end{defn}
We remark that expansion preserves the dimension:
\[
\dim \left(\widehat{\mathcal{U}} \times \mathbb{R}^m\right) = (\dim U + m) - (\text{rk}E + m) = \dim U - \text{rk}E =  \dim \widehat{\mathcal{U}}.
\]
\begin{notation}
Let $\left(X, \widehat{\mathcal{U}}\right)$ be a Kuranishi atlas. We write 
\[
\left(X, \widehat{\mathcal{U}}^0\right) < \left(X, \widehat{\mathcal{U}}\right), \text{ or simply } \widehat{\mathcal{U}}^0 <\widehat{\mathcal{U}}, 
\]
for its open subatlas $\left(X, \widehat{\mathcal{U}}^0\right).$
\end{notation}
With this notation, we define an equivalence relation between the atlases.
\begin{defn}[Equivalence of atlases]\label{eqats}
Let $\left(X, \widehat{\mathcal{U}}_1\right)$ and $\left(X, \widehat{\mathcal{U}}_2\right)$ be Kuranishi atlases. We say that they are \textit{equivalent} and write
\[
\left(X, \widehat{\mathcal{U}}_1\right) \sim \left(X, \widehat{\mathcal{U}}_2\right), \text{ or simply } \widehat{\mathcal{U}}_1 \sim \widehat{\mathcal{U}}_2
\]
if 
\begin{equation}\label{urur}
\widehat{\mathcal{U}}^0_1 \times \mathbb{R}^{m_1} = \widehat{\mathcal{U}}^0_2 \times \mathbb{R}^{m_2}
\end{equation}
for some $m_1, m_2 \geq 0$ by which we mean that the following conditions hold:

\begin{enumerate}[label = (\roman*)]
    \item There exists a commutative diagram as follows
    \[
    \begin{tikzcd}[row sep=large, column sep=large]
    E^0_{1,p}|_{U_{1,p}^0 \times \mathbb{R}^{m_1}} \arrow[r, "\simeq"] 
    & E^0_{2,p}|_{U_{2,p}^0 \times \mathbb{R}^{m_2}}\\
    U_{1,p}^0 \times \mathbb{R}^{m_1} \arrow[r, "\simeq"] \arrow[u, "s_{1,p}^0 \times \mathrm{id}_{\mathbb{R}^{m_1}}"'] 
    & U_{2,p}^0 \times \mathbb{R}^{m_2} \arrow[u, "s_{2,p}^0 \times \mathrm{id}_{\mathbb{R}^{m_2}}"]\\
   (s_{1,p}^0 \times \mathrm{id}_{\mathbb{R}^{m_1}})^{-1}(0) \arrow{r}{(1), \simeq} \arrow[u, hook] 
& (s_{2,p}^0 \times \mathrm{id}_{\mathbb{R}^{m_2}})^{-1}(0). \arrow[u, hook]
    \end{tikzcd}
    \]

    \item There exists a group isomorphism $\Gamma^0_{1,p} \simeq \Gamma^0_{2,p}$.
    
    \item There exists a commutative diagram as follows
\[
    \begin{tikzcd}[row sep=large, column sep=large]
     \frac{(s_{1,p}^0)^{-1}(0)}{\Gamma_{1,p}}
    \arrow[r, "\simeq"] \arrow[d, swap, "\simeq"] 
    & \frac{(s_{1,p}^0 \times \mathrm{id}_{\mathbb{R}^{m_1}})^{-1}(0)}{\Gamma_{1,p}} \arrow[r, hook, "\psi_{1,p}^0"] & X  \\
   \frac{(s_{2,p}^0)^{-1}(0)}{\Gamma_{2,p}} \arrow[r, "\simeq"] 
    &  \frac{(s_{2,p}^0 \times \mathrm{id}_{\mathbb{R}^{m_2}})^{-1}(0)}{\Gamma_{2,p}}. \arrow[ru, hook, "\psi_{2,p}^0"] &{}
    \end{tikzcd}
\]
    \item For each pair of the zero points $x_1 \overset{(1), \simeq}{\leftrightarrow} x_2$ in (i), there exists an \textit{isomorphism}
    \begin{equation}\label{bnhh}
    \mathcal{C}_{{1,p}, (x_1, 0)}^{0,\mathbb{R}^{m_1}} \xrightarrow{\cong} \mathcal{C}_{{2,p}, (x_2, 0)}^{0,\mathbb{R}^{m_2}}.
    \end{equation}
    
    \item There exists a commutative diagram as follows
    \[
    \begin{tikzcd}[row sep=large, column sep=large]
   U_{1,p}^0 \times \mathbb{R}^{m_1} \arrow[r, "\simeq"] 
    & U_{2,p}^0 \times \mathbb{R}^{m_2}\\
    U_{1,pq}^0 \times \mathbb{R}^{m_1} 
    \arrow{r}{(2), \simeq} \arrow[u, hook] 
    & U_{2,pq}^0 \times \mathbb{R}^{m_2}. \arrow[u, hook]
    \end{tikzcd}
    \]
    
    \item We have 
    \[
    \phi_{1,pq}^{0, \mathbb{R}^{m_1}} = \phi_{2,pq}^{0, \mathbb{R}^{m_2}}
    \]
     modulo the diffeomorphism $(2), \simeq$.
    
\item We have 
\[
\begin{split}
\widehat{\phi}_{1,pq,x}^{0, \mathbb{R}^{n_1}} = \widehat{\phi}_{2,pq,x}^{0, \mathbb{R}^{n_2}}
\end{split}
\] 
for each $x \in s_{p}^{-1}(0)$ modulo the isomorphism $(\ref{bnhh})$ in (iv).

\end{enumerate}
\end{defn}

We list some of the properties of the above-mentioned equivalences by the following lemma.
\begin{lem}\label{gggadbbbb}
We have:
\begin{enumerate}[label=(\roman*)]
\item $\widehat{\mathcal{U}}^0 \sim \widehat{\mathcal{U}}$ for an open subatlas $\widehat{\mathcal{U}}^0 < \widehat{\mathcal{U}}.$
\item $\widehat{\mathcal{U}} \sim \widehat{\mathcal{U}} \times V$ for a finite dimensional vector space $V.$
\item $\sim$ is an equivalence relation.
\item $\left(\widehat{\mathcal{U}} \times \mathbb{R}^{m}\right) \times \mathbb{R}^{m'} = \widehat{\mathcal{U}} \times \mathbb{R}^{m + m'}$ for all $m, m' \geq 0.$
\end{enumerate}
\end{lem}

\begin{proof}
\begin{enumerate}[label=(\roman*)]
\item One can take $\widehat{\mathcal{U}}$ itself as the open subatlas and $m_1 = m_2 = 0$ in (\ref{urur}).
\item After identifying $V$ with $\mathbb{R}^m$ for some $m \geq 0,$ one can take $\widehat{\mathcal{U}}^0_1 = \widehat{\mathcal{U}},$ $ \widehat{\mathcal{U}}_{2}^0 = \widehat{\mathcal{U}} \times V,$ and $m_1= m, \ m_2 = 0$ in (\ref{urur}).
\item Symmetry and reflexivity hold trivially. For transitivity, suppose that we are given
\[
\left(X, \widehat{\mathcal{U}}_1\right) \sim \left(X, \widehat{\mathcal{U}}_2\right), \ \left(X, \widehat{\mathcal{U}}'_2\right) \sim \left(X, \widehat{\mathcal{U}}_3\right)
\]
with 
\[
\left[\widehat{\mathcal{U}}_2\right] = \left[\widehat{\mathcal{U}}'_2\right]
\]
and
\[
\widehat{\mathcal{U}}^0_1 \times \mathbb{R}^{n_1} = \widehat{\mathcal{U}}^{0}_2 \times \mathbb{R}^{n_2}, \ \widehat{\mathcal{U}}^{'0}_2 \times \mathbb{R}^{m_2} = \widehat{\mathcal{U}}^0_3 \times \mathbb{R}^{m_3},
\]
respectively, for open subatlases $\widehat{\mathcal{U}}^0_1< \widehat{\mathcal{U}}_1, \ \widehat{\mathcal{U}}^0_2, \widehat{\mathcal{U}}^{'0}_2 < \widehat{\mathcal{U}}_2,$ and $\widehat{\mathcal{U}}^0_3 < \widehat{\mathcal{U}}_3$. Taking a common subatlas of $\widehat{\mathcal{U}}^{0}_2 $ and $\widehat{\mathcal{U}}^{'0}_2$ and multiplying $\widehat{\mathcal{U}}_1,$ $\widehat{\mathcal{U}}_2$ by the same $\mathbb{R}^m$ for sufficiently large $m$ (and $\widehat{\mathcal{U}}_1,$ $\widehat{\mathcal{U}}_2$ by $\mathbb{R}^{m'}$ for some $m'$) will suffice.
\item $\widehat{\mathcal{U}} \times \mathbb{R}^{m} =  \widehat{\mathcal{U}} \times \mathbb{R}^{m + m'}$ is a simple exercise, and we can apply (ii) for $V = \mathbb{R}^{m'}.$
\end{enumerate}
\end{proof}

With these preparations, we are now ready to define $L_{\infty}$-Kuranishi spaces.
\begin{defn}[$L_{\infty}$-Kuranishi spaces]
We call an equivalence class of the equivalence relation $\sim$ an $L_{\infty}$\textit{-Kuranishi space.} Given a Kuranishi atlas $\left(X, \widehat{\mathcal{U}}\right)$, we write
\[
\mathfrak{X} = \left(X, \left[\widehat{\mathcal{U}}\right]\right)
\]
for the Kuranishi space determined by $\widehat{\mathcal{U}}$.
\end{defn}

\subsection{Morphisms of $L_{\infty}$-Kuranishi spaces}
Our discussion can be formulated in categorical terms. In this subsection, we define morphisms between Kuranishi spaces, beginning with the definition of pre-morphisms.

\begin{defn}[Pre-morphism]\label{prem}
\label{morphismkur}
Let $\mathfrak{X} = \left(X, \left[\widehat{\mathcal{U}}\right]\right)$ and $\mathfrak{X}' = \left(X', \left[\widehat{\mathcal{U}}'\right]\right)$ be two Kuranishi spaces. Consider a tuple
\begin{equation}\label{tuplemor}
\overline{F} = \left(\widehat{\mathcal{U}}, \widehat{\mathcal{U}'}, f, \left\{f_{p}\right\}, \left\{\widehat{f}_{p,x}\right\}\right) \\
\end{equation}
that consists of:

\begin{enumerate}
\item Kuranishi atlases on $X$ and $X'$ 
\[
\begin{cases}
\widehat{\mathcal{U}} = \left(\{\widehat{\mathcal{U}}_p\}, \left\{\Phi_{pq}\right\}\right) = \left(\left\{\left(U_p, E_p, s_p, \Gamma_p, \psi_p\right)\right\}, \left\{\left(U_{pq},\phi_{pq},\left\{\widehat{\phi}_{pq,x}\right\}\right)\right\}\right),\\
\widehat{\mathcal{U}}' = \left(\{\widehat{\mathcal{U}}'_{p'}\}, \left\{\Phi'_{p'q'}\right\}\right) = \left(\left\{\left(U'_{p'}, E'_{p'}, s'_{p'}, \Gamma'_{p'}, \psi'_{p'}\right)\right\}, \left\{\left(U'_{p'q'},\phi'_{p'q'},\left\{\widehat{\phi}'_{p'q',x'}\right\}\right)\right\}\right)
\end{cases}
\]
such that $\widehat{\mathcal{U}} \sim \widehat{\mathcal{U}}'.$
\item $f: X \to X',$ a continuous map.
\item $\left(\left\{f_p\right\} ,\left\{\widehat{f}_{p,x}\right\}_{x \in s^{-1}_p(0)}\right) : \widehat{\mathcal{U}} \rightarrow \widehat{\mathcal{U}}'$ for each $p \in X,$ a collection of morphisms of charts. 
\end{enumerate}
We call it a \textit{pre-morphism} if the following compatibilities hold:
For $p,q \in X$ with Im$\psi_p \cap\mathrm{Im} \psi_q \neq \emptyset,$
\begin{enumerate}[label = (\roman*)]
\item $\phi'_{f(p)f(q)} \circ f_p = f_q \circ \phi_{pq}$ \text{on the set of zero points} $s_{p}^{-1}(0) \cap U_{pq},$
\item $\widehat{\phi}_{pq,x} \circ \widehat{f}_{q, \phi_{pq}(x)} = \widehat{f}_{p,x} \circ \widehat{\phi}'_{f(p)f(q),f_p(x)}$
for each $x \in s_p^{-1}(0) \cap U_{pq},$ up to $L_{\infty}[1]$-homotopy.
\end{enumerate}
\end{defn}

\begin{rem}\label{exttild}
\begin{enumerate}
\item The definition of morphisms of charts implies that we have $f_p\big({s^{0}_p}^{-1}(0)\big) \subset s_{f(p)}^{'-1}(0).$
\item (iii) reduces to
$\widehat{f}_{p,x}\circ \widehat{\phi}'_{f(p)f(q), f_p(x)} = \widehat{\phi}_{pq,x} \circ \widehat{f}_{q, \phi_{pq}(x)}$
up to $L_{\infty}[1]$-homotopy when the base maps $\phi_{pq}, \phi'_{f(p)f(q)}, f_p$, and $f_q$ all happen to be surjective, in which case the $\widehat{\varepsilon}_{(\cdots)}$'s are $L_{\infty}[1]$-isomorphisms (in fact, identities).
\end{enumerate}
\end{rem}

We now consider a pair of pre-morphisms from $\mathfrak{X} = \left(X, \left[\widehat{\mathcal{U}}\right]\right)$ to $\mathfrak{X}' = \left(X', \left[\widehat{\mathcal{U}}'\right]\right)$
\begin{equation}\label{f1f2}
\begin{cases}
\overline{F}_1 = \left(\widehat{\mathcal{U}}_1, \widehat{\mathcal{U}'}_1, f_1, \left\{f_{1,p}\right\}, \left\{\widehat{f}_{1,p,x}\right\}\right), \\
\overline{F}_2 = \left(\widehat{\mathcal{U}}_2, \widehat{\mathcal{U}'}_2, f_2, \left\{f_{2,p}\right\}, \left\{\widehat{f}_{2,p,x}\right\}\right)
\end{cases}
\end{equation}
with the properties:
\[
\left[\widehat{\mathcal{U}}_1\right] = \left[\widehat{\mathcal{U}}_2\right] = \left[\widehat{\mathcal{U}}\right]; \ \ \left[\widehat{\mathcal{U}}'_1\right] = \left[\widehat{\mathcal{U}}'_2\right] = \left[\widehat{\mathcal{U}}'\right].
\]

Note that $\overline{F}_1$ and $\overline{F}_2$ can be extended to
\begin{equation}\label{extfi}
\begin{cases}
\overline{F}^{n_1,n'_1}_1 = \left(\widehat{\mathcal{U}}^0_1 \times \mathbb{R}^{n_1}, \widehat{\mathcal{U}'}_1 \times \mathbb{R}^{n'_1}, f_1, \left\{\widetilde{f}_{1,p}\right\}, \left\{\widetilde{\widehat{f}}_{1,p,(x,0)}\right\}\right) \\
\overline{F}^{n_2,n'_2}_2 = \left(\widehat{\mathcal{U}}^0_2 \times \mathbb{R}^{n_2}, \widehat{\mathcal{U}'}_2 \times \mathbb{R}^{n'_2}, f_2, \left\{\widetilde{f}_{2,p}\right\}, \left\{\widetilde{\widehat{f}}_{2,p,(x,0)}\right\}\right),
\end{cases}
\end{equation}
with following properties:
\begin{enumerate}[label=(\roman*)]
\item $\widehat{\mathcal{U}}^0_1 \times \mathbb{R}^{n_1} = \widehat{\mathcal{U}}^0_2 \times \mathbb{R}^{n_2}$,
\item $n_i \geq n'_i$, $i = 1,2$,
\item $\widetilde{f}_{i,p} : U^{0}_{i,p} \times \mathbb{R}^{n_i} \rightarrow U^{'0}_{i,p} \times \mathbb{R}^{n'_i}$, $i = 1,2$, is a \textit{surjective} map extending $f_{i,p}$, that is, $\widetilde{f}_{i,p}|_{U_1^0 \times \{0\}} \equiv f_{i,p}$ (cf.\ the assumption $\max\limits_{p \in X}U_p < \infty$ in Definition \ref{kstr}). In particular, $\widetilde{f}_{i,p}|_{s^{-1}_p(0) \times \{0\}} \equiv f_{i,p}|_{s_p^{-1}(0)}$.

\item $\widetilde{\widehat{f}}_{i,p,x}$, $i = 1,2$, is the $L_{\infty}[1]$-morphism given by
\begin{equation}
\begin{split}
\widetilde{\widehat{f}}&_{i, p, x} :=
\widehat{\pi}_{i,p,(x,0)} \circ \widehat{f}_{i,p,x} \circ \widehat{\pi}_{i, f(p), (f_{i,p}(x),0)}^{-1}\\
&:{} \mathcal{C}^{'\mathbb{R}^{n_i'}}_{f(p),(f_{i,p}(x),0)} \xrightarrow{\widehat{\pi}_{i, f(p), (f_{i,p}(x),0)}^{-1}} \mathcal{C}'_{f(p), f_{i,p}(x)} \xrightarrow{\widehat{f}_{i,p,x}} \mathcal{C}_{p, x} \xrightarrow{\widehat{\pi}_{i, p, (x,0)}} \mathcal{C}^{\mathbb{R}^{n_i}}_{p, (x,0)},
\end{split}
\end{equation}
where $\widehat{\pi}_{i, f(p), (f_{i,p}(x),0)}^{-1}$ and $\widehat{\pi}_{i, p, (x,0)}$, $i = 1,2$, are the $L_{\infty}[1]$-quasi-isomorphisms defined in Example \ref{cmfae} and Lemma \ref{phxv}.
\end{enumerate}

We remark that having $\widetilde{f}_{i,p}$ as in condition (ii) is always possible for each $i$, by the contractibility of the base $U_i^0$.

\begin{defn}[Equivalence of pre-morphisms]\label{defmoreq}
Without loss of generality, one can assume that $n_1 \geq n_2$ in (\ref{urur}). We say that two pre-morphisms are \textit{equivalent} and write
\begin{equation}\label{morrel}
\overline{F}_1 \sim \overline{F}_2
\end{equation}
if there exist extensions $\overline{F}^{n_1,n'_1}_1$ and $\overline{F}^{n_2,n'_2}_2$ as in (\ref{extfi}) such that the following hold:
\begin{enumerate}[label = (\roman*)]
\item $f_1 = f_2,$
\item $\widetilde{f}_{1, p}|_{V_1}  =  \widetilde{f}_{2, p}|_{V_2}$ 
for some open neighborhoods $V_1 \simeq V_2$ of $(s_{1,p}^{0})^{-1}(0) \times \{0\}$ and $(s_{2,p}^{0})^{-1}(0) \times \{0\},$ respectively in ${{U}}^0_1 \times \mathbb{R}^{n_1}$ and ${{U}}^0_2 \times \mathbb{R}^{n_2}$ (precise meaning provided below),
\item $\widetilde{\widehat{f}}_{1,p,x} = \widetilde{\widehat{f}}_{2,p,x}$ up to $L_{\infty}[1]$-homotopy, that is, the following diagram commutes up to $L_{\infty}[1]$-homotopy:
\begin{equation}\label{octa}
\begin{tikzcd}
\mathcal{C}'_{f(p),\widetilde{f}_{1,p}(x)} \arrow{rr}{{\widehat{f}}_{1,p,x}}  & {} & \mathcal{C}_{p, x}  \arrow{d}{\widehat{\pi}_{1,p,(x,0)}}\\
 \mathcal{C}^{'\mathbb{R}^{n_1'}}_{f(p),(\widetilde{f}_{1,p}(x),0)} \cong \mathcal{C}^{'\mathbb{R}^{n_2'}}_{f(p),(f_{2,p}(x),0)} \arrow{u}{\widehat{\pi}^{-1}_{1,f(p),(\widetilde{f}_{1,p}(x),0)}} \arrow{d}[swap]{\widehat{\pi}^{-1}_{2,f(p),(\widetilde{f}_{2,p}(x),0)}} &{}& \mathcal{C}^{\mathbb{R}^{n_1}}_{p, (x,0)} \cong \mathcal{C}^{\mathbb{R}^{n_2}}_{p, (x,0)} \\
\mathcal{C}'_{f(p),\widetilde{f}_{2,p}(x)}  \arrow{rr}{{\widehat{f}}_{2,p,x}}  & {} & \mathcal{C}_{p, x} \arrow{u}[swap]{\widehat{\pi}_{2,p,(x,0)}}
\end{tikzcd}
\end{equation}
for each $x \in (s_{1,p}^{0})^{-1}(0) \times \{0\}.$ 

\end{enumerate}
Here ${(s_{i,p}^{0})^{-1}(0) \times \{0\}}$'s in the conditions (ii) and (iii) are to be understood as the same subset of both ${{U}}^0_1 \times \mathbb{R}^{n_1}$ and ${{U}}^0_2 \times \mathbb{R}^{n_2}$ modulo the identification from (\ref{urur}).
\end{defn}

\begin{lem}\label{hgjgkg}
$\sim$ is an equivalence relation.
\end{lem}
\begin{proof}
Symmetry and reflexivity are obvious. For transitivity, suppose we have $\overline{F}_2 \sim \overline{F}_3$ for some pre-morphism $\overline{F}_3$ in addition to (\ref{f1f2}) with its extension 
\[
\overline{F}^{n_3,n'_3}_3 = \left(\widehat{\mathcal{U}}^0_3 \times \mathbb{R}^{n_3}, \widehat{\mathcal{U}'}_3 \times \mathbb{R}^{n'_3}, f_3, \left\{\widetilde{f}_{3,p}\right\}, \left\{\widetilde{\widehat{f}}_{3,p,x}\right\}\right).
\]
By choosing smaller ${\widehat{\mathcal{U}}^0_i}$'s and larger ${\mathbb{R}^{n_i}}$'s if necessary, one can assume that
\[
\widehat{\mathcal{U}}^0_1 \times \mathbb{R}^{n_1} = \widehat{\mathcal{U}}^0_2 \times \mathbb{R}^{n_2} = \widehat{\mathcal{U}}^0_3 \times \mathbb{R}^{n_3}.
\]
Then it is straightforward to show that $\overline{F}_1 \sim \overline{F}_3$ in both cases $n_1 \geq n_3$ and $n_1 < n_3.$ In particular, for transitivity of the condition (iii), we can apply Lemma \ref{phxv} and Theorem \ref{anhp}.
\end{proof}

\begin{defn}[Morphism of Kuranishi spaces]
We define a \textit{morphism} from $\mathfrak{X} = \left(X, \left[\widehat{\mathcal{U}}\right]\right)$ to $\mathfrak{X}' = \left(X', \left[\widehat{\mathcal{U}}'\right]\right)$
by an equivalence class of a pre-morphism $\overline{F}$ from $\mathfrak{X}$ to $\mathfrak{X}':$
\[F := \left[\overline{F}\right] : \mathfrak{X} \rightarrow \mathfrak{X}'.\]
\end{defn}

\begin{defn}[Composition of morphisms]\label{afjfjggg}
Let $\mathfrak{X} = \left(X, \left[\widehat{\mathcal{U}}\right]\right),$ $\mathfrak{X}' = \left(X', \left[\widehat{\mathcal{U}'}\right]\right),$ and $\mathfrak{X}'' = \left(X'', \left[\widehat{\mathcal{U}''}\right]\right)$ be Kuranishi spaces. Let $F : \mathfrak{X} \rightarrow \mathfrak{X}'$ and $G : \mathfrak{X}' \rightarrow \mathfrak{X}''$ be morphisms between them represented by 
\begin{equation}\nonumber
\begin{cases}
\overline{F} = \left(\widehat{\mathcal{U}}, \widehat{\mathcal{U}'}, f, \left\{f_{p}\right\}, \left\{\widehat{f}_{p,x}\right\}\right),\\
\overline{G} = \left(\underline{\widehat{\mathcal{U}'}}, \widehat{\mathcal{U}''}, g, \left\{g_{f(p)}\right\}, \left\{\widehat{g}_{f(p),y}\right\}\right),
\end{cases}
\end{equation}
respectively with $[ \widehat{\mathcal{U}'}] = [\underline{\widehat{\mathcal{U}'}}]$.

There exist extended pre-morphisms
\begin{equation}\label{extfg1}
\begin{cases}
\overline{F}^{n_d, n'_t} = \left(\widehat{\mathcal{U}}^0 \times \mathbb{R}^{n_d}, \widehat{\mathcal{U}}^{'0} \times \mathbb{R}^{n'_t}, f, \left\{\widetilde{f}_{p} \right\}, \left\{\widetilde{\widehat{f}}_{p,x}\right\}\right),\\
\overline{G}^{\underline{n}_d, \underline{n}'_t} = \left(\underline{\widehat{\mathcal{U}}}^{'0} \times \mathbb{R}^{\underline{n}_d}, \widehat{\mathcal{U}''} \times \mathbb{R}^{\underline{n}_t'}, g, \left\{\widetilde{g}_{p'}\right\}, \left\{\widetilde{\widehat{g}}_{p',x'}\right\}\right),
\end{cases}
\end{equation}
of $\overline{F}$ and $\overline{G},$ respectively, so that $\widehat{\mathcal{U}}^{'0} \times \mathbb{R}^{n'_t} = \underline{\widehat{\mathcal{U}}}^{'0} \times \mathbb{R}^{\underline{n}_d}$ holds for some open atlases $\widehat{\mathcal{U}}^{'0} < \widehat{\mathcal{U}}^{'},$ and $\underline{\widehat{\mathcal{U}}}^{'0} < \underline{\widehat{\mathcal{U}}}^{'},$ and that all the base maps $\widetilde{f}_p$ and $\widetilde{g}_{p'}$ are surjective.

We define the \textit{composition} $G \circ F$ to be the following equivalence class:
\begin{equation}\label{morcom}
G \circ F := \left[\left(\widehat{\mathcal{U}}^0 \times \mathbb{R}^{n_d}, \widehat{\mathcal{U}}'' \times \mathbb{R}^{\underline{n}'_t}, g\circ f, \left\{\widetilde{g}_{f(p)} \circ \widetilde{f}_{p}\right\}, \left\{\widetilde{\widehat{f}}_{p,x} \circ \widetilde{\widehat{g}}_{f(p),f_p(x)}\right\} \right)\right].
\end{equation}
\end{defn}

\begin{prop}\label{ppidx}
The composition is well-defined and associative with the identity given by
\begin{equation}\label{idxm}
\mathrm{id}_{\mathfrak{X}}:= \left[\left(\widehat{\mathcal{U}}, \widehat{\mathcal{U}}, \mathrm{id}_X, \left\{\mathrm{id}_p\right\}, \left\{\widehat{\mathrm{id}}_{p,x}\right\}\right)\right]
\end{equation}
of each $\mathfrak{X} = \left(X, \left[\widehat{\mathcal{U}}\right]\right).$
\end{prop}
\begin{proof}
For well-definedness, we consider different choices of pre-morphism with respect to open subatlases $\widehat{\mathcal{U}'}^{\overline{0}} < \widehat{\mathcal{U}'}$ and $\underline{\widehat{\mathcal{U}'}}^{\overline{0}} < \underline{\widehat{\mathcal{U}'}}$ with nonnegative integers $\overline{n}_t'$ and $\overline{\underline{n}}_d',$, respectively, satisfying:
\[
\widehat{\mathcal{U}'}^{\overline{0}} \times \mathbb{R}^{\overline{n}_t'} = \underline{\widehat{\mathcal{U}'}}^{\overline{0}} \times \mathbb{R}^{\overline{\underline{n}}_d}.
\]

Then the equivalence
\[
\begin{split}
\bigg( \widehat{\mathcal{U}}^0 \times \mathbb{R}^{n_d}, &\widehat{\mathcal{U}}'' \times \mathbb{R}^{\underline{n}'_t}, g\circ f, \left\{\widetilde{g}_{f(p)} \circ \widetilde{f}_{p}\right\}, \left\{\widetilde{\widehat{f}}_{p,x} \circ \widetilde{\widehat{g}}_{f(p),f_p(x)}\right\} \bigg) \\
&\sim \left(\widehat{\mathcal{U}}^{\overline{0}} \times \mathbb{R}^{\overline{n}_d}, \widehat{\mathcal{U}}'' \times \mathbb{R}^{\overline{\underline{n}}_t'}, g \circ f, \left\{\widetilde{g}_{f(p)} \circ \widetilde{f}_{p} \right\}, \left\{\widetilde{\widehat{f}}_{p,x} \circ \widetilde{\widehat{g}}_{f(p),f_p(x)}\right\}\right)
\end{split}
\]
can be established by taking a common subatlas of $\widehat{\mathcal{U}}^0$ and $\widehat{\mathcal{U}}^{\overline{0}}$ (which always exists) and expanding them appropriately. Conditions (i) and (ii) in Definition \ref{defmoreq} are trivial. Condition (iii) is less trivial, but one can apply Theorem \ref{anhp}.

Consider the composition
\[
\widehat{f}_{p,x} \circ \widehat{g}_{f(p),f_p(x)}:
\mathcal C''_{g\circ f(p)}
\rightarrow
\mathcal C_p.
\]
It is not difficult to see that it satisfies
\[
\left(\widehat{f}_{p,x} \circ \widehat{g}_{f(p),f_p(x)}\right)_k \left(\overline{a}^k,\overline{\xi}^k\right)=0
\]
for all $k \geq 1, \ \left(\overline{a}^k,\overline{\xi}^k\right)  = \left({a}_1 \otimes \cdots \otimes a_k, {\xi}_1 \otimes \cdots \otimes \xi_k\right) \in \mathcal C_{g\circ f(p)}^{''\otimes k}$ such that
\[
\operatorname{supp}({a}_i,{\xi}_i)\subset W''_{g \circ f (x)} \setminus \mathrm{Im}(g\circ f) \text{ for some } i.
\]

Hence, by Lemma \ref{pcpc}, it factors through $\mathcal C''_{g\circ f(p),\,g\circ f}$ as
\[
\mathcal C''_{g\circ f(p)}
\;\xrightarrow{\;\widehat{\varepsilon}_{g\circ f(p),\,g\circ f}\;}
\mathcal C''_{g\circ f(p),\,g\circ f}
\;\xrightarrow{\;(\widehat{g}_{f(p)}\circ \widehat{f}_p)^c\;}
\mathcal C_p.
\]

The associativity is a simple consequence of that for each type of compositions: $f \circ g, g_{f(p)} \circ f_p,$ and $\widehat{f}_{p,x} \circ \widehat{g}_{f(p), f_p(x)}.$

For the identity morphism $\mathrm{id}_{\mathfrak{X}},$ and a different choice of representative than (\ref{idxm}), say
\begin{equation}\nonumber
\left(\widehat{\underline{\mathcal{U}}}, \widehat{\underline{\mathcal{U}}}, \mathrm{id}_X, \left\{\underline{\mathrm{id}}_p\right\}, \left\{\underline{\widehat{\mathrm{id}}}_{p,x}\right\}\right)
\end{equation}
with ${\widehat{\mathcal{U}}} \sim \underline{\widehat{\mathcal{U}}},$ so that
\[
\widehat{\mathcal{U}}^0 \times \mathbb{R}^{{n}} = \widehat{\underline{\mathcal{U}}}^0 \times \mathbb{R}^{\underline{n}}
\]
for some subatlases $\widehat{\mathcal{U}}^0 $ and $\underline{\widehat{\mathcal{U}}}^0$ of $\widehat{\mathcal{U}}$ and $\underline{\widehat{\mathcal{U}}}$ with $n, \underline{n} \geq 0,$ respectively, it follows that
\[
\left(\widehat{\mathcal{U}}, \widehat{\mathcal{U}}, \mathrm{id}_X, \left\{\mathrm{id}_p\right\}, \left\{\widehat{\mathrm{id}}_{p,x}\right\}\right)^{n,n} := \left(\widehat{{\mathcal{U}}} \times \mathbb{R}^{n}, \widehat{{\mathcal{U}}} \times \mathbb{R}^{n}, \mathrm{id}_X, \left\{\widetilde{\mathrm{id}}_p\right\}, \left\{{\widetilde{\widehat{\mathrm{id}}}}_{p,x}\right\}\right),
\]
and
\[
\left(\widehat{\underline{\mathcal{U}}}, \widehat{\underline{\mathcal{U}}}, \mathrm{id}_X, \left\{\underline{\mathrm{id}}_p\right\}, \left\{\underline{\widehat{\mathrm{id}}}_{p,x}\right\}\right)^{\underline{n}, \underline{n}} := \left(\widehat{\underline{{\mathcal{U}}}} \times \mathbb{R}^{\underline{n}}, \widehat{\underline{{\mathcal{U}}}} \times \mathbb{R}^{\underline{n}}, \mathrm{id}_X, \left\{\widetilde{\underline{{\mathrm{id}}}}_p\right\}, \left\{{\underline{\widetilde{\widehat{\mathrm{id}}}}}_{p,x}\right\}\right)
\]
are equivalent: The conditions (i) and (ii) in Definition \ref{defmoreq} are trivial, and (iii) is a consequence of Theorem \ref{anhp}.
 
Consider the morphisms 
\begin{equation}\nonumber
\begin{cases}
{F} = \left[\left(\widehat{\mathcal{U}}, \widehat{\mathcal{U}'}, f, \left\{f_{p}\right\}, \left\{\widehat{f}_{p,x}\right\}\right)\right],\\
\mathrm{id}_{\mathfrak{X}} := \left[\left(\widehat{\underline{\mathcal{U}}}, \widehat{\underline{\mathcal{U}}}, \mathrm{id}_X, \left\{\underline{\mathrm{id}}_p\right\}, \left\{\widehat{\underline{\mathrm{id}}}_{p,x}\right\}\right)\right],\\
\mathrm{id}_{\mathfrak{X}'} := \left[\left(\widehat{\underline{\mathcal{U}}}', \widehat{\underline{\mathcal{U}}}', \mathrm{id}_{X'}, \left\{\underline{\mathrm{id}}_{p'}\right\}, \left\{\widehat{\underline{\mathrm{id}}}_{p',x'}\right\}\right)\right]
\end{cases}
\end{equation}
with the relations
\begin{equation}\label{rel2equ}
\widehat{\mathcal{U}}^{0} \times \mathbb{R}^{n} = \underline{\widehat{\mathcal{U}}}^{0} \times \mathbb{R}^{\underline{n}} \text{ and } \widehat{\mathcal{U}'}^{0} \times \mathbb{R}^{n'} = \underline{\widehat{\mathcal{U}'}}^{0} \times \mathbb{R}^{\underline{n}'}
\end{equation}
for some $n, \underline{n}, n',$ and $\underline{n}' \geq 0.$
Then the composition of $F$ and $\mathrm{id}_{\mathfrak{X}}$ is given by
\[
\begin{split}
F \circ \mathrm{id}_{\mathfrak{X}} & = \left[\left( \widehat{\underline{\mathcal{U}}}^{0} \times \mathbb{R}^{\underline{n}}, \widehat{\mathcal{U}}' \times \mathbb{R}^{{n}}, f\circ \mathrm{id}_X, \left\{ \widetilde{f}_{p} \circ \widetilde{\mathrm{id}}_{p}\right\}, \left\{\widetilde{\widehat{\mathrm{id}}}_{p,x} \circ \widetilde{\widehat{f}}_{p,x} \right\} \right)\right]\\
&  = \left[\left({\widehat{\underline{\mathcal{U}}}^{0}} \times \mathbb{R}^{\underline{n}}, \widehat{\mathcal{U}}' \times \mathbb{R}^{{n}}, f, \left\{\widetilde{f}_{p}\right\}, \left\{\widetilde{\widehat{f}}_{p,x}\right\}\right)\right].
\end{split}
\]
We claim that $\left(\widehat{\underline{\mathcal{U}}}^{0} \times \mathbb{R}^{\underline{n}}, \widehat{\mathcal{U}}' \times \mathbb{R}^{{n}}, f, \left\{\widetilde{f}_{p}\right\}, \left\{\widetilde{\widehat{f}}_{p,x}\right\}\right)$ is equivalent to $F.$ Consider the following extension of $F:$ 
\[
F^{n,n} = \left(\widehat{\mathcal{U}}, \widehat{\mathcal{U}}', f, \left\{f_{p}\right\}, \left\{\widehat{f}_{p,x}\right\}\right)^{n,n} := \left(\widehat{\mathcal{U}}^0 \times \mathbb{R}^n, \widehat{\mathcal{U}'} \times \mathbb{R}^n , f, \left\{\widetilde{f}_{p}\right\}, \left\{\widetilde{\widehat{f}}_{p,x}\right\}\right).
\]
Since the subatlas $\widehat{\underline{\mathcal{U}}}^{0}$ can be chosen in such a way that (\ref{rel2equ}) holds, we then have 
\[
\begin{split}
\bigg({\widehat{\underline{\mathcal{U}}}^{0}} \times \mathbb{R}^{\underline{n}}, \widehat{\mathcal{U}}' \times \mathbb{R}^{{n}}, f, &\left\{\widetilde{f}_{p}\right\}, \left\{\widetilde{\widehat{f}}_{p,x}\right\} \bigg)\\ 
&\sim \bigg(\widehat{\mathcal{U}}^0 \times \mathbb{R}^n, \widehat{\mathcal{U}}' \times \mathbb{R}^n , f, \left\{\widetilde{f}_{p}\right\}, \left\{\widetilde{\widehat{f}}_{p,x}\right\}\bigg);
\end{split}
\]
the conditions (i) and (ii) in Definition \ref{defmoreq} hold trivially, and (iii) is a consequence of Theorem \ref{anhp}.
$\mathrm{id}_{\mathfrak{X}'} \circ F = F$ can be shown similarly, so we omit its proof.
\end{proof}

\begin{defn}
\textit{The category of Kuranishi spaces} is defined by a category \textbf{Kur} that consists of:
\begin{equation}\nonumber
\begin{cases}
\text{Ob}(\textbf{Kur}) =\{L_{\infty}\text{-Kuranishi spaces}\}\\
\text{Mor}(\textbf{Kur}) = \{\text{Equivalence classes of pre-morphisms with the composition } \circ \}.
\end{cases}
\end{equation}
\end{defn}

\subsection{Manifold as an $L_{\infty}$-Kuranishi space}\label{masks}

Example \ref{man} illustrates that a smooth manifold can be regarded as a special type of an $L_{\infty}$-Kuranishi space in the sense that it naturally determines an $L_{\infty}$-Kuranishi atlas and therefore gives rise to an $L_{\infty}$-Kuranishi space. In fact, we have the following result:

\begin{thm}\label{pman}
The category of smooth manifolds $\mathbf{Man}$ is a subcategory of $\mathbf{Kur}$. In other words, there exists an embedding of categories
\begin{equation}\nonumber
\mathscr{F} : \mathbf{Man} \hookrightarrow \mathbf{Kur}.
\end{equation}
\end{thm}

\begin{proof}
Let $M$ be a smooth manifold (equipped with the zero form) and $\left(M, \widehat{\mathcal{U}}^{\mathrm{man}}_M\right)$ the Kuranishi atlas in Example \ref{man}. The mapping that assigns $\left(M, \widehat{\mathcal{U}}^{\mathrm{man}}_M\right)$ to $M$ determines an injection between the equivalence classes, hence an injective object map

\begin{equation}\nonumber
\mathscr{F}_{\text{ob}} : \textbf{Man}_{\text{ob}} \rightarrow \textbf{Kur}_{\text{ob}}.
\end{equation}

Let $f: M \to N$ be a smooth map between manifolds. Then we assign
\begin{equation}\nonumber
\begin{split}
\mathscr{F}_{\mathrm{Mor}}(M \stackrel{f}{\to} N) &:= \left[\left(\widehat{\mathcal{U}}^{\mathrm{man}}_M, \widehat{\mathcal{U}}^{\mathrm{man}}_N, f, \{f^{\mathrm{man}}_p\}, \left\{\widehat{f}^{\mathrm{man}}_{p,x}\right\}\right)\right]\\ 
&\quad\quad\quad\quad\quad\quad\quad\quad : \left(M, \left[\widehat{\mathcal{U}}^{\mathrm{man}}_M\right]\right) \to \left(N, \left[\widehat{\mathcal{U}}^{\mathrm{man}}_N\right]\right),
\end{split}
\end{equation}
where the base component map 
\[
f_p^{\mathrm{man}}: U_p \to U_{f(p)}'
\]
is given by
\[
f_p^{\mathrm{man}} := \psi_{f(p)}^{-1} \circ f \circ \psi_{p},
\]
and the $L_{\infty}$-component 
\[
\widehat{f}_{p,x}^{\mathrm{man}, \mathrm{c}}: \Omega^{\bullet+1}\left(W'_{f_p(x)}\right)_{f_p} \to \Omega^{\bullet+1}(W_x)
\] 
with $\widehat{f}_{p,x}^{\mathrm{man}} = \widehat{f}_{p,x}^{\mathrm{man}, \mathrm{c}} \circ \widehat{\varepsilon}_{p, f_p(x), f_p}$ by
\begin{equation}\nonumber
\begin{split}
\widehat{f}_{p,x}^{\mathrm{man}, \mathrm{c}} : C^{\infty}_{f_p}\left(W'_{f_p(x)}\right)^{(j)} \otimes \Omega^{\bullet+1}\left(W'_{f_p(x)}\right) &\to \Omega^{\bullet+1}(W_x).\\
h \otimes \xi &\mapsto f_p^*\widetilde{h} \cdot f_p^* \xi
\end{split}
\end{equation}
for each $x \in s_p^{-1}(0) = U_p,$ $j \geq 1,$ and $\widetilde{h} \in C^{\infty}\left(W'_{f_p(x)}\right)$ with $\left[\widetilde{h}\right]_j = h.$ Its well-definedness follows from $\widetilde{g} \circ f_p =0$ for every $\widetilde{g} \in I_{f_p}^j.$

It further follows that $\widehat{f}_{p,x}^{\mathrm{man}, \mathrm{c}}$ is a chain map: We have
\[
\begin{split}
(j) : \widehat{f}_{p,x}^{\mathrm{man}, \mathrm{c}} \left(d(h \otimes \xi)\right) &= \widehat{f}_{p,x}^M \left( [1]_{j-1} \otimes d(\widetilde{h}\xi) \right) = f_p^*1 \cdot f_p^*\left(d(\widetilde{h}\xi)\right)\\ 
&= f_p^*\left(d(\widetilde{h}\xi)\right) =  df_p^*\left(\widetilde{h}\xi\right) = d\left(f_p^*\widetilde{h} \cdot f_p^*\xi\right) = d\widehat{f}_{p,x}^{\mathrm{man}, \mathrm{c}} ( h \otimes \xi)
\end{split}
\]
for arbitrary $j \geq 1.$
We then verify the conditions (i) to (iii) in Definition \ref{prem}. (i) and (ii) follow immediately from the definition of the base coordinate change $\phi_{pq}:= \psi_q^{-1} \circ \psi_p|_{U_{pq}}$ in Example \ref{man} and the above definition of the base component map $f_p:= \psi_f(p) \circ f \circ \psi_p,$ respectively.

For (iii), we observe that
\[
\begin{split}
\widehat{\phi}_{pq,x} \circ &\widehat{f}^{\mathrm{man}}_{q, \phi_{pq}(x)}(\xi) = \widehat{\phi}^{\mathrm{c}}_{pq,x} \circ \widehat{\varepsilon}_{q, \phi_{pq}(x), \phi_{pq}} \circ \widehat{f}^{\mathrm{man},\mathrm{c}}_{q, \phi_{pq}(x)} \circ \widehat{\varepsilon}_{f(q), f_q\circ \phi_{pq}(x), f_q}(\xi)\\ 
&= \widehat{\phi}^{\mathrm{c}}_{pq,x} \circ \widehat{\varepsilon}_{q, \phi_{pq}(x), \phi_{pq}} \circ \widehat{f}^{\mathrm{man},\mathrm{c}}_{q, \phi_{pq}(x)}(1 \otimes \xi) \\
&= \widehat{\phi}^{\mathrm{c}}_{pq,x} \circ \widehat{\varepsilon}_{q, \phi_{pq}(x), \phi_{pq}}(f_{q}^*\xi) = \widehat{\phi}^{\mathrm{c}}_{pq,x}(1 \otimes f_q^*\xi) {=} \phi_{pq}^*f^*_q\xi = f_p^*(\phi'_{f(p)f(q)})^*\xi\\
&= \widehat{f}^{\mathrm{man},\mathrm{c}}_{p,x} \big(1 \otimes ({\phi}'_{f(p)f(q)})^*(\xi)\big)= \widehat{f}^{\mathrm{man},\mathrm{c}}_{p,x} \circ \widehat{\varepsilon}_{f(p),f_p(x),f_p}({\phi}'_{f(p)f(q)})^*(\xi)\\ 
&= \widehat{f}^{\mathrm{man},\mathrm{c}}_{p,x} \circ \widehat{\varepsilon}_{f(p),f_p(x),f_p} \circ \widehat{\phi}^{'\mathrm{c}}_{f(p)f(q),f_p(x)}(1 \otimes \xi)\\
&= \widehat{f}^{\mathrm{man},\mathrm{c}}_{p,x} \circ \widehat{\varepsilon}_{f(p),f_p(x),f_p} \circ \widehat{\phi}^{'{\mathrm{c}}}_{f(p)f(q),f_p(x)} \circ \widehat{\varepsilon}_{f(q),\phi'_{f(p)f(q)}\circ f_p(x), \phi'_{f(p)f(q)}}(\xi)\\ 
&= \widehat{f}^{\mathrm{man}}_{p,x} \circ \widehat{\phi}'_{f(p)f(q),f_p(x)}(\xi)
\end{split}
\]
for every $\xi \in \Omega^{\bullet +1}\left(W'_{f(q), f_q \circ \phi_{pq}(x)}\right) = \Omega^{\bullet +1}\left(W'_{f(q), \phi'_{f(p)f(q)} \circ f_p(x)}\right).$ In other words, the above diagram commutes on the nose. We then have the same diagram for the augmented de Rham complexes, which also commutes strictly. (The chain maps between the augmented chain complexes are obtained by adding the induced map between the augmentations.) Thus, the condition (iii) holds.

Moreover, the following three properties are immediate:
\begin{enumerate}
  \item The identity morphisms are preserved, that is,
  \[
    \mathscr{F}(\mathrm{id}_M) = \mathrm{id}_{\mathscr{F}(M)}.
  \]
In fact, we have $\mathscr{F}(\mathrm{id}_M) = \left[\left(\widehat{\mathcal{U}}^{\mathrm{man}}_M, \widehat{\mathcal{U}}^{\mathrm{man}}_M, \mathrm{id}_M, \left\{I_p\right\}, \left\{\widehat{I}_{p,x}\right\}\right)\right],$ where $I_p = \psi_{\mathrm{id}_M(p)} \circ \mathrm{id}_M \circ \psi_p = \psi_{p}^{-1} \circ \psi_p = \mathrm{id}_p.$ Also note that $\mathrm{id}_p$ is surjective, so we can identify $ C^{\infty}_{f_p}\left(W'_{f_p(x)}\right)^{(j)} \otimes \Omega^{\bullet+1}\left(W'_{f_p(x)}\right)$ with $\Omega^{\bullet+1}\left(W'_{f_p(x)}\right).$ Then $\widehat{I}_{p,x}$ is given by $\widehat{I}_{p,x}\left(\widetilde{h} \xi\right) = \mathrm{id}_p^*\widetilde{h} \cdot \mathrm{id}_p^*(\xi) = \widetilde{h}\xi.$ 
\item 
$\mathscr{F}_{\mathrm{Mor}}$ respects the compositions:
\[
\mathscr{F}_{\mathrm{Mor}}\left(M \xrightarrow{f} N \xrightarrow{g} P\right) = \mathscr{F}_{\mathrm{Mor}}\left(N \xrightarrow{g} P\right) \circ \mathscr{F}_{\mathrm{Mor}}\left(M \xrightarrow{f} N\right).
\]
We need to show that the two pre-morphisms
\[
\left(\widehat{\mathcal{U}}^{\mathrm{man}}_M, \widehat{\mathcal{U}}^{\mathrm{man}}_P, g \circ f, \left\{g_{f(p)} \circ f_{p}\right\}, \left\{\widehat{\left({g}_{f(p)} \circ {f}_{p}\right)}_{x}\right\}\right)
\]
and
\[
\left(\widehat{\mathcal{U}}^{\mathrm{man}}_M \times \mathbb{R}^{n_d}, \widehat{\mathcal{U}}^{\mathrm{man}}_P \times \mathbb{R}^{\underline{n}'_t}, g \circ f, \left\{\widetilde{g}_{f(p)} \circ \widetilde{f}_p\right\}, \left\{\widetilde{\widehat{f}}_{p,x} \circ \widetilde{\widehat{g}}_{f(p), f_p(x)}\right\}\right)
\]
are equivalent, where the notations are as in Definition 6.15 and the proof of Proposition \ref{ppidx}. The conditions (i) and (ii) of Definition \ref{defmoreq} obviously hold, and (iii) follows from the fact that the two $L_{\infty}[1]$-morphisms are quasi-isomorphisms (between acyclic chain complexes by Corollary \ref{corpoinc}), hence $L_{\infty}[1]$-homotopic to each other by Corollary \ref{anhp}.

\item It is straightforward to verify that the map on morphism sets
\begin{equation}\nonumber
\begin{split}
\mathscr{F}_{\mathrm{Mor}} \colon \operatorname{Mor}_{\mathbf{Man}}(M, N) &\longrightarrow \operatorname{Mor}_{\mathbf{Kur}}\bigl((M, \widehat{\mathcal{U}}^{\mathrm{man}}_M), (N, \widehat{\mathcal{U}}^{\mathrm{man}}_N)\bigr), \\
(f, \{f_p\}) &\longmapsto \left[\left(\widehat{\mathcal{U}}^{\mathrm{man}}_M, \widehat{\mathcal{U}}^{\mathrm{man}}_N, f, \{f_p\}, \{\widehat{f}_{p,x}\}\right)\right]
\end{split}
\end{equation}
is injective (cf.\ Definition~\ref{defmoreq}).

In general, it is not surjective because a linear map $\widehat{k}_{p,x} \colon C^{\infty}\left(W'_{f_p(x)}\right) \longrightarrow C^{\infty}(W_x)$, as an $L_{\infty}[1]$-morphism between the trivial $L_{\infty}[1]$-algebras, need not be $L_{\infty}[1]$-homotopic to $\widehat{f}_{p,x}$. However, the surjectivity for the de Rham part follows from the fact that any two chain maps between augmented de Rham complexes are automatically $L_{\infty}[1]$-homotopic to each other by Theorem~\ref{anhp}.
\end{enumerate}

This completes the proof of Theorem~\ref{pman}.
\end{proof}

\begin{rem}
It is not difficult to show that Kuranishi spaces without local group actions naturally form a subcategory $\mathbf{Kur}^u \hookrightarrow \mathbf{Kur}$. Moreover, the functor $\mathscr{F}$ factors through $\mathbf{Kur}^u$.
\end{rem}

\appendix

\section{$L_{\infty}[1]$-algebras and their homotopy theory}

In this section, we briefly introduce $L_{\infty}[1]$-algebras and their homotopy theories, following \cite{Kim2}. We first recall the notion of \textit{graded symmetric algebra} $SC$ of a vector space $C$ over a field $\mathbf{k},$
$$
S C := TC / \langle v \otimes v' - (-1)^{|v|\cdot|v'|} v' \otimes v \rangle,
$$
with its degree $k$ component $S^k C := \{v \in S C \mid v \text{ is homogeneous of degree } k\}.$
We have a decomposition
$$
S C = \bigoplus\limits_{k=0}^{\infty} S^k C
$$
with the induced product $\odot$ on each component.
We denote by $\text{Sh}(i, k-i)$ the set of $(i, k-i)$-unshuffles, and the sign $\sgn(\tau)$ for $\tau \in \text{Sh}(i, k-i)$ is defined for homogeneous elements $a_1, \cdots, a_k \in C$, we write
$$
a_{\tau(1)} \odot \cdots \odot a_{\tau(k)} = \sgn(\tau) a_1 \odot \cdots \odot a_k.
$$

\begin{defn}
An \textit{$L_{\infty}[1]$-algebra} is a pair $\left(C, \{l_k\}\right)$ consisting of a vector space $C$ and a family of degree 1 linear maps
$$
l_k : S^k C \rightarrow C, \ k \geq 0,
$$
satisfying the relations
\begin{equation}\label{quadrel}
\sum\limits_{i = 0}^{k} \sum\limits_{\tau \in {\text{Sh}}(i, k-i)} {\sgn(\tau)} l_{k-i+1}\left(l_i(a_{\tau(1)}, \cdots, a_{\tau(i)}),a_{\tau(i+1)}, \cdots, a_{\tau(k)}\right) = 0.
\end{equation}
\end{defn}

\begin{defn}
Let $(C,\{l_k\})$ and $(C', \{l'_k\})$ be two $L_{\infty}[1]$-algebras. An $L_{\infty}[1]$\textit{-algebra morphism}, or simply $L_{\infty}[1]$\textit{-morphism}
\begin{equation}\label{lrel}
f : C \rightarrow C'
\end{equation}
is a family of degree 0 linear maps
$$
f_k : S^kC \rightarrow C', \ k \geq 0,
$$
satisfying the relations
\begin{equation}\label{frel}
\begin{split}
\sum\limits_{i = 0}^{k}& \sum\limits_{\tau \in {\text{Sh}}(i, k-i)} {\sgn(\tau)} f_{k-i+1}\left(l_i(a_{\tau(1)}, \cdots, a_{\tau(i)}),a_{\tau(i+1)}, \cdots, a_{\tau(k)}\right)\\
&= \sum\limits_{\substack{t, j_1, \cdots, j_t \geq 1,\\ j_1 + \cdots + j_t = k}} \sum\limits_{\tau \in S_k}  \frac{\sgn(\tau)}{t! j_1! \cdots j_t!} \  l'_{t}\bigl(f_{j_1}(a_{\tau(1)}, \cdots, a_{\tau(j_1)}), \cdots, \\
& \quad \quad \quad \quad \quad \quad \quad \quad \quad \quad \quad \quad f_{j_t}(a_{\tau(k -(j_1 + \cdots + j_{t-1}))}, \cdots, a_{\tau(k)})\bigr).
\end{split}
\end{equation}
Here, $S_k$ denotes the symmetric group of permutations of $k$ elements. 
\end{defn}

\begin{defn}
For two $L_{\infty}[1]$-morphisms
$$
f : C \rightarrow C', \ g: C' \rightarrow C'',
$$
we define their \textit{composition}
$$
g \circ f : C \rightarrow C''
$$
by a family of linear maps of degree 0 for $k \geq 0$
\begin{equation}\nonumber
\begin{split}
(g \circ f)_k := &\sum\limits_{i = 0}^{k} \sum\limits_{\tau \in S_k}  \frac{\sgn(\tau)}{t! j_1! \cdots j_t!} \ g_{t}\bigl(f_{j_1}(a_{\tau(1)}, \cdots, a_{\tau(j_1)}), \cdots,\\
&\quad \quad\quad \quad\quad \quad\quad \quad f_{j_t}(a_{\tau(k -(j_1 + \cdots + j_{t-1}))}, \cdots, a_{\tau(k)})\bigr).
\end{split}
\end{equation}
It is straightforward to verify that $\{(g \circ f)_k\}_{k \geq 0}$ satisfies the relation (\ref{frel}).
\end{defn}

\begin{defn}
We say an $L_{\infty}[1]$-algebra $\{l_k\}_{k \geq 0}$ is \textit{strict} if $l_0 = 0.$ Otherwise, we say it is \textit{curved.} We similarly define \textit{strict/curved} $L_{\infty}[1]$-morphisms.
\end{defn}

In the strict case, the relations (\ref{lrel}) and (\ref{frel}) coincide with the differential and the chain map relations, respectively. That is, they satisfy
\[
l_1 \left( l_1(a) \right) = 0; \ l'_1 \left(f_1 (a) \right) = f_1 \left( l_1(a) \right).
\]
\begin{defn}
We say that a strict $L_{\infty}[1]$-algebra $(C,\{l_k\})$ is \textit{acyclic} if its cohomology for each degree vanishes, that is, if 
\[
H^*(C) = \frac{\ker{l_1}}{\mathrm{Im}l_1} = 0.
\]
We say that a strict $L_{\infty}[1]$-morphism $\{f_k\}_{k \geq 1}$ between strict $L_{\infty}[1]$-algebras is a \textit{quasi-isomorphism} if $f_1$ is a quasi-isomorphic chain map.
\end{defn}

From now on and elsewhere in this paper, we always assume the strictness of $L_{\infty}[1]$-algebras and their morphisms.

To define homotopies between two $L_{\infty}[1]$-morphisms, we need to introduce the following notion:
\begin{defn}[Models of $\Delta^1 \times C$]\label{mdl1}
Let $C$ be an $L_{\infty}[1]$-algebra. We say an $L_{\infty}[1]$-algebra $\mathfrak{C}$ is a \textit{model} of $\Delta^1 \times C$ if there exist $L_{\infty}[1]$-morphisms
\[
\Eval_j : \mathfrak{C} \rightarrow C, \ \ j = 0,1\\
\]
and
\[
\Incl : {C} \rightarrow \mathfrak{C},
\]
with the following properties:
\begin{enumerate}[label = (\roman*)]
\item $\left(\Eval_j\right)_{k \geq 2} \equiv 0, \ j =0,1, \ \Incl_{k \geq 2} \equiv 0.$
\item $\Eval_j, \ j =0,1$ and  $\Incl$ are quasi-isomorphisms.
\item $(\Eval_j)_1 \circ \Incl = \mathrm{id}_C.$
\item $(\Eval_0)_1 \oplus (\Eval_{1})_1 : \mathfrak{C} \rightarrow {C} \oplus C$ is surjective.
\end{enumerate}
\end{defn}

Using the notion of models, we can define homotopies between $L_{\infty}[1]$-morphisms:

\begin{defn}[Homotopy]\label{defn:homotopy}
We say that two $L_{\infty}[1]$-morphisms $f_0, f_1 : (C, \{l_k\}) \rightarrow (C', \{l'_k\})$ are \textit{homotopic} if there exist a \textit{model of} $\Delta^1 \times C'$ denoted by $\mathfrak{C}'$ and an $L_{\infty}[1]$-morphism $h: C \rightarrow \mathfrak{C}'$ such that we have $f_j = \Eval_j \circ h, \ j = 0,1.$
\end{defn}

\begin{lem}\cite[Lemma 3.3]{Kim2}
Homotopies define an equivalence relation.
\end{lem}

We introduce  homotopy equivalences for $L_\infty[1]$-algebras.

\begin{defn}
An $L_{\infty}[1]$-morphism $f : C \rightarrow C'$ is a \textit{homotopy equivalence} if there exists another $L_{\infty}[1]$-morphism $g : C' \rightarrow C$ such that $g \circ f$ and  $f \circ g$ are homotopic to $\mathrm{id}_C$ and $\mathrm{id}_{C'},$ respectively.
\end{defn}

\begin{prop}\cite[Proposition 3.7]{Kim2}\label{prop:heer}
Homotopy equivalences define an equivalence relation.
\end{prop}

We show that a key theorem on quasi-isomorphisms, which shows the usefulness of our definition.
\begin{thm}\label{anhp}
Any two quasi-isomorphic $L_{\infty}[1]$-morphisms $f_0, f_1 : C \rightarrow C'$ are homotopic.
\end{thm}

\begin{proof}
This theorem is a special case of \cite[Corollary 4.6]{Kim2} for two morphisms.
\end{proof}

We close this section by stating an $L_{\infty}[1]$-algebra version of the Whitehead theorem over a field.

\begin{thm}[Whitehead theorem]\cite[Theorem 3.13]{Kim2}\label{wht}
Over a field and for strict $L_{\infty}[1]$-algebras, a quasi-isomorphism is a homotopy equivalence.
\end{thm}

\section{$L_{\infty}[1]$-structures arising from V-algebras}

In this section, we study an example of $L_{\infty}[1]$-algebras arising from presymplectic foliations. For this purpose, we introduce V-algebras and define their completions.

\subsection{V-algebras}

We introduce V-algebras following \cite{Voronov1} and \cite{CS}.

\begin{defn}[V-algebras]\cite{Voronov1}\label{Voronov1}
A \textit{V-algebra} is a triple $(\mathfrak{h}, \mathfrak{a}, \Pi)$ such that
\begin{enumerate}
\item[--] $\mathfrak{h}$ is a graded Lie algebra over a field $\mathbf{k}$.
\item[--] $\mathfrak{a}$ is an abelian subalgebra of $\mathfrak{h}$.
\item[--] $\Pi : \mathfrak{h} \rightarrow \mathfrak{a}$ is the associated projection.
\item[--] $\ker \Pi$ is a Lie subalgebra of $\mathfrak{h}$.
\end{enumerate}
Let $P$ be a Maurer-Cartan element in $\mathfrak{h}$, i.e., an element of degree $1$ with $[P,P]=0$. The triple $(\mathfrak{h}, \mathfrak{a}, \Pi)$ together with such a choice of $P$ determines a family of operators:
\begin{equation}\label{dlpk}
\begin{split}
&l^{P}_k : \mathfrak{a}^{\otimes k} \rightarrow \mathfrak{a},\\
&\begin{cases}
(x_1, \cdots, x_k) &\mapsto \Pi[ \cdots [[P, {x_1}], {x_2}], \cdots, {x_k}], \quad \text{if }k \geq 1,\\
\quad \quad 1 &\mapsto \Pi P, \quad \quad \quad \quad \quad \quad \quad \quad \quad \quad \ \text{if } k = 0.
\end{cases}
\end{split}
\end{equation}
\end{defn}

Then we have:
\begin{lem}\label{valinf}
The family $\{l^P_k\}_{k \geq 0}$ forms a curved $L_{\infty}[1]$-algebra.
\end{lem}

\begin{proof}
The Jacobiator can be shown to vanish for each $n$, as it is given by $l^{\frac{1}{2}[P,P]}_n \equiv 0$. For a detailed proof, see \cite[Theorem 1]{Voronov1}.
\end{proof}

\begin{exam}[Derivations on graded commutative algebras]
Let $A$ be a graded commutative algebra over a field $\mathbf{k}$. We denote by $\mathrm{Der}(A)$ the derivations on $A$, namely, $\mathbf{k}$-linear maps $D : A \rightarrow A$ satisfying the Leibniz rule. Note that $\mathrm{Der}(A)$ is a module over $A$; each $a \in A$ acts on $D$ by $D \mapsto a \cdot D$. Moreover, $\mathrm{Der}(A)$ has a natural graded Lie structure. We then consider
\begin{equation}\nonumber
\widehat{S}_A\big(\mathrm{Der}(A)[-1]\big)[1],
\end{equation}
the completed symmetric algebra of $\mathrm{Der}(A)$ over $A$. This is generated by the graded Lie subalgebra
\begin{equation}\nonumber
A[1] \oplus \mathrm{Der}(A),
\end{equation}
whose Lie structure is induced from those of $A$ and $\mathrm{Der}(A)$. For example, the Lie brackets for crossing terms are given by $\mathrm{Der}(A) \ni [a, D] := a \cdot D - (-1)^{|a|\cdot|D|}D(a \cdot -)$. The following lemma follows immediately.
\end{exam}

\begin{lem}\label{vallem}
$A[1]$ is an abelian Lie subalgebra of $\widehat{S}_A(\mathrm{Der}(A)[-1])[1]$, and the triple
\begin{equation}\nonumber
\big(\widehat{S}_A\big(\mathrm{Der}(A)[-1]\big)[1], A[1], \Pi \big)
\end{equation}
with a Maurer-Cartan element is a V-algebra.
\end{lem}
\begin{exam}
Let $A = C^{\infty}(M)$ be the space of smooth functions on a manifold $M$ with the standard commutative product. Then $\widehat{S}_A\big(\mathrm{Der}(A)[-1]\big)[1]$ can be shown to equal the space of degree-shifted multivectors, 
\[
\Gamma\big(M, \bigwedge{}^{\bullet+1}TM[-1]\big).
\]
\end{exam}

\cite[Theorem 3.2]{CS} shows that a smooth $1$-parameter family of V-algebras
\[
\mathcal{V}(t) = \big(\mathfrak{h}(t), \mathfrak{a}(t), \Pi(t)\big), \quad t \in [0,1],
\]
with a family of Maurer-Cartan elements $P(t) \in \mathfrak{h}(t)^1$ produces an $L_{\infty}[1]$-isomorphism from $\mathfrak{a}(0)$ to $\mathfrak{a}(1)$. We briefly explain their result. The smooth family $\{\mathfrak{h}(t)\}_{t \in [0,1]}$ determines a flow
\[
\phi_t : \mathfrak{h}(0) \rightarrow \mathfrak{h}(t), \quad t \in [0,1],
\]
with generating vector field $m_t \in T\mathfrak{h}(t)$, satisfying the differential equation $\frac{d\phi_t}{dt} = m_t \circ \phi_t$.

We assume that the family satisfies
\begin{equation}\label{pkkpt}
\phi_t\big(\ker\big(\Pi(0)\big)\big) \simeq \ker\big(\Pi(t)\big) \simeq \ker\big(\Pi(0)\big) \quad \text{for all } t \in [0,1].
\end{equation}

Regarding the $L_{\infty}[1]$-algebra structure on $\mathfrak{a}(t)$ as the coalgebra structure on $S\big(\mathfrak{a}(t)\big)$, we define the following coalgebra maps: $Q(t)$, $M(t)$, and $U(t)$.

\begin{enumerate}
\item The coalgebra map
\[
Q(t) : S\big(\mathfrak{a}(t)\big) \rightarrow S\big(\mathfrak{a}(t)\big), \quad t \in [0,1]
\]
is defined by
\[
Q^k(t)(\xi_1, \cdots, \xi_k) := \Pi_t[\cdots[P(t), \xi_1], \cdots, \xi_k]
\]
with the property that $Q(t)^0 = 0$ as in Lemma \ref{liiifdr} (i).

\item The coalgebra map
\[
M(t) : S\big(\mathfrak{a}(t)\big) \rightarrow S\big(\mathfrak{a}(t)\big), \quad t \in [0,1]
\]
is defined by
\[
M^k(t)(\xi_1, \cdots, \xi_k) := \Pi_t[\cdots[m_t, \xi_1], \cdots, \xi_k].
\]

\item The coalgebra map
\[
U(t) : S\big(\mathfrak{a}(0)\big) \rightarrow S\big(\mathfrak{a}(t)\big), \quad t \in [0,1]
\]
is defined inductively by:
\begin{enumerate}[label=(\roman*)]
\item If $k = 1$, we define
\begin{equation}\nonumber
U^0(t) = 0; \quad U^1(t)(\xi_1) := \Pi_t\phi_t(\xi).
\end{equation}
\item If $k \geq 2$, we define
\begin{equation}\nonumber
\begin{aligned}
U^k&(t)(\xi_1, \cdots, \xi_k)\\
&:= \sum\limits_{\sigma \in S_k} \mathrm{sgn}(\sigma) \sum\limits_{m \geq 1} \sum\limits_{\mu_1 + \cdots + \mu_m = k-1} \frac{1}{(k m! \mu_1! \cdots \mu_m!)}\\
& \quad\quad\quad\quad\quad\quad \Pi_t[[\cdots[\phi_t(\xi_{\sigma(1)}), U^{\mu_1}(t)(\xi_{\sigma(2)}, \cdots, \xi_{\sigma(\mu_1+1)})], \cdots],\\
&\quad\quad\quad\quad\quad\quad\quad\quad\quad\quad U^{\mu_m}(t)(\xi_{\sigma(\mu_1 + \cdots +\mu_{m-1}+2)}, \cdots, \xi_{\sigma(\mu_1 + \cdots +\mu_{m}+1)})].
\end{aligned}
\end{equation}
\end{enumerate}
\end{enumerate}

We need the following lemma, whose proof can be found in \cite{CS}.
\begin{lem}\cite[Lemmata 3.3 \& 3.5]{CS}
\begin{enumerate}[label=(\roman*)]
\item $Q(t)$ satisfies the ordinary differential equation:
\[
\frac{dQ(t)}{dt} = M(t) \circ Q(t) - Q(t) \circ M(t).
\]
\item $U(t)$ satisfies the ordinary differential equation:
\[
\frac{dU(t)}{dt} = M(t) \circ U(t), \quad U(0) = \mathrm{id}_{S(\mathfrak{a}(0))}.
\]
\end{enumerate}
\end{lem}

\begin{cor}\label{indlm}
The coalgebra map
\[
U(1) : S(\mathfrak{a}(0)) \rightarrow S(\mathfrak{a}(1))
\]
is invertible and compatible with the codifferentials, that is, $U(1) \circ Q(0) = Q(1) \circ U(1)$. In other words, $U(1)$ determines an $L_{\infty}[1]$-isomorphism.
\end{cor}
\begin{proof}
Denote $Z(t) := Q(t) \circ U(t) - U(t) \circ Q(0)$ and observe that
\begin{equation}\nonumber
\frac{dZ(t)}{dt} = \cdots = M(t) \circ Z(t),
\end{equation}
and $Z(0) = 0$. Note that $Z^0(t) = 0$, since $U^0(t) = 0$. By the uniqueness of solutions of ordinary differential equations, $Z \equiv 0$ is the unique solution. Thus,
\[
Q(1) \circ U(1) = U(1) \circ Q(0),
\]
which proves that $U(t)$ is an $L_{\infty}[1]$-algebra morphism.
\end{proof}
We can restate the previous corollary as follows.
\begin{cor}\label{indliso}
For a smooth $1$-parameter family of V-algebras with a Maurer-Cartan element satisfying condition (\ref{pkkpt}), there exists an induced $L_{\infty}[1]$-isomorphism.
\end{cor}

\subsection{Curved $L_{\infty}[1]$-algebra structure on $\Gamma\left(\bigwedge\nolimits^{\bullet}NW\right)$}
In our geometric context, a V-algebra is realized on the section space of the exterior algebra of the vertical bundle of a presymplectic manifold $(W,\omega_W)$, giving rise to a curved (i.e., $l_0 \neq 0$ in general) $L_{\infty}[1]$-algebra associated to $W$.

Let $\pi\colon F \to W$ be a vector bundle over a presymplectic manifold $(W, \omega_W)$ and $\sigma : W \rightarrow F$ a smooth section of $\pi$. (Here, $F$ should not be confused with the obstruction bundle for a Kuranishi chart introduced in Section 2.1.) Let $VF \subset TF$ be the vertical tangent bundle, that is, for each $x \in W$, the fiber is $V_x F := \ker(d\pi_x)$. We write
\[
NW := VF|_{\sigma(W)}
\]
for the vector bundle
\[
VF|_{\sigma(W)} \rightarrow \sigma(W) \subset F \stackrel{\pi}{\to} W
\]
over $W$, which is canonically isomorphic to $F$.

We are interested in (i) the degree-shifted section space of the exterior algebra bundle $\bigwedge^{\bullet}NW$,
\begin{equation}\nonumber
A[1] := \Gamma\left(W, \bigwedge\nolimits^{\bullet}NW\right)[1] = \Gamma\left(W, \bigwedge\nolimits^{\bullet+1}NW\right),
\end{equation}
(ii) that of $\bigwedge^{\bullet}TNW$,
\begin{equation}\label{tnulim}
\Gamma\left(NW, \bigwedge\nolimits^{\bullet}TNW\right)[1] = \Gamma\left(NW, \bigwedge\nolimits^{\bullet+1}TNW\right),
\end{equation}
and (iii) the completion of (\ref{tnulim}) at the image $\sigma(W)$,
\begin{equation}\nonumber
\lim\limits_{\longleftarrow}\frac{\Gamma(NW, \bigwedge\nolimits^{\bullet+1}TNW)}{\big(I(W)|_{NW}\big)^n \cdot \Gamma(NW, \bigwedge\nolimits^{\bullet+1}TNW)},
\end{equation}
where $I(W) := \{f \in C^{\infty}(NW) \mid f \circ \sigma \equiv 0\}$.

\begin{lem}
We have
\begin{equation}\label{iotainvlim}
\widehat{S}_A(\mathrm{Der}(A)[-1])[1] \simeq \lim\limits_{\longleftarrow}\frac{\Gamma(NW, \bigwedge\nolimits^{\bullet+1}TNW)}{\big(I(W)|_{NW}\big)^n \cdot \Gamma(NW, \bigwedge\nolimits^{\bullet+1}TNW)}.
\end{equation}
\end{lem}

\begin{proof}[Sketch of proof]
We follow the arguments in \cite[Subsection 4.1]{CS}. First, note that we can regard $A := \Gamma\left(W, \bigwedge^{\bullet}NW\right)$ as the function algebra on $N^*[1]W$, that is, the dual vertical bundle with degree-shifted fibers. Hence $\widehat{S}_A\left(\mathrm{Der}(A)[-1]\right)[1]$ is the completed Gerstenhaber algebra of multivector fields on $N^*[1]W$, which is isomorphic to the completed Gerstenhaber algebra of multivector fields on $NW$ under the Legendre transform in its version studied in \cite{Roytenberg}. Finally, we can identify this with the right-hand side of (\ref{iotainvlim}), namely the Gerstenhaber algebra of multivector fields on the formal neighborhood of $W$ in $NW$.
\end{proof}

Let $P$ be a Poisson structure on $F$, i.e., $P \in \Gamma\left(F, \bigwedge\nolimits^{1}TF\right)[1]$ satisfying $[P,P]=0$ and the Jacobi identity. Fixing an embedding $\iota : NW \hookrightarrow F$ such that $\mathrm{Im}\,\sigma \subset NW$, we can readily see that there exists an isomorphism (still denoted by $\iota$) of graded Lie algebras:
\begin{equation}\nonumber
\iota : \lim\limits_{\longleftarrow}\frac{\Gamma(NW, \bigwedge\nolimits^{\bullet+1}TNW)}{\big(I(W)|_{NW}\big)^n \cdot \Gamma(NW, \bigwedge\nolimits^{\bullet+1}TNW)}
\xrightarrow{\simeq} \lim\limits_{\longleftarrow}\frac{\Gamma(F, \bigwedge\nolimits^{\bullet+1}TF)}{I(W)^n \cdot \Gamma(F, \bigwedge\nolimits^{\bullet+1}TF)},
\end{equation}
by virtue of the isomorphism $NW \simeq F$. Then $P$ determines an element on the right-hand side and hence on the left-hand side (still denoted by $P$), still satisfying $[P, P] = 0$ and the Jacobi identity.

We denote
\begin{enumerate}
\item[--] $\mathfrak{h} := \lim\limits_{\longleftarrow}\frac{\Gamma(NW, \bigwedge\nolimits^{\bullet+1}TNW)}{(I(W)|_{NW})^n \cdot \Gamma(NW, \bigwedge\nolimits^{\bullet+1}TNW)}$,
\item[--] $\mathfrak{a} := \Gamma\left(W, \bigwedge\nolimits^{\bullet+1}NW\right)$,
\item[--] $\Pi : \mathfrak{h} \rightarrow \mathfrak{a}$, the restriction to $W$ followed by the fiberwise projection map.
\end{enumerate}

\begin{prop}\label{afthap}
$(\mathfrak{h},\mathfrak{a},\Pi)$ as defined above is a V-algebra. Thus, with Maurer-Cartan element $P$ and Lemma \ref{valinf}, we obtain a curved $L_{\infty}[1]$-algebra $\{l^{P}_k\}_{k \geq 0}$.
\end{prop}

\begin{proof}
We can apply Lemma \ref{vallem}, or directly verify that the triple $(\mathfrak{h},\mathfrak{a},\Pi)$ satisfies the axioms in Definition \ref{Voronov1} as follows.
\begin{enumerate}[label=(\roman*)]
\item (\textit{$\mathfrak{h}$ is a graded Lie algebra over $\mathbf{k}$.}) We write $\Gamma$ for $\Gamma\left(NW, \bigwedge\nolimits^{\bullet+1}TNW\right)$ and $I$ for $I(W)$ for convenience. First, we show that the Nijenhuis-Schouten bracket $[\,,\,]_{\Gamma}$ on $\Gamma\left(NW, \bigwedge\nolimits^{\bullet+1}TNW\right)$ determines a bracket on $\mathfrak{h}$. For $j \geq 2$ and $\xi + I^j\Gamma,\, \xi' + I^j\Gamma \in \Gamma/I^j\Gamma$, we have
\begin{equation}\nonumber
\begin{aligned}
\Gamma / I^{j} \Gamma \otimes \Gamma / I^{j} \Gamma &\rightarrow \Gamma / I^{j-1} \Gamma,\\
(\xi + I^{j} \Gamma) \otimes (\xi' + I^{j} \Gamma) &\mapsto [\xi, \xi']_{\Gamma} + I^{j-1} \Gamma,
\end{aligned}
\end{equation}
which is well-defined because for other representative choices $\xi + \eta$ and $\xi' + \eta'$ with $\eta, \eta' \in I^j\Gamma$, we have
\begin{equation}\nonumber
[\xi + \eta, \xi' + \eta']_{\Gamma} - [\xi, \xi']_{\Gamma} = [\xi, \eta']_{\Gamma} + [\eta, \xi']_{\Gamma} + [\eta, \eta']_{\Gamma} \in I^{j-1}\Gamma,
\end{equation}
by the definition of the Nijenhuis-Schouten bracket. Moreover, such operations for all distinct $j$'s are compatible in the sense that the following diagram commutes.
\begin{equation}\nonumber
\begin{tikzcd}
\Gamma / I^{j} \Gamma \otimes \Gamma / I^{j} \Gamma  \arrow{r}{[\,,\,]_{\Gamma}} & \Gamma / I^{j-1} \Gamma\\
\Gamma / I^{j+1} \Gamma \otimes \Gamma / I^{j+1} \Gamma  \arrow{r}{[\,,\,]_\Gamma} \arrow{u}{p_{j+1,j} \otimes p_{j+1,j}} & \Gamma / I^{j} \Gamma. \arrow{u}{p_{j,j-1}}
\end{tikzcd}
\end{equation}
Here, $p_{j+1,j}$ is the canonical projection from $\Gamma/I^j\Gamma$ to $\Gamma/I^{j-1}\Gamma$ appearing in the inverse system $\left\{\Gamma/I^j\Gamma\right\}_j$. The axioms (e.g., bilinearity, antisymmetry, Jacobi identity, and compatibility with the grading) for a graded Lie algebra follow immediately from those for the bracket $[\,,\,]_{\Gamma}$.

\item (\textit{$\mathfrak{a}$ is an abelian Lie subalgebra of $\mathfrak{h}$.}) $\mathfrak{a}$ is equipped with the bracket $[\,,\,]_{\mathfrak{a}}$, given by the graded commutator for the natural multiplication of multivector fields $\bigwedge\nolimits^{\bullet+1}TNW$. Since the multiplication is graded commutative, the Schouten-Nijenhuis bracket vanishes for $0$-multivector fields, i.e., elements in $\Gamma\left(W, \bigwedge\nolimits^{\bullet+1}NW\right) \subset \Gamma\left(W, \bigwedge\nolimits^{\bullet+1}TNW\right)$. It is also straightforward to verify that the bracket $[\,,\,]$ from (i) restricts to $[\,,\,]_{\mathfrak{a}}$.

\item (\textit{$\ker\Pi$ is a Lie subalgebra of $\mathfrak{h}$.})
The Nijenhuis-Schouten bracket of multivectors with nonnegative degrees yields a multivector with nonnegative degree, which follows from two facts: (a) $\ker\Pi$ consists of linear combinations of elements having horizontal components of $TNW$; (b) differentiation in the horizontal direction preserves the vanishing of a function at $W$.

\item (\textit{$P$ on $\Gamma$ induces a Maurer-Cartan element (still denoted by $P$) on $\mathfrak{h}$.}) We consider
\begin{equation}\nonumber
P := \{P + I^n \in \Gamma/I^n\Gamma\}_n.
\end{equation}
We then have
\begin{equation}\nonumber
[P + I^n\Gamma, P + I^n\Gamma] = [P, P] + I^{n-1}\Gamma = 0 + I^{n-1}\Gamma.
\end{equation}
\end{enumerate}
\end{proof}

\subsection{Example from Gotay's embedding}
\cite{Gotay} proves that a presymplectic manifold can be embedded as a coisotropic submanifold in a symplectic manifold. The foliation cotangent bundle arising from a presymplectic structure provides an interesting example via this theorem, which was studied in \cite{OP} using physics-inspired methods. Indeed, we can reformulate their results using V-algebras.

Let $(W^n, \omega_W)$ be a presymplectic manifold. We consider the distribution
\[
T\mathcal{F} := \ker \omega_W \subset TW.
\]
It follows readily from the closedness of $\omega_W$ that $T\mathcal{F}$ is involutive, hence integrable by the Frobenius theorem.

Note that we can choose a splitting of $TW$, that is, a vector bundle $G$ satisfying
\begin{equation}\label{wdecomp}
TW = T\mathcal{F} \oplus G.
\end{equation}

Let $(y_1, \cdots, y_k, q_1, \cdots, q_{n-k})$ be local coordinates of $x$ in $W$, where $(q_1, \cdots, q_{n-k})$ are the foliation coordinates, that is, $y_i = c_i$, $i = 1, \cdots, k$, form the plaque for the foliation near $x$. In these coordinates, we have
\begin{equation}\label{txfgxq}
\begin{split}
T_x\mathcal{F} &= \mathrm{span} \Bigg\{ \frac{\partial}{\partial q_1}, \cdots, \frac{\partial}{\partial q_{n-k}} \Bigg\}, \\
G_x &= \mathrm{span} \Bigg\{ \frac{\partial}{\partial y_i} + \sum\limits_{\alpha=1}^{m} R^{\alpha}_i \frac{\partial}{\partial q_{\alpha}} \Bigg\}_{1 \leq i \leq k}
\end{split}
\end{equation}
for some functions $R^{\alpha}_i$ in $y_i$'s and $q_{\alpha}$'s. Here, $R^{\alpha}_i$ can be regarded as the Christoffel symbol for the `connection' determined by the decomposition (\ref{wdecomp}).

\begin{exam}\label{gtyemb}
We present an example arising from \cite{Gotay}: Any presymplectic manifold can be coisotropically embedded into a symplectic manifold. Let $T^*\mathcal{F} \to W$ be the foliation cotangent bundle, that is, the dual bundle to the foliation tangent bundle arising from the involutive distribution $T\mathcal{F} \subset TW$. Gotay's theorem is realized by the vector bundle $F := T^*\mathcal{F}$ equipped with the symplectic form
\begin{equation}\label{otf}
\omega_{T^*\mathcal{F}} := \pi^*\omega_W - d\theta,
\end{equation}
where $\theta$ is the canonical one-form. It is straightforward to show that $\omega_{T^*\mathcal{F}}$ is nondegenerate, hence a symplectic form. Gotay's theorem states that on $T^*\mathcal{F}$ there exists a coisotropic embedding
\begin{equation}\nonumber
\sigma : (W, \omega_{W}) \hookrightarrow (T^*\mathcal{F}, \omega_{T^*\mathcal{F}}),
\end{equation}
such that $\sigma(W)$ coincides with the zero section of $T^*\mathcal{F}$.

With respect to the symplectic structure from $\omega_{T^*\mathcal{F}}$, we obtain a Poisson structure $P \in \Gamma\left(T^*\mathcal{F}, \bigwedge\nolimits^2 TT^*\mathcal{F}\right)$, i.e., a bivector field $P \in \Gamma\left(F, \bigwedge\nolimits^{2}TF\right)$ satisfying $[P,P]=0$ for the Nijenhuis-Schouten bracket $[\,,\,]$. In local coordinates, it is written as
\begin{equation}\label{pico}
P = \frac{1}{2} \sum\limits_{i,j}\omega^{ij} e_i \wedge e_j + \sum\limits_{\alpha}\frac{\partial}{\partial q^{\alpha}} \wedge \frac{\partial}{\partial p_{\alpha}},
\end{equation}
where
\begin{equation}\label{eppy}
e_j := \frac{\partial}{\partial y_j} + \sum\limits_{\alpha} R_j^{\alpha} \frac{\partial}{\partial q^{\alpha}} - \sum\limits_{\beta, \nu} p_{\beta} \frac{\partial R^{\beta}_j}{\partial q^{\nu}} \frac{\partial}{\partial p_{\nu}},
\end{equation}
with $R^{\alpha}_j$ from (\ref{txfgxq}). We refer the reader to Sections 7 through 9 of \cite{OP} for a detailed analysis.

For the zero section $\sigma \equiv 0$ of $T^*\mathcal{F}$, there exists a canonical decomposition
\[
T_{(x,0)}T^*F = T_xW \oplus T^*_xF
\]
at $x \in W$ into the horizontal and vertical components. Then we have
\[
NW = \bigcup\limits_{x \in W} T^*_x\mathcal{F} = T^*\mathcal{F}.
\]
In this case, $\mathfrak{h}$, $\mathfrak{a}$, and $\Pi$ in Proposition \ref{afthap} for the V-algebra are identified as follows:
\begin{equation}\nonumber
\begin{split}
\mathfrak{h} &:=
\lim\limits_{\longleftarrow}\frac{\Gamma(T^*\mathcal{F}, \bigwedge\nolimits^{\bullet+1}TT^*\mathcal{F})}{(I(W)|_{T^*\mathcal{F}})^n \cdot \Gamma(T^*\mathcal{F}, \bigwedge\nolimits^{\bullet+1}TT^*\mathcal{F})},\\
\mathfrak{a} &:= \Gamma\left(W; \bigwedge\nolimits^{\bullet+1}NW\right) = \Gamma\left(W; \bigwedge\nolimits^{\bullet+1}T^*\mathcal{F}\right) = \Omega^{\bullet+1}\left(\mathcal{F}\right),
\end{split}
\end{equation}
and $\Pi$ is the projection onto the subspace of $\mathfrak{h}$ generated by elements of the form $\frac{\partial}{\partial p_{i_1}} \wedge \cdots \wedge \frac{\partial}{\partial p_{i_l}}$ for $i_1, \cdots, i_l$ and $l \geq 1$, followed by evaluation at $p_i = 0$ for all $i$. With a choice of Poisson structure, we obtain an $L_{\infty}[1]$-algebra structure on $\mathfrak{a} = \Omega^{\bullet+1}(\mathcal{F})$, that is, on the degree-shifted foliation de Rham complex, by Proposition \ref{afthap}. We write $\left\{l^{\mathcal{F}}_k\right\}_{k \geq 0}$ for the resulting $L_{\infty}[1]$-algebra.
\end{exam}

\begin{lem}\label{liiifdr} We have:
\begin{enumerate}[label=(\roman*)]\label{strlinf}
\item $\left\{l^{\mathcal{F}}_k\right\}$ is strict, i.e., $l^{\mathcal{F}}_0 = 0$.
\item $l_1^{\mathcal{F}}$ coincides with the foliation de Rham differential $d_{\mathcal{F}}$.
\item For the zero presymplectic form, i.e., the case when $T\mathcal{F} = TU$, $l^{\mathcal{F}}_1$ is the ordinary de Rham differential and all $l^{\mathcal{F}}_k$ with $k \geq 2$ vanish.
\item For different choices of the splitting (\ref{wdecomp}), we obtain \textit{isomorphic} $L_{\infty}[1]$-algebras.
\item For different choices of local coordinate system, we obtain \textit{isomorphic} $L_{\infty}[1]$-algebras.
\end{enumerate}
\end{lem}

\begin{proof}
\begin{enumerate}[label=(\roman*)]
\item We have
\begin{equation}\nonumber
\begin{split}
l_0^{\mathcal{F}}(1) &= \Pi P = \Pi \left(\sum\limits_{i,j}\frac{1}{2} \widetilde{\omega}^{ij} e_i \wedge e_j + \sum\limits_{\alpha}\frac{\partial}{\partial q^{\alpha}} \wedge \frac{\partial}{\partial p_{\alpha}}\right)\\
&= \sum\limits_{i,j,\beta,\gamma,\mu,\nu}\frac{1}{2}\widetilde{\omega}^{ij} p_{\beta} p_{\gamma} \frac{\partial R^{\beta}_i}{\partial q^{\nu}} \frac{\partial R^{\gamma}_j}{\partial q^{\mu}}\frac{\partial}{\partial p_{\nu}} \wedge \frac{\partial}{\partial p_{\mu}} \bigg|_{\vec{p} = 0} = 0.
\end{split}
\end{equation}
\item For $\xi = \sum\limits_{\alpha} a_{\alpha} \frac{\partial}{\partial p_{\alpha}} \in \Gamma(T^*\mathcal{F})$ with $a_{\alpha} = a_{\alpha}(\vec{y},\vec{q}) \in C^{\infty}(T^*\mathcal{F})$, we have
\begin{equation}\nonumber
\begin{split}
\Pi [P, \xi] & = \Pi \left[\sum\limits_{i,j}\frac{1}{2} \widetilde{\omega}^{ij} e_i \wedge e_j + \sum\limits_{\alpha}\frac{\partial}{\partial q^{\alpha}} \wedge \frac{\partial}{\partial p_{\alpha}},\sum\limits_{\alpha'}a_{\alpha'} \frac{\partial}{\partial p_{\alpha'}}\right]\\
& = \Pi \left[\sum\limits_{i,j}\frac{1}{2} \widetilde{\omega}^{ij} e_i \wedge e_j ,\sum\limits_{\alpha'}a_{\alpha'} \frac{\partial}{\partial p_{\alpha'}}\right] + \Pi \left[\sum\limits_{\alpha} \frac{\partial}{\partial q^{\alpha}} \wedge \frac{\partial}{\partial p_{\alpha}},\sum\limits_{\alpha'}a_{\alpha'} \frac{\partial}{\partial p_{\alpha'}}\right]\\
& = \sum\limits_{\alpha, \alpha'} \left[ \frac{\partial}{\partial q_{\alpha}}, a_{\alpha'} \frac{\partial}{\partial p_{\alpha'}} \right] \wedge \frac{\partial}{\partial p_{\alpha}} = \sum\limits_{\alpha, \alpha'} \frac{\partial a_{\alpha'}}{\partial q^{\alpha}} \frac{\partial}{\partial p_{\alpha'}} \wedge \frac{\partial}{\partial p_{\alpha}} = d_{\mathcal{F}}(\xi).
\end{split}
\end{equation}
\item This follows directly from the observation that in (\ref{otf}), only the term $d\theta$ survives. Consequently, the Poisson structure given by (\ref{pico}) reduces to
\[
\sum\limits_{\alpha = 1}^{\dim W} \frac{\partial}{\partial q^{\alpha}} \wedge \frac{\partial}{\partial p_{\alpha}}.
\]
Moreover, for $a_{\alpha'} = a_{\alpha'}(\vec{q})$, we have
\[
\begin{split}
l^{\mathcal{F}}_1\left(\sum\limits_{\alpha'}a_{\alpha'} \frac{\partial}{\partial p_{\alpha'}}\right) &= \Pi\left[\sum\limits_{\alpha=1}^{\dim W} \frac{\partial}{\partial q^{\alpha}} \wedge \frac{\partial}{\partial p_{\alpha}}, \sum\limits_{\alpha'}a_{\alpha'} \frac{\partial}{\partial p_{\alpha'}}\right]\\
&= \sum\limits_{\alpha=1}^{\dim W} \sum\limits_{\alpha'} \frac{\partial a_{\alpha'}}{\partial q^{\alpha}} \frac{\partial}{\partial p_{\alpha}} \wedge \frac{\partial}{\partial p_{\alpha'}} = d\left( \sum\limits_{\alpha'}a_{\alpha'} \frac{\partial}{\partial p_{\alpha'}} \right).
\end{split}
\]
On the other hand, all higher-order repeated brackets vanish, as follows.
\[
\begin{split}
l^{\mathcal{F}}_k &\left(\sum\limits_{\alpha'}a_{\alpha'} \frac{\partial}{\partial p_{\alpha'}}, \sum\limits_{\alpha''}a_{\alpha''} \frac{\partial}{\partial p_{\alpha''}}, \cdots \right)\\
&= \Pi\left[ \cdots \left[ \left[\sum\limits_{\alpha=1}^{\dim W} \frac{\partial}{\partial q^{\alpha}} \wedge \frac{\partial}{\partial p_{\alpha}}, \sum\limits_{\alpha'}a_{\alpha'} \frac{\partial}{\partial p_{\alpha'}}\right], \sum\limits_{\alpha''}a_{\alpha''} \frac{\partial}{\partial p_{\alpha''}}\right] \cdots \right]\\
&= \Pi \left[ \cdots \left[ \sum\limits_{\alpha=1}^{\dim W} \sum\limits_{\alpha'} \frac{\partial a_{\alpha'}}{\partial q^{\alpha}} \frac{\partial}{\partial p_{\alpha}} \wedge \frac{\partial}{\partial p_{\alpha'}}, \sum\limits_{\alpha''}a_{\alpha''} \frac{\partial}{\partial p_{\alpha''}} \right] \cdots \right] = 0,
\end{split}
\]
where the last equality holds since $a_{\alpha''}$ is independent of the $\vec{p}$-variables.
\item The splitting (\ref{wdecomp}) affects only the Poisson structure. If we connect two V-algebras with Poisson structures via a $1$-parameter family
\[
\big( (\mathfrak{h}, \mathfrak{a}, \Pi), P_0 \big) \rightsquigarrow
\big( (\mathfrak{h}, \mathfrak{a}, \Pi), P_1 \big)
\]
that preserves the underlying $(\mathfrak{h}, \mathfrak{a}, \Pi)$, then the induced $L_{\infty}[1]$-isomorphism from Corollary \ref{indliso} yields the desired result. Note that condition (\ref{pkkpt}) is satisfied trivially in this case.
\item Following the same approach as in the proof of (iv), we obtain an isomorphic family of V-algebras. The result then follows by applying Corollary \ref{indliso}.
\end{enumerate}
\end{proof}
\begin{notation}
From now on, we write $\left\{l^{\mathcal{F}}_k\right\}_{k \geq 1}$ for this (strict) $L_{\infty}[1]$-algebra, omitting $\mathcal{F}$ when the context is clear.
\end{notation}

In our subsequent discussion of $L_{\infty}$-Kuranishi spaces, the following lemma plays a crucial role.
\begin{lem}[Poincar\'{e} lemma for foliation de Rham complexes]\cite[Theorem 4.1]{MS}\label{pcrlem}
Let $T\mathcal{F}$ be a regular foliation on a simply connected manifold $W$. Then the Poincar\'{e} lemma holds for the foliation de Rham complex $\Omega^*(\mathcal{F})$: if $\xi \in \Omega^{* \geq 1}(\mathcal{F})$ is closed, i.e., $d_{\mathcal{F}}(\xi) = 0$, then there exists $\gamma \in \Omega^*(\mathcal{F})$ such that $d_{\mathcal{F}}(\gamma) = \xi$.
\end{lem}
\begin{proof}
Consider the projection maps
\[
W \xleftarrow{\pi_W} W \times [0,1] \xrightarrow{\pi_{[0,1]}} [0,1]
\]
and the foliation tangent bundle
\[
T\overline{\mathcal{F}} := \pi_{W}^* T\mathcal{F} \oplus \pi_I^* T[0,1]
\]
on $W \times [0,1]$, where we regard $[0,1]$ as equipped with the zero differential form.
We define a map
\[
W \times [0,1] \xrightarrow{p} W
\]
by
\[
(y_1,\cdots, y_{n-k}, q_1, \cdots, q_k, t) \mapsto (y_1,\cdots, y_{n-k}, tq_1, \cdots, tq_k),
\]
and the following induced maps between foliation differential forms
\[
\begin{split}
\Omega^*(\mathcal{F}) \rightarrow \Omega^*\left(\overline{\mathcal{F}}\right) \rightarrow \Omega^*(\mathcal{F}),\\
\xi \;\mapsto\; p^*\xi \;\mapsto\; \int_0^1 p^*\xi\, dt,
\end{split}
\]
where the second map vanishes, by definition, for $\eta \in \Omega^*\left(\overline{\mathcal{F}}\right)$ such that $\eta\!\left(\tfrac{\partial}{\partial t}\right) = 0$.

Denote $p_t := p(\cdot, t)$ and
\[
V_t := \frac{d}{ds} \left(p_{t+s} \circ p_t^{-1}\right) \bigg|_{s=0}.
\]
Then it is straightforward to verify that $V_t$ necessarily lies in the foliation directions. Computing $\frac{d}{ds}\big(p_{t+s}\circ p_t^{-1}\big)^*\xi\big(p_t(x)\big)|_{s=0}$ for $x \in \Omega^*(\mathcal{F})$, we can show that
\[
\frac{d}{dt} p_t^* \xi = p_t^* L_{V_t} \xi,
\]
where $V_t$ is the tangent vector field along $p_t$. Applying the Cartan magic formula (for the foliation differentiation) and integrating both sides over $[0,1]$, we obtain the homotopy formula:
\[
p_{t=1}^* \xi - p_{t=0}^* \xi
= \int_0^1 p_t^*\iota_{V_t}(d_{\mathcal{F}}\xi)\, dt
+ d_{\mathcal{F}} \left( \int_0^1 p_t^* \xi\, dt \right).
\]
Note that $p_{t=0}^* \xi = 0$, since
\[
(y_1,\ldots,y_{n-k},0,\ldots,0)
\]
has no foliation coordinates, while $p_{t=1}^* \xi = \xi$ and $p_{t=1} = \mathrm{id}_W$. Thus, for $d_{\mathcal{F}}$-closed $\xi$, we obtain
\[
\xi = d_{\mathcal{F}}\left( \int_0^1 p_t^* \xi\, dt \right).
\]
\end{proof}

\begin{defn}[Augmented foliation de Rham complex]\label{obpoaug}
We further consider the foliation de Rham complex \textit{with augmentation,} defined by
\begin{equation}\nonumber
\Omega^{\bullet+1}_{\mathrm{aug}}(\mathcal{F}) :=
\begin{cases}
\Omega^{\bullet+1}(\mathcal{F}) &\text{if } \bullet \geq -1,\\
C^{\infty}(W)_{\mathcal{F}} := \{h \in C^{\infty}(W) \mid d_{\mathcal{F}}(h) = 0\} &\text{if } \bullet = -2,
\end{cases}
\end{equation}
whose differential is given by $d_{\mathcal{F}}$ for elements of degree $\geq -1$ and by the inclusion $C^{\infty}(W)_{\mathcal{F}} \hookrightarrow C^{\infty}(W)$ for those of degree $-2$.
\end{defn}

\begin{prop}\label{augomega}
In the above situation, there exists an $L_{\infty}[1]$-algebra structure on the chain complex $\Omega^{\bullet+1}_{\mathrm{aug}}(\mathcal{F})$ that extends $\{l^{\mathcal{F}}_k\}$ on $\Omega^{\bullet+1}(\mathcal{F})$.
\end{prop}

\begin{proof}
Let $k$ be the rank of the foliation tangent bundle $T\mathcal{F}$. Considering the operation $l_k^{\mathcal{F}}$, we denote
\begin{equation}\nonumber
\begin{cases}
m := \text{the number of inputs,}\\
d := \text{the number of degree } {-2} \text{ inputs,}\\
s := \text{the sum of degrees of all inputs.}
\end{cases}
\end{equation}
Observe that $d$ and $s$ must lie in the ranges:
\begin{equation}\nonumber
\begin{cases}
1 \leq d \leq k,\\
s \geq -2d + (k-1)(m-d),
\end{cases}
\end{equation}
respectively. We proceed by induction on the tuple $(m,d,s)$:
\begin{enumerate}[label=(\roman*)]
\item $(m,d,s) = (1,1,-2)$: We have
\[
l_1(g) =: \overline{g} \in C^{\infty}(W_x)
\]
for $g \in C^{\infty}(W)_{\mathcal{F}}$, that is, $|g| = -2$.
\item $(m,d,s) = (2,1,k-3)$: As an induction hypothesis, we assume that $l_2$ has been defined by
\[
l_2(g, \xi) := 0
\]
for $|\xi| = k-1$.
\item $(m,d,s) = (2,1,s')$: Suppose that we have defined $l_2$'s for all $s'+1 \leq s \leq k-3$. For $g$ and $\xi$ with $|g| = -2$, $|\xi| \geq -1$, we denote
\[
A(g, \xi) := l^{\mathcal{F}}_2\big(l_1(g), \xi\big) + (-1)^{|g|}l_2\big(g, l_1(\xi)\big) = l_2(\overline{g}, \xi) + l_2\big(g, d_{\mathcal{F}}(\xi)\big),
\]
which is written in terms of quantities appearing in the initial conditions of the earlier induction steps together with the $L_{\infty}[1]$-relations for them. When $|A(g, \xi)| \geq -1$, we have
\begin{equation}\nonumber
\begin{split}
d_{\mathcal{F}}A(g, \xi) &=: l_1\big(A(g, \xi)\big) = d_{\mathcal{F}}l_2(\overline{g}, \xi) + l_1 \circ l_2(g, d_{\mathcal{F}}\xi)\\
&= -l_2(d_{\mathcal{F}}\overline{g}, \xi) + l_2(\overline{g}, d_{\mathcal{F}}\xi) - l_2(\overline{g}, d_{\mathcal{F}}\xi) - l_2(\overline{g}, d_{\mathcal{F}}^2\xi) = 0.
\end{split}
\end{equation}
By the foliation Poincar\'{e} lemma, there exists $B(g,\xi)$ such that $A(g,\xi) = d_{\mathcal{F}}B(g,\xi)$. We define
\begin{equation}
l_2(g, \xi) := \begin{cases}
-B(g, \xi) &\text{if } |A(g, \xi)| \geq -1,\\
0 &\text{otherwise.}
\end{cases}
\end{equation}
\item Suppose that we have defined the following two cases:
\begin{equation}\nonumber
\begin{cases}
l_{m} \text{ for } m \leq m'-1,\\
l_{m} \text{ for } m = m', d \leq d'-1, s'+1 \leq s
\end{cases}
\end{equation}
with the initial condition $l_m(\cdots) := 0$ for $d = d'+1$ and $s = -2d + (k-1)(m-d)$.
Then it suffices to define $l_m(g_1, \cdots, g_{d'}, \xi_1, \cdots, \xi_{m-d'})$ for $g_1, \cdots, g_{d'} \in \Omega^{-2}(\mathcal{F})[1]$ and $\xi_1, \cdots, \xi_{m-d'} \in \Omega^{\geq -1}(\mathcal{F})[1]$ with $|\xi_1| + \cdots + |\xi_{m-d'}| = s'$. We write $\overline{g_i} \in C^{\infty}(W) = \Omega^0(\mathcal{F})$ for the image of $g_i$ under the inclusion $C^{\infty}_{\mathcal{F}}(W) \hookrightarrow C^{\infty}(W)$. We denote
\begin{equation}\label{ag4}
\begin{split}
A(g_1, &\cdots, g_{d'}, \xi_1, \cdots, \xi_{m-d'}) := \sum\limits_i (-1)^{i-1}l_m(g_1, \cdots, \overline{g_i}, \cdots, g_{d'}, \xi_1, \cdots, \xi_{m-d'})\\
& + \sum\limits_j (-1)^{d'+j-1}l_m(g_1, \cdots, \overline{g_i}, \cdots, g_{d'}, \xi_1, \cdots, d_{\mathcal{F}}\xi_j, \cdots, \xi_{m-d'})\\
& + \sum\limits_{m_1+m_2=m+1} l_{m_1}\big(l_{m_2}(g_1, \cdots, g_*, \xi_*, \cdots, \xi_*), g_*, \cdots, g_*, \xi_*, \cdots, \xi_{m-d'}\big).
\end{split}
\end{equation}
The terms on the right-hand side are all known either from the initial condition or from the earlier induction steps. For the case $|A(g_1, \cdots, \xi_{m-d'})| \geq -1$, it follows directly that 
\[
d_{\mathcal{F}}A(g_1, \cdots, g_{d'}, \xi_1, \cdots, \xi_{m-d'}) = 0
\]
from the fact that $d_{\mathcal{F}}(\overline{g_i}) = 0$, $i = 0, \cdots, k$, and the $L_{\infty}[1]$-relations for $l_{*}(\cdots)$'s from the earlier steps. Then the Poincar\'{e} lemma implies that there exists $B(g_1, \cdots, g_{d'}, \xi_1, \cdots, \xi_{m-d'})$ such that
\[
A(g_1, \cdots, g_{d'}, \xi_1, \cdots, \xi_{m-d'}) = d_{\mathcal{F}}B(g_1, \cdots, g_{d'}, \xi_1, \cdots, \xi_{m-d'}).
\]
We define
\begin{equation}\label{ssttrr}
l_m(g_1, \cdots, g_{d'}, \xi_1, \cdots, \xi_{m-d'}) := \begin{cases}
-B(g_1, \cdots, g_{d'}, \xi_1, \cdots, \xi_{m-d'})&\\
\quad\quad\quad\quad \text{if } |A(g_1, \cdots, \xi_{m-d'})| \geq -1,&\\
0 \quad\quad\quad\ \text{otherwise.}&
\end{cases}
\end{equation}
The above induction process yields the desired $L_{\infty}[1]$-algebra structure on $\Omega_{\mathrm{aug}}^{\bullet+1}(\mathcal{F})$. In particular, the $L_{\infty}[1]$-relation holds by construction.
\end{enumerate}
\end{proof}
We remark that the cohomology of $\Omega_{\mathrm{aug}}^{\bullet+1}(\mathcal{F})$ is trivial, but its $L_{\infty}[1]$-structure is not.

\subsection{Completed V-algebras}
The $L_{\infty}[1]$-algebra in the previous subsection yields another $L_{\infty}[1]$-algebra arising from the V-algebra \textit{completed} at the image of a smooth map to the base $W$.

Let $\phi : V \rightarrow W$ be a smooth map between manifolds. We consider an ideal of $C^{\infty}(W)$,
\begin{equation}\nonumber
I_{\phi} := \{f \in C^{\infty}(W) \mid f|_{\mathrm{Im}\,\phi} \equiv 0\}
\end{equation}
and denote
\[
C^{\infty}(W)^{(j)} := C^{\infty}(W)/I_{\phi}^j \cdot C_{\phi}^{\infty}(W) = C^{\infty}(W)/I_{\phi}^j
\]
with
\begin{equation}\label{invphi32}
\begin{split}
C_{\phi}^{\infty}(W) &:= \lim_{\longleftarrow} C^{\infty}(W)^{(j)}\\
&= \lim\limits_{\longleftarrow} \left\{ \cdots \xrightarrow{p_{3,2}}
C^{\infty}(W)^{(2)} \xrightarrow{p_{2,1}}
C^{\infty}(W)^{(1)} \right\}.
\end{split}
\end{equation}
Here $p_{j+1,j} : C^{\infty}\left(W'_{\phi(x)}\right)^{(j+1)} \rightarrow C^{\infty}\left(W'_{\phi(x)}\right)^{(j)}$ is the canonical projection for each $j \geq 1$.

\begin{defn}[Completed V-algebras]\label{locvalg}
For the V-algebra $\mathcal{V} = (\mathfrak{h},\mathfrak{a},\Pi)$ with Maurer-Cartan element $P$ from Proposition \ref{afthap} and its preceding paragraph, we define its \textit{completed V-algebra at $\phi$} as the tuple
\begin{equation}\nonumber
\mathcal{V}_{\phi} := \left(\mathfrak{h}_{\phi}, \mathfrak{a}_{\phi}, \Pi_{\phi}\right),
\end{equation}
where
\begin{equation}\nonumber
\begin{cases}
\mathfrak{h}_{\phi} := C_{\phi}^{\infty}(W) \otimes_{C^{\infty}(W)} \mathfrak{h},\\
\mathfrak{a}_{\phi} := C_{\phi}^{\infty}(W) \otimes_{C^{\infty}(W)} \mathfrak{a},\\
\Pi_{\phi} := \mathrm{id}_{C_{\phi}^{\infty}(W)} \otimes_{C^{\infty}(W)} \Pi,
\end{cases}
\end{equation}
with the Maurer-Cartan element $P_{\phi} := 1 \otimes_{C^{\infty}(W)} P$. Here $C^{\infty}(W)$ acts on the modules $C^{\infty}_{\phi}(W)$, $\mathfrak{h}$, and $\mathfrak{a}$ in the obvious way.
\end{defn}

\begin{lem}\label{vlocp}
$\mathcal{V}_{\phi} := \left(\mathfrak{h}_{\phi}, \mathfrak{a}_{\phi}, \Pi_{\phi}\right)$ is a V-algebra with Maurer-Cartan element $P_{\phi}$.
\end{lem}

\begin{proof}
(i) (\textit{$\mathfrak{h}_{\phi}$ is a graded Lie algebra.}) We first define the bracket $[\,,\,]$ on $\mathfrak{h}_{\phi}$ by
\begin{equation}\nonumber
\begin{aligned}
\left(C^{\infty}(W)^{(j)} \otimes \mathfrak{h}\right) \otimes_{\mathbf{k}} \left(C^{\infty}(W)^{(j)} \otimes \mathfrak{h}\right) &\rightarrow C^{\infty}(W)^{(j-1)} \otimes \mathfrak{h}, \\
(b \otimes \xi) \otimes_{\mathbf{k}} (b' \otimes \xi') &\mapsto [1]_{j-1} \otimes \left[\widetilde{b}\xi, \widetilde{b'}\xi'\right]
\end{aligned}
\end{equation}
for each $j \geq 2$, where $\widetilde{b}, \widetilde{b'}$ are representatives of some classes in $C^{\infty}(W)^{(j)}$ and $\xi, \xi' \in \mathfrak{h}$. Its well-definedness can be shown as follows. For different choices of representatives $\widetilde{b} + c$ and $\widetilde{b'} + c'$ with $c, c' \in I_{\phi}^j$, we have
\begin{equation}\nonumber
\begin{split}
[1]_{j-1} \otimes &\left( \left[ \left( \widetilde{b} + c \right) \xi, \left( \widetilde{b'} + c' \right) \xi' \right] - \left[ b\xi, b'\xi' \right] \right)\\
&= [1]_{j-1} \otimes \left[ \widetilde{b}\xi, c'\xi' \right] + [1]_{j-1} \otimes \left[c\xi, \widetilde{b'}\xi' \right] + [1]_{j-1} \otimes \left[ c\xi, c'\xi' \right] = 0
\end{split}
\end{equation}
as a simple consequence of the definition of the Nijenhuis-Schouten bracket and the fact that $c, c' \in I_{\phi}^j \subset I_{\phi}^{j-1}$. Moreover, such operations for all different $j$'s are compatible in the sense that the following diagram commutes for each $j \geq 2$:
\begin{equation}\nonumber
\begin{tikzcd}
(C^{\infty}(W)^{(j)} \otimes \mathfrak{h}) \otimes_{\mathbf{k}} (C^{\infty}(W)^{(j)} \otimes \mathfrak{h}) \arrow{r}{[\,,\,]} & C^{\infty}(W)^{(j-1)} \otimes \mathfrak{h} \\
(C^{\infty}(W)^{(j+1)} \otimes \mathfrak{h}) \otimes_{\mathbf{k}} (C^{\infty}(W)^{(j+1)} \otimes \mathfrak{h}) \arrow{r}{[\,,\,]} \arrow{u}{(p_{j+1,j} \otimes \mathrm{id}_{\mathfrak{h}}) \otimes (p_{j+1,j} \otimes \mathrm{id}_{\mathfrak{h}})} & C^{\infty}(W)^{(j)} \otimes \mathfrak{h}. \arrow{u}{p_{j,j-1} \otimes \mathrm{id}_{\mathfrak{h}}}
\end{tikzcd}
\end{equation}
The axioms for a graded Lie algebra follow immediately from those for $\mathfrak{h}$. We only verify the Jacobi identity, which is less immediate.

We consider the repeated bracket:
\[
[\,[\,,\,],\,] : \left((C^\infty(W)^{(j)} \otimes \mathfrak{h}) \otimes (C^\infty(W)^{(j)} \otimes \mathfrak{h})\right) \otimes (C^\infty(W)^{(j-1)} \otimes \mathfrak{h})
\rightarrow C^\infty(W)^{(j-2)} \otimes \mathfrak{h},
\]
\[
\begin{split}
[[h_1 \otimes \xi_1, h_2 \otimes \xi_2], h_3 \otimes \xi_3]
&= \left[[1]_{j-1} \otimes \left[\widetilde{h}_1\xi_1, \widetilde{h}_2\xi_2\right], h_3 \otimes \xi_3\right]\\
&= [1]_{j-2} \otimes \left[\left[\widetilde{h}_1\xi_1, \widetilde{h}_2\xi_2\right], \widetilde{h}'_3\xi_3\right]
\end{split}
\]
for representatives $\widetilde{h}_1, \widetilde{h}_2$, and $\widetilde{h}_3$ in $C^\infty(W)$ with $h_1 = \left[\widetilde{h}_1\right]_{j}$, $h_2 = \left[\widetilde{h}_2\right]_{j}$, and $h_3 = \left[\widetilde{h}_3\right]_{j-1}$.

Observe that for the pair $h_3 \overset{p_{j,j-1}}{\mapsto} h'_3$, their representatives $\widetilde{h}_3$ and $\widetilde{h}'_3$ are related by $\widetilde{h}'_3 = \widetilde{h}_3 + g$ for some $g \in I_\phi^{j-1}$. Thus we can insert $\widetilde{h}_3 + g \in C^\infty(W)$ in place of $\widetilde{h}'_3$. We then compute:
\[
\begin{split}
[[h_1 \otimes \xi_1, &h_2 \otimes \xi_2], h_3 \otimes \xi_3] = [1]_{j-2} \otimes \left[\left[\widetilde{h}_1\xi_1, \widetilde{h}_2\xi_2\right], \widetilde{h}'_3\xi_3\right]\\
&= [1]_{j-2} \otimes \left[\left[\widetilde{h}_1\xi_1, \widetilde{h}_2\xi_2\right], \widetilde{h}_3\xi_3\right]
+ [1]_{j-2} \otimes \left[\left[\widetilde{h}_1\xi_1, \widetilde{h}_2\xi_2\right], g_3\xi_3\right]\\
&= [1]_{j-2} \otimes \left[\left[\widetilde{h}_1\xi_1, \widetilde{h}_2\xi_2\right], \widetilde{h}_3\xi_3\right]
\end{split}
\]
for $g_3 \in I_\phi^{j-1}$.

Finally, we have
\[
\begin{split}
\sum_{\sigma \in S_3} &\mathrm{sgn}(\sigma)
\left[\left[h_{\sigma(1)} \otimes \xi_{\sigma(1)}, h_{\sigma(2)} \otimes \xi_{\sigma(2)}\right],
h_{\sigma(3)} \otimes \xi_{\sigma(3)}\right]\\
&= \sum_{\sigma \in S_3} [1]_{j-2} \otimes
\mathrm{sgn}(\sigma)
\left[\left[\widetilde{h}_{\sigma(1)}\xi_{\sigma(1)}, \widetilde{h}_{\sigma(2)}\xi_{\sigma(2)}\right],
\widetilde{h}_{\sigma(3)}\xi_{\sigma(3)}\right]\\
&= [1]_{j-2} \otimes
\sum_{\sigma \in S_3} \mathrm{sgn}(\sigma)
\left[\left[\widetilde{h}_{\sigma(1)}\xi_{\sigma(1)}, \widetilde{h}_{\sigma(2)}\xi_{\sigma(2)}\right],
\widetilde{h}_{\sigma(3)}\xi_{\sigma(3)}\right] = 0.
\end{split}
\]

(ii) (\textit{$\mathfrak{a}_{\phi}$ is an abelian Lie subalgebra of $\mathfrak{h}_{\phi}$.}) We have $\mathfrak{a}_{\phi} \subset \mathfrak{h}_{\phi}$, with $\mathfrak{a}_{\phi}$ being abelian as an immediate consequence of the definition of the bracket and the abelian property of $\mathfrak{a}$.

(iii) (\textit{$\ker\Pi_{\phi}$ is a Lie subalgebra of $\mathfrak{h}_{\phi}$.}) Observe that $\ker\Pi_{\phi} \simeq C^{\infty}(W) \otimes \ker\Pi$, and one can use the fact that $\ker\Pi \subset \mathfrak{h}$ is a Lie subalgebra.

(iv) (\textit{$P_{\phi}$ induces a Maurer-Cartan element.}) We have
\begin{equation}\nonumber
[P_{\phi}, P_{\phi}] = \left[[1]_j \otimes P, [1]_j \otimes P\right] = [1]_{j-1} \otimes \left[\widetilde{1} \cdot P, \widetilde{1} \cdot P\right] = [1]_{j-1} \otimes [P, P] = 0
\end{equation}
for each $j \geq 1$.
\end{proof}\begin{defn}[Completed de Rham complexes]\label{defldrc}
Let $\phi : V \rightarrow W$ be a smooth map of manifolds. In the context of Example \ref{gtyemb}, \textit{the completed foliation de Rham complex (at the image of $\phi$)}, denoted $\Omega^{\bullet +1}(\mathcal{F})_{\phi}$, is defined via the $L_{\infty}[1]$-algebra structure on $\mathfrak{a}_{\phi}$ introduced in Definition \ref{locvalg}. The $L_{\infty}[1]$-relations are given explicitly as follows.

For each $j \geq 2$ and $h \in C^{\infty}_{\phi}\left(W'_{\phi(x)}\right)^{(j)}$, we have
\begin{equation}\label{lkphi}
\begin{split}
l_k^{\phi}(&h_1 \otimes \xi_1, \cdots, h_k \otimes \xi_k) := \Pi \left[ \cdots \left[P_{\phi},h_1 \otimes \xi_1 \right], \cdots, h_k \otimes \xi_k\right]\\
&= [1]_{j-k} \otimes \Pi\left[ \cdots \left[P, \widetilde{h}_1 \xi_1\right], \cdots, \widetilde{h}_k \xi_k\right] = [1]_{j-k} \otimes l_k\left(\widetilde{h}_1\xi_1, \cdots, \widetilde{h}_k\xi_k\right),
\end{split}
\end{equation}
where $\widetilde{h}_i \in C^{\infty}\left(W'_{\phi(x)}\right)$ denotes any representative satisfying $\widetilde{h}_i + I_{\phi}^j = h_i$, and we set $[1]_{j-k} := 0$ whenever $j \leq k$.

In particular, for $k=1$ we have
\[
l_1^{\phi}(h_1 \otimes \xi_1) = 1 \otimes d_{\mathcal{F}} \left(\widetilde{h}_1 \xi_1\right).
\]

By the preceding lemma, the definition (\ref{lkphi}) is well posed: (i) it does not depend on the choice of representatives $\widetilde{h}_i$, and (ii) it is compatible with the natural projection maps
\[
p_{j+1, j} : C^{\infty}\left(W'_{\phi(x)}\right)^{(j+1)} \rightarrow  C^{\infty}\left(W'_{\phi(x)}\right)^{(j)}, \qquad j \geq 1.
\]
\end{defn}

The next lemma is the morphism-level analogue of Proposition \ref{augomega}. Specializing to the V-algebra of Example \ref{gtyemb} yields an $L_{\infty}[1]$-algebra structure on the foliation de Rham complex.

\begin{lem}\label{auglmo}
Given an $L_{\infty}[1]$-morphism
\begin{equation}\label{htphi}
\widehat{\phi} : \Omega^{\bullet +1}\left(\mathcal{F}'\right) \rightarrow \Omega^{\bullet +1}(\mathcal{F}),
\end{equation}
there exists an $L_{\infty}[1]$-algebra quasi-isomorphism, again denoted by $\widehat{\phi}$,
\[
\widehat{\phi} : \Omega^{\bullet +1}_{\mathrm{aug}}(\mathcal{F}') \rightarrow \Omega^{\bullet +1}_{\mathrm{aug}}(\mathcal{F}),
\]
extending (\ref{htphi}).
\end{lem}

\begin{proof}[Sketch of proof]
The argument follows the same strategy as in the proof of Proposition \ref{augomega}, proceeding by induction on the triple $(m,d,s)$. More precisely, for $g_1, \cdots, g_{d'} \in \Omega^{-2}(W)[1]$ and $\xi_1, \cdots, \xi_{m-d'} \in \Omega^{\geq -1}(W)[1]$, set
\begin{equation}\nonumber
\begin{split}
A(g_1, \cdots, g_{d'}, &\xi_1, \cdots, \xi_{m-d'}) := \sum \pm \widehat{\phi}_{(\cdots)}\left(l_{*}(g_{(\cdots)}, \cdots, \xi_{(\cdots)}), g_{(\cdots)}, \cdots, \xi_{(\cdots)}\right)\\
&+ \sum\limits_{\substack{\sum m_j = m,\\ t \geq 2}} \pm l_{\mathrm{aug}, t}\left(\widehat{\phi}_{(\cdots)}(g_{(\cdots)}, \cdots, \xi_{(\cdots)}), \cdots, \widehat{\phi}_{(\cdots)}(g_{(\cdots)}, \cdots, \xi_{(\cdots)})\right),
\end{split}
\end{equation}
where $l_m$ is assumed to have already been constructed at earlier stages of the induction, together with the corresponding initial conditions. One checks that 
\[
A(g_1, \cdots, g_{d'}, \xi_1, \cdots, \xi_{m-d'})
\]
is $l_1$-closed; indeed, every term on the right-hand side has already been determined at a previous step of the induction. By the Poincar\'{e} lemma, there exists $B := B(g_1, \cdots, g_{d'}, \xi_1, \cdots, \xi_{m-d'})$ such that $A(g_1, \cdots, g_{d'}, \xi_1, \cdots, \xi_{m-d'}) = l_{1}(B).$
We then define $l_{m}(g_1, \cdots, g_{d'}, \xi_1, \cdots, \xi_{m-d'}) := -B$.

Since $\Omega_{\mathrm{aug}}^{\bullet+1}\left(\mathcal{F}'\right)$ and $\Omega_{\mathrm{aug}}^{\bullet+1}\left(\mathcal{F}\right)$ are acyclic $L_{\infty}[1]$-algebras, the resulting $L_{\infty}[1]$-morphism of Lemma \ref{auglmo} is automatically a quasi-isomorphism.
\end{proof}


\begin{thebibliography}{9}

\bibitem[1]{AKSZ}
Mikhail Alexandrov, Maxim Kontsevich, Albert Schwarz, Oleg Zaboronsky, \emph{The geometry of the master equation and topological quantum field theory}, Int. J. Modern Phys. A 12(7):1405–1429, 1997.

\bibitem[2]{AT}
Lino Amorim, Junwu Tu, \emph{The inverse function theorem for curved L-infinity spaces,} J. Noncommut. Geom. 16, no. 4, pp. 1445–1477, 2022.

\bibitem[3]{Bandiera}
Ruggero Bandiera, \emph{Cumulants, Koszul brackets, and homological perturbation theory for commutative $BV_{\infty}$ and $IBL_{\infty}$ algebras},  J. Homotopy Relat. Struct. 17, 227–270, 2022.

\bibitem[4]{BLX}
Kai Behrend, Hsuan-Yi Liao, Ping Xu, \emph{Derived Differentiable Manifolds}, Preprint, arXiv:2006.01376, 2020.

\bibitem[5]{Costello}
Kevin Costello, \emph{A geometric construction of Witten genus, II}, Preprint, arXiv:1112.0816, 2011.

\bibitem[6]{CS}
Alberto S. Cattaneo, Florian Sch\"{a}tz, \emph{Equivalences of higher derived brackets}, Journal of Pure and Applied Algebra, 212, 2450-2460, 2008.

\bibitem[7]{DGMS}
B. A. Dubrovin, M.Giordano, D.Marmo, A. Simoni, \emph{Poisson brackets on presymplectic manifolds}, International journal of modern physics A, Vo, 8, No. 21, 3747-3771, 1993.

\bibitem[8]{DHI}
Daniel Dugger, Sharon Hollander, Daniel C. Isaksen, \emph{Hypercovers and simplicial presheaves
}, Math. Proc. Cambridge Philos. Soc. 136, no. 1, 9–51, 2004.

\bibitem[9]{Eisenbud}
David Eisenbud, \emph{Commutative algebra with a view toward algebraic geometry,} Graduate Texts in Mathematics 150, Springer, 2004.

\bibitem[10]{Fukaya}
Kenji Fukaya, \emph{Deformation theory, homological algebra, and mirror symmetry},
Geometry and Physics of Branes, 121-209, CRC Press, 2002.

\bibitem[11]{FO}
Kenji Fukaya, Kaoru Ono, \emph{Arnold conjecture and Gromov-Witten invariants}, Topology, Volume 38, Issue 5, Pages 933-1048, 1999.

\bibitem[12]{FOOO1}
Kenji Fukaya, Yong-Geun Oh, Hiroshi Ohta, Kaoru Ono, \emph{Kuranishi structures and Virtual fundamental chain}, Springer Monographs in Mathematics, Springer, 2020.

\bibitem[13]{FOOO2}
Kenji Fukaya, Yong-Geun Oh, Hiroshi Ohta, Kaoru Ono, \emph{Lagrangian Intersection Floer Theory : Anomaly and Obstruction Part I, II}, 2009.

\bibitem[14]{FOOO3}
Kenji Fukaya, Yong-Geun Oh, Hiroshi Ohta, Kaoru Ono, \emph{Shrinking good coordinate systems associated to Kuranishi structures}, Journal of Symplectic Geometry, Vol. 14, No. 4 2016.

\bibitem[15]{Gotay}
Mark Gotay, \emph{On coisotropic imbeddings of presymplectic manifolds,} Proceedings of the American Mathematical Society, 84(1):111–114, 1982.

\bibitem[16]{GKZ}
I. M. Gelfand, M. M. Kapranov, A. V. Zelevinsky, \emph{Discriminants, Resultants and Multidimensional Determinants,} Birkh{a}user, 1994.

\bibitem[17]{GLRR}
Xavier Gràcia, Javier de Lucas, Xavier Rivas, Narciso Román-Roy, \emph{On Darboux theorems for geometric structures induced by closed forms}, Rev. Real Acad. Cienc. Exactas Fis. Nat. Ser. A-Mat. 118, 131, 2024.

\bibitem[18]{Joyce1}
Dominic Joyce, \emph{Kuranishi homology and Kuranishi cohomology,} Preprint, arXiv:0710.5634, 2007.

\bibitem[19]{Joyce2}
Dominic Joyce, \emph{Kuranishi spaces as a 2-category}, Virtual Fundamental Cycles in Symplectic Topology, Mathematical Surveys and Monographs 237, Americal Mathematical Society, 253-298, 2019.

\bibitem[20]{Kim1}
Taesu Kim, \emph{$L_\infty$-Kuranishi spaces and the moduli space of pseudoholomorphic maps,} Preprint, arXiv:2511.05206, 2025.

\bibitem[21]{Kim2}
Taesu Kim, \emph{Homotopy models for $L_{\infty}[1]$-algebras in higher degrees,} Preprint, arXiv:2606.28985, 2026.

\bibitem[22]{Kim3}
Taesu Kim, \emph{Kuranishi chart categories and higher cocycle conditions,} Preprint, arXiv:2607.01732, 2026.

\bibitem[23]{Kim4}
Taesu Kim, \emph{Kuranishi spaces in a category of fibrant objects and their homotopy fiber products,} Preprint, arXiv:2610.11594, 2026.

\bibitem[24]{KO}
Taesu Kim, Yong-Geun Oh, \emph{Stratifications associated to generic closed two-forms and stratified $L_{\infty}$ spaces}, Preprint, arXiv:2602.24099, 2026.

\bibitem[25]{Lurie}
Jacob Lurie, \emph{Higher Topos Theory}, Annals of Mathematics Studies 170, Princeton University Press, 2009.

\bibitem[26]{Markl}
Martin Markl, \emph{On the origin of higher braces and higher-order derivations}, Journal Homotopy and Related Structures, 10, 637–667, 2015.

\bibitem[27]{MS}
Eva Miranda, Romero Solha, \emph{On a Poincar\'{e} lemma for foliations}, Foliations 2012, 115-137, World Scientific, 2013.

\bibitem[28]{MW1} Dusa McDuff, Katrin Wehrheim, \emph{The topology of Kuranishi atlases}, Proc. London Math. Soc., 115: 221-292, 2017.

\bibitem[29]{MW2} D. McDuff and K. Wehrheim, Smooth Kuranishi atlases with isotropy, Geom. Topol. 21, no. 5, 2725–2809, 2017.

\bibitem[30]{OP}
Yong-Geun Oh, Jae-Suk Park, \emph{Deformations of coisotropic submanifolds and strong homotopy Lie algebroids}, Inventiones mathematicae, Volume 161, 287–360 2005.

\bibitem[31]{Pardon}
John Pardon, \emph{An algebraic approach to virtual fundamental cycles on moduli spaces of J-holomorphic curves}, Geom. Topol. 20, 779-1034, 2016.

\bibitem[32]{Roytenberg} 
Dmitry Roytenberg, \emph{Courant algebroids, derived brackets and even symplectic supermanifolds}, Ph.D. Thesis, UC Berkeley, 1999, math.DG/9910078.

\bibitem[33]{Tu1}
Junwu Tu, \emph{Homotopy L-infinity Spaces}, Preprint, arXiv:1411.5115, 2014.

\bibitem[34]{Tu2}
Junwu Tu, \emph{Homotopy L-infinity spaces and Kuranishi manifolds, I: categorical structures}, Preprint, arXiv:1602.00150, 2016.

\bibitem[35]{Voronov1}
Theodore Voronov, \emph{Higher derived brackets and homotopy algebras}, Journal of Pure and Applied Algebra, Volume 202, Issues 1–3, 1 November, 133-153, 2005.

\bibitem[36]{Voronov2}
Theodore Voronov, \emph{Higher derived brackets for arbitrary derivations}, Travaux Math. XVI 163–186, 2005.

\end{thebibliography}
\end{document}